\documentclass[reqno]{amsart} 
\usepackage{amsmath,amsthm,amssymb,amsfonts}
\usepackage{mathrsfs}
\usepackage{longtable}
\numberwithin{equation}{section}
\usepackage{color}
\usepackage{enumerate}
\usepackage{multirow} 
\newtheoremstyle{mystyle}
{}
{}
{\normalfont}
{ }
{\bfseries}
{}
{10pt}
{ }
\theoremstyle{mystyle}
\newtheorem{theorem}{Theorem}
\newtheorem{proposition}{Proposition}
\newtheorem{lemma}{Lemma}
\newtheorem{remark}{Remark}
\newtheorem{corollary}{Corollary}

\DeclareMathOperator*{\argmax}{arg\,max}
\usepackage{mathtools}
\usepackage[pdftex]{graphicx}
\usepackage{color} 
\def\diag{\mathop{\rm diag}\nolimits}

\def\tr{\mathop{\rm tr}\nolimits}

\DeclareMathOperator*{\argmin}{arg\,min}
\DeclareMathOperator*{\argsup}{arg\,sup}
\allowdisplaybreaks[4]
\usepackage[foot]{amsaddr}
\usepackage[top=2.5cm, bottom=2.5cm, left=2.5cm, right=2.5cm]{geometry}
\usepackage[hang,small,bf]{caption}
\usepackage[subrefformat=parens]{subcaption}
\usepackage{bm}
\usepackage{comment}
\usepackage{enumitem}
\large
\title[clustering for diffusion processes with constant volatility coefficients]{Model-based clustering of diffusion processes with constant volatility coefficients using high-frequency data}
\author[S. Kusano]{Shogo Kusano $^{1}$}
\author[M. Uchida]{Masayuki Uchida $^{2,3,4}$}
\address{$^{1}$ Faculty of Advanced Science and Technology, Kumamoto University, Kumamoto, Japan}
\address{$^{2}$ Graduate School of Engineering Science, The University of Osaka, Toyonaka, Japan}
\address{$^{3}$ Center for Mathematical Modeling and Data Science (MMDS), The University of Osaka, Toyonaka, Japan}
\address{$^{4}$ Graduate School of Mathematical Sciences, The University of Tokyo, Meguro, Japan and CREST, Japan Science and Technology Agency, Tokyo, Japan}
\begin{document}
\begin{abstract} 
\fontsize{9pt}{13pt}\selectfont
This paper considers model-based clustering of diffusion processes using high-frequency data. We formulate a finite mixture model for diffusion processes and develop a clustering method based on their constant diffusion coefficients. Using a quasi-likelihood, we propose estimators of the number of clusters, the diffusion coefficients and the cluster labels, and establish their consistency and the asymptotic normality of the diffusion coefficient estimators under high-frequency asymptotics with the number of observed processes fixed. Simulation studies are conducted to investigate the finite-sample performance of the proposed method.
\end{abstract}
\maketitle

\section{Introduction} 
\fontsize{10.5pt}{16pt}\selectfont
\fontsize{11pt}{18pt}\selectfont
\raggedbottom
\setlength{\abovedisplayskip}{10pt}
\setlength{\belowdisplayskip}{10pt}
Time-series clustering is an unsupervised learning method for grouping time series into clusters and has been applied across a wide range of domains. Time-series clustering methods can be classified into three main categories: raw-data-based, feature-based, and model-based approaches; see, e.g., Liao \cite{Liao(2005)} and Aghabozorgi et al. \cite{Aghabozorgi(2015)}. In raw-data-based approaches, time series are clustered directly by using their observed values.
This is often done by applying hierarchical clustering or $k$-medoids using a suitable time-series distance measure. In feature-based approaches, time series are clustered by using extracted features, such as autocorrelation coefficients, spectral entropy and Fourier or wavelet coefficients. In model-based approaches, time series are represented by statistical models, and clustering is often performed by using likelihoods or estimated model parameters; see, e.g., Piccolo \cite{Piccolo(1990)} and Kalpakis et al. \cite{Kalpakis(2001)} for 
autoregressive integrated moving-average (ARIMA) models, Ramoni et al. \cite{Ramoni(2000)} for Markov chain models, and Li and Biswas \cite{Li(1999)} and Panuccio et al. \cite{Panuccio(2002)} for hidden Markov models. Although model-based approaches require statistical models to be specified in advance, they can reveal cluster structures that raw-data-based or feature-based approaches may fail to capture.

In model-based time-series clustering, mixtures of autoregressive moving-average (ARMA) models have been extensively studied. Xiong and Yeung \cite{Xiong(2002),Xiong(2004)} proposed a 
model-based clustering method using a finite mixture of ARMA models and 
employed the EM algorithm to estimate the model parameters. Although determining the number of clusters is important in practice, it is often prespecified or selected using heuristic validation criteria. In contrast, Xiong and Yeung selected the number of clusters using the Bayesian information criterion (BIC). This demonstrates one of the major advantages of likelihood-based mixture modeling. Recently, Ren and Barnett \cite{Ren(2022)} proposed an autoregressive mixture clustering method using a Wishart mixture likelihood and established the consistency of the estimators of the autoregressive parameters and cluster labels. Notably, they established these asymptotic results when the number of time series is fixed and the length of each time series tends to infinity. This implies that their theory can be applied even when only a small number of time series are available due to technical or budgetary constraints.

In recent decades, it has become easier to obtain high-frequency data, such as stock price data and lifelog data, due to technological advances. Such data can be regarded as discrete observations from continuous-time stochastic processes such as diffusion processes. This has motivated extensive research on statistical inference for continuous-time stochastic processes based on high-frequency data; see, e.g., Yoshida \cite{Yoshida(1992)}, Genon-Catalot and Jacod \cite{Genon(1993)}, Kessler \cite{kessler(1997)}, Uchida and Yoshida \cite{Uchi-Yoshi(2012)} and references therein. In contrast, only a few studies have investigated clustering of continuous-time stochastic processes. De Gregorio and Iacus \cite{De Gregorio(2010)} proposed a
distance-based clustering method for discretely observed diffusion
processes based on the quadratic distance between the corresponding
estimated Markov operators. They applied hierarchical clustering based on the proposed distance to both simulated and real financial data.
Their theoretical framework assumes a fixed sampling interval and
therefore does not directly cover the high-frequency asymptotic
regime considered in this paper. Delattre et al. \cite{Delattre(2016)} developed a model-based approach in which random effects in the drift of stochastic differential equations follow a Gaussian mixture distribution. Under continuous observation and assuming that the number of mixture components is known, they established the consistency of the  maximum likelihood estimators of the model parameters as the number of observed processes tends to infinity. In the context of functional data analysis, Cuesta-Albertos and Dutta \cite{Cuesta-Albertos(2023)} studied a model-based clustering method for Gaussian processes. When the true number of clusters is two, they established the consistency of the estimators of the number of clusters and the cluster labels as the number of observed processes tends to infinity. For studies on clustering within a single observed path, see, e.g., Boyko et al. \cite{Boyko(2022)}.

As mentioned above, few studies have considered model-based clustering of diffusion processes using high-frequency data. To the best of our knowledge, no asymptotically justified method has been developed for estimating the number of clusters, the diffusion coefficients and the cluster labels when the number of observed processes is fixed. To address this gap, we formulate a finite mixture model for diffusion processes and develop a clustering method based on their constant diffusion coefficients. Using a quasi-likelihood, we construct estimators of the number of clusters, the diffusion coefficients and the cluster labels, and establish their consistency under high-frequency asymptotics with the number of observed processes fixed. We also establish the asymptotic normality of the diffusion coefficient estimators. The proposed method may identify groups characterized by machine operating conditions or fault severity in vibration data, risk levels in financial data, and sleep quality in lifelog data. 

The remainder of this paper is organized as follows. In Section \ref{notation}, we introduce the model, notation and assumptions. In Section \ref{main}, we propose a quasi-likelihood-based clustering method and establish its asymptotic properties. In Section \ref{simulation}, we conduct simulation studies to investigate the finite-sample performance of the proposed method. Section \ref{proof} contains the proofs of the theoretical results presented in Section \ref{main}.
\section{Model, notation and assumptions}\label{notation}
We first introduce a finite mixture model for diffusion processes. Let $(\Omega,\mathcal{F},(\mathcal{F}_t)_{t\geq 0},\mathbb{P})$ be a filtered probability space satisfying the usual conditions. We fix an upper bound $M\in\mathbb{N}$ on the number of clusters, which means that we consider models with $1,\ldots,M$ clusters. For a given number of clusters $G\in\{1,\ldots,M\}$, we define the model as follows.  Let $\{X^{(G)}_t\}_{t\geq 0}$ be an observed process belonging to one of the $G$ clusters. We introduce an $\mathcal F_0$-measurable latent random variable $Z^{(G)}$ to represent the cluster membership of the observed process and assume that
\begin{align*}
    \mathbb{P}\bigl(Z^{(G)}=g\bigr)=\pi^{(G)}_g>0
\end{align*}
for $g=1,\ldots,G$, and 
\begin{align*}
    \sum_{g=1}^G\pi^{(G)}_g=1,
\end{align*}
where $\pi^{(G)}_g$ is an unknown mixing proportion.
In addition, we suppose that for each $g=1,\ldots,G$, the conditional law of $\{X^{(G)}_t\}_{t\geq 0}$ given $Z^{(G)}=g$ is specified as the law of the diffusion process satisfying
\begin{align}
    dY^{(G)}_{g,t}=
    a^{(G)}_g(Y^{(G)}_{g,t})dt+\sigma^{(G)}_g dW_t,\quad Y^{(G)}_{g,0}=\xi^{(G)}_{g},\label{X}
\end{align}
where $a^{(G)}_g:\mathbb{R}\to\mathbb{R}$ is an unknown Borel measurable function, $\sigma^{(G)}_g>0$ is an unknown constant diffusion coefficient, $\xi^{(G)}_{g}$ is an $\mathcal{F}_0$-measurable real-valued random variable, and $\{W_t\}_{t\geq 0}$ is a one-dimensional standard $(\mathcal F_t)_{t\geq 0}$-Wiener process. To avoid the label-switching problem, we assume that
\begin{align*}
    \frac{1}{d}<\sigma^{(G)}_1<\sigma^{(G)}_2<\cdots<\sigma^{(G)}_G<d
\end{align*}
for some constant $d>1$. Set
\begin{align*}
    {\bm{\pi}}^{(G)}&=
    \bigl(\pi^{(G)}_1,\pi^{(G)}_2,\ldots,\pi^{(G)}_G\bigr)^{\top}
    \in\Pi^{(G)}=
    \Biggl\{{\bm{\pi}}^{(G)}\in (0,1]^G\ \bigg|\ 
    \sum_{g=1}^G\pi^{(G)}_g=1\Biggr\}
\end{align*}
and
\begin{align*}
    {\bm{\sigma}}^{(G)}&=
    \bigl(\sigma^{(G)}_1,\sigma^{(G)}_2,\ldots,\sigma^{(G)}_G
    \bigr)^{\top}
    \in\Sigma^{(G)}=
    \biggl\{{\bm{\sigma}}^{(G)}\in(1/d,d)^G\ \Big|\ 
    \sigma^{(G)}_1<\sigma^{(G)}_2<\cdots<\sigma^{(G)}_G\biggr\},
\end{align*}
where $\top$ denotes transposition. $P_{{\bm{\pi}}^{(G)}}$ and $P_{\sigma^{(G)}_g}$ denote the distribution of $Z^{(G)}$ and the law of the diffusion process in \eqref{X}, respectively. For a fixed $J\in\mathbb{N}$, we consider $J$ independent copies 
\begin{align*}
    \Bigl(Z^{(G)}_j,\{X^{(G)}_{j,t}\}_{t\geq 0}\Bigr)\quad (j=1,\ldots,J)
\end{align*}
of 
\begin{align*} 
    \Bigl(Z^{(G)},\{X_t^{(G)}\}_{t\geq0}\Bigr),
\end{align*}
which means that we observe $J$ independent time series whose cluster labels are unobserved. More specifically, for each $j=1,\ldots,J$, we suppose that the $\mathcal{F}_0$-measurable latent random variable $Z^{(G)}_j$ satisfies
\begin{align*}
    Z^{(G)}_j\sim P_{\bm{\pi}^{(G)}}
\end{align*}
and for each $g=1,\ldots,G$, given $Z_j^{(G)}=g$, $\{X^{(G)}_{j,t}\}_{t\geq 0}$ satisfies
\begin{align*}
    dX^{(G)}_{j,t}=
    a^{(G)}_g(X^{(G)}_{j,t})dt+\sigma^{(G)}_g dW_{j,t},\quad X^{(G)}_{j,0}=\xi^{(G)}_{j,g},
\end{align*}
where $\xi^{(G)}_{j,g}$ is an $\mathcal F_0$-measurable real-valued random variable, and $\{W_{j,t}\}_{t\geq 0}$ is a one-dimensional standard $(\mathcal F_t)_{t\geq 0}$-Wiener process. In what follows, we fix $G^*\in\{1,\ldots,M\}$ as the true number of clusters and fix ${\bm{\pi}}^{*}\in\Pi^{(G^*)}$ and ${\bm{\sigma}}^{*}\in\Sigma^{(G^*)}$ as the true values of ${\bm{\pi}}^{(G^*)}$ and ${\bm{\sigma}}^{(G^*)}$, respectively. For each $j=1,\ldots,J$, we fix $z_j\in\{1,\ldots,G^*\}$ as the unobserved cluster label of the $j$-th observed process, and define the conditional probability measure by
\begin{align*}
    {\rm{P}}_{\bm{z}}(A)
    =\mathbb{P}\Bigl(A\, \big|\, Z^{(G^*)}_1=z_1,Z^{(G^*)}_2=z_2,\ldots,Z^{(G^*)}_J=z_J\Bigr)
\end{align*}
for $A\in\mathcal{F}$, where ${\bm{z}}=(z_1,z_2,\ldots,z_J)^{\top}$. ${\rm{E}}_{\bm{z}}$ denotes the expectation under ${\rm{P}}_{\bm{z}}$. For each $j=1,\ldots,J$, we note that
\begin{align*}
    \bigl\{X^{(G^*)}_{j,t}\bigr\}_{t\geq 0}\sim P_{\sigma^*_{z_j}}
\end{align*}
under ${\rm{P}}_{\bm{z}}$. For notational simplicity, we write $\{X^{(G^*)}_{j,t}\}_{t\geq 0}$ as $\{X_{j,t}\}_{t\geq 0}$ for $j=1,\ldots,J$. For each $j=1,\ldots,J$, we assume that only the discrete observations
\begin{align*}
    \mathbb{X}_{j,n}=\bigl\{X_{j,t_i^{n}}\bigr\}_{i=0}^{n}
\end{align*}
are available, where $t_i^{n}=ih_{n}$, $h_{n}=T/n$, and $T>0$ is fixed. Furthermore, we suppose that $G^*\leq J$ and at least one observed process belongs to each cluster, i.e.,
\begin{align*}
    S_g=\Bigl\{j\in\{1,\ldots,J\}\ \big|\ z_j=g\Bigr\}\neq\emptyset
\end{align*}
for each $g=1,\ldots,G^*$. The aim of the clustering problem is to estimate the true number of clusters $G^*$, the true diffusion parameter vector ${\bm{\sigma}}^*$ and the vector of true cluster labels $\bm{z}$ based on the discrete observations under ${\rm P}_{\bm{z}}$. In this paper, we treat ${\bm{\pi}}^*$ as a nuisance parameter and do not estimate it. This point is discussed further in Section \ref{main}.
\begin{remark}
To simplify the theoretical analysis, we assume that the processes
within each cluster have a common drift function and that all processes are observed synchronously. Moreover, we suppose that the diffusion coefficients are constant. Extending the proposed method to state-dependent diffusion coefficients is nontrivial and is left for future work.
\end{remark}
\begin{remark}
For ergodic diffusion processes observed over a long time horizon, clustering based on both drift and diffusion parameters may also be possible under suitable regularity conditions. We leave this extension for future work.
\end{remark}
Next, we define the notation used throughout the paper. For a vector $v$, $|v|$ is defined by
\begin{align*}
    |v|=\sqrt{\tr(vv^\top)}.
\end{align*}
For a matrix $M$, $M_{ij}$ denotes its $(i,j)$-th entry, and we define
\begin{align*}
    |M|=\sqrt{\tr(MM^\top)}.
\end{align*}
If a matrix $M$ is positive definite, we write $M>0$. For $a_1,\ldots,a_k\in\mathbb{R}$, $\diag[a_1,\ldots,a_k]$
denotes the $k\times k$ diagonal matrix whose $i$-th diagonal entry is $a_i$. $\mathbb{I}_{k}$ denotes the identity matrix of size $k$. The symbols $\stackrel{p}{\longrightarrow}$ and 
$\stackrel{d}{\longrightarrow}$ denote convergence in probability and convergence in distribution, respectively. The $d$-dimensional normal distribution with mean 
$\mu\in\mathbb{R}^d$ and covariance matrix $\Sigma$ is denoted by $N_d(\mu,\Sigma)$. Set 
\begin{align*}
    \partial_{{\bm{\sigma}}}=\frac{\partial}{\partial{\bm{\sigma}}},\quad \partial^2_{{\bm{\sigma}}}=\frac{\partial^2}{\partial{\bm{\sigma}}\partial{\bm{\sigma}}^{\top}}.
\end{align*}
Let ${\bm{0}}_N$ denote the $N$-dimensional zero vector and $O_{M,N}$ the $M\times N$ zero matrix. For ${\bm{a}}\in\mathbb{R}^M$, when $N=0$,  we interpret the $(M+N)$-dimensional vector 
\begin{align*}
    \begin{pmatrix}
        {\bm{a}}\\
        {\bm{0}}_N
    \end{pmatrix}
\end{align*}
as ${\bm{a}}$. Likewise, for $A\in\mathbb{R}^{M\times M}$, when $N$=0, we interpret the $(M+N)\times (M+N)$-dimensional block matrix
\begin{align*}
    \begin{pmatrix}
        A & O_{M,N}\\
        O_{N,M} & O_{N,N}
    \end{pmatrix}
\end{align*}
as $A$. For a finite set $S$, $|S|$ denotes the number of elements in $S$. Finally, we state the assumptions. 
\begin{enumerate}
\vspace{3mm}
\item[{[\bf{A1}]}] For each $g=1,\ldots,G^*$, there exists a constant $C_g>0$ such that
\begin{align*}
    \bigl|a^{(G^*)}_g(x)-a^{(G^*)}_g(y)\bigr|\leq C_g|x-y|
\end{align*}
for all $x,y\in\mathbb{R}$. \vspace{3mm}
\item[{[\bf A2]}]\ For each $j=1,\ldots,J$,
\begin{align*} 
    {\rm{E}}_{{\bm{z}}}\Bigl[\bigl|\xi^{(G^*)}_{j,z_j}\bigr|^L\Bigr]<\infty
\end{align*}
for every $L>0$.
\end{enumerate}
\begin{remark}
For each $j=1,\ldots,J$, since 
\begin{align*}
    X_{j,0}=\xi^{(G^*)}_{j,z_j}
\end{align*}
under ${\rm{P}}_{{\bm{z}}}$, {\bf{[A1]}} and {\bf{[A2]}} imply
\begin{align*}
    {\rm{E}}_{{\bm{z}}}\Bigl[\bigl|X_{j,t}\bigr|^{2L}\Bigr]
    \leq C_{L}e^{C_{L}t}
\end{align*}
for every $L\geq 1$, where $C_L>0$ is some constant that depends on $L$; see, e.g., Theorem 1.6 in Kutoyants \cite{Kutoyants(2004)}. Consequently, since $T$ is fixed, for each $j=1,\ldots,J$, we obtain
\begin{align*}
     \sup_{t\in[0,T]}{\rm{E}}_{{\bm{z}}}\Bigl[\bigl|X_{j,t}\bigr|^{2L}\Bigr]
     <\infty
\end{align*}
for every $L>0$.
\end{remark}
\section{Main results}\label{main}
Using the Euler--Maruyama approximation, for each $G=1,\ldots,M$ and a given ${\bm{\pi}}^{(G)}\in\Pi^{(G)}$, we define a quasi-likelihood for ${\bm{\sigma}}^{(G)}\in\overline{\Sigma}^{(G)}$ by
\begin{align}
    {\bf{L}}^{(G)}_{n}({\bm{\sigma}}^{(G)};{\bm{\pi}}^{(G)})=\prod_{j=1}^J\sum_{g=1}^{G}\pi^{(G)}_{g}\bar{\bf{L}}_{j,n}(\sigma^{(G)}_{g}),\label{quasi}
\end{align}
where for $j=1,\ldots,J$,
\begin{align*}
    \bar{\bf L}_{j,n}(\sigma)
    =\prod_{i=1}^n
    \frac{1}{\sqrt{2\pi h_{n}\sigma^2}}
    \exp\Biggl\{
    -\frac{\bigl(X_{j,t_i^n}-X_{j,t_{i-1}^n}\bigr)^2}
    {2h_{n}\sigma^2}\Biggr\}
\end{align*}
for $\sigma>0$, and the closure of the parameter space $\Sigma^{(G)}$ is given by
\begin{align*}
    \overline{\Sigma}^{(G)}=
    \Bigl\{
    {\bm{\sigma}}^{(G)}\in[1/d,d]^G
    \ \Big|\
    \sigma^{(G)}_{1}\leq\sigma^{(G)}_{2}\leq\cdots\leq\sigma^{(G)}_{G}
    \Bigr\}
\end{align*}
for $G=1,\ldots,M$. Details of the quasi-likelihood are given in Appendix \ref{likelihood}. Let
\begin{align*}
    {\bm{\pi}}_M=\begin{bmatrix}
        {\bm{\pi}}^{(1)}\\
        {\bm{\pi}}^{(2)}\\
        \vdots\\
        {\bm{\pi}}^{(M)}
    \end{bmatrix}\in\Pi_M=
    \Pi^{(1)}\times\Pi^{(2)}\times\cdots
    \times\Pi^{(M)}
\end{align*}
and
\begin{align*}
    {\bf{H}}^{(G)}_{n}({\bm{\sigma}}^{(G)};{\bm{\pi}}^{(G)})=\log\Bigl\{(2\pi h_n)^{nJ/2}{\bf{L}}^{(G)}_{n}({\bm{\sigma}}^{(G)};{\bm{\pi}}^{(G)})\Bigr\}.
\end{align*}
Note that it can be rewritten as follows:
\begin{align*}
   {\bf{H}}^{(G)}_{n}({\bm{\sigma}}^{(G)};{\bm{\pi}}^{(G)})&=\log\Biggl(
    \prod_{j=1}^J
    \sum_{g=1}^{G}\pi^{(G)}_g(2\pi h_n)^{n/2}
    \bar{\bf L}_{j,n}(\sigma^{(G)}_g)
    \Biggr)\\
    &=\sum_{j=1}^J\log\Biggl(
    \sum_{g=1}^{G}
    \pi^{(G)}_g\exp\Bigl\{\bar{\bf H}_{j,n}(\sigma^{(G)}_g)\Bigr\}
    \Biggr),
\end{align*}
where for $j=1,\ldots,J$,
\begin{align*}
    \bar{\bf H}_{j,n}(\sigma)
    &=-\frac{n}{2}\log\sigma^2-\sum_{i=1}^n\frac{\bigl(X_{j,t_i^n}-X_{j,t_{i-1}^n}\bigr)^2}{2h_n\sigma^2}
\end{align*}
for $\sigma>0$. In addition, we define
\begin{align*}
    \mathbb{H}_{\bm{z}}^{(G)}({\bm{\sigma}}^{(G)})=\sum_{j=1}^J\max_{g\in\{1,\ldots,G\}}\bar{\mathbb H}_{z_j}(\sigma^{(G)}_g)
\end{align*}
for ${\bm{\sigma}}^{(G)}\in\overline{\Sigma}^{(G)}$, where for each $k=1,\ldots,G^*$,
\begin{align*}
    \bar{\mathbb{H}}_{k}(\sigma)=-\frac{1}{2}\log\sigma^2
    -\frac{\sigma_{k}^{*2}}{2\sigma^2}
\end{align*}
for $\sigma>0$. Note that for each $G=1,\ldots,M$, $n^{-1}{\bf{H}}^{(G)}_{n}$ converges to $\mathbb{H}_{\bm{z}}^{(G)}$ uniformly in ${\bm{\sigma}}^{(G)}\in\overline{\Sigma}^{(G)}$ in probability as $n\longrightarrow\infty$; see Proposition \ref{unifprop}.

First of all, we consider the estimation of the number of clusters. An estimator of the number of clusters is defined by
\begin{align}
    \hat{G}_n({\bm{\pi}}_M)\in\argmin_{G\in\{1,\ldots,M\}}{\bf{F}}_n(G;{\bm{\pi}}^{(G)}),\label{criterion}
\end{align}
where
\begin{align*}
    {\bf{F}}_n(G;{\bm{\pi}}^{(G)})=-2\sup_{{\bm{\sigma}}^{(G)}\in\overline{\Sigma}^{(G)}}{\bf{H}}^{(G)}_{n}({\bm{\sigma}}^{(G)};{\bm{\pi}}^{(G)})+P^{(G)}_{n},
\end{align*}
and $P^{(G)}_{n}$ is a deterministic penalty term depending on $n$ and $G$. The following condition is imposed on the penalty terms.
\begin{enumerate}
    \vspace{3mm}
    \item[{\bf{[B]}}] 
    \begin{enumerate}
        \item For each $G=1,\ldots,M$, $P^{(G)}_{n}=o(n)$ as $n\longrightarrow\infty$.
        \vspace{4mm}
        \item For any $K,L\in\{1,\ldots,M\}$ such that $K<L$, 
        \begin{align*}
            P^{(K)}_{n}<P^{(L)}_{n}
        \end{align*}
        for all sufficiently large $n$ and 
        \begin{align*}
            \frac{1}{P^{(L)}_{n}-P^{(K)}_{n}}=o(1)
        \end{align*}
        as $n\longrightarrow\infty$.
    \end{enumerate}
    \vspace{3mm}
\end{enumerate}
\begin{remark}
The BIC-type penalty 
\begin{align*}
    P_n^{(G)}=G\log n
\end{align*}
satisfies {\bf{[B]}}. Indeed, for each $G=1,\ldots,M$, we have
\begin{align*}
    \frac{P_n^{(G)}}{n}=G\times\frac{\log n}{n}\longrightarrow 0
\end{align*}
as $n\longrightarrow\infty$. Thus, {\bf{[B]}} (a) holds. Moreover, for any $K,L\in\{1,\ldots,M\}$ such that $K<L$,
\begin{align*}
    P^{(L)}_{n}-P^{(K)}_{n}=(L-K)\log n>0
\end{align*}
for $n\geq 2$ and
\begin{align*}
    \frac{1}{P^{(L)}_{n}-P^{(K)}_{n}}=\frac{1}{(L-K)\log n}=o(1)
\end{align*}
as $n\longrightarrow\infty$. Hence, ${\bf{[B]}}$ (b) also holds.
\end{remark}
\begin{remark}
When $P_n^{(G)}=G\log n$, the selection criterion in (\ref{criterion}) has a form similar to the quasi-BIC (QBIC); see, e.g., Eguchi and Masuda \cite{Eguchi(2018)}. Strictly speaking, however, this criterion should not be regarded as a QBIC. The reason is that the model is non-regular when $G>G^*$, and hence the usual mathematical justification for QBIC is not directly available. Indeed, when $G>G^*$, the asymptotic covariance matrix of the estimator may be degenerate, so that standard arguments are not directly applicable. In particular, it is difficult to obtain a justification of the type given in Theorem 3.7 of Eguchi and Masuda \cite{Eguchi(2018)}. This issue is not specific to our model, but is common to mixture models in general. We refer to Nguyen \cite{Nguyen(2024)} and Nguyen and Nguyen \cite{Nguyen(2026)} for further details on the problem.
\end{remark}
We obtain the following consistency result for the estimator of the number of clusters.
\begin{theorem}\label{Gconsis}
Suppose that ${[\bf{A1}]}$, ${[\bf{A2}]}$ and ${[\bf{B}]}$ hold. Then, for every ${\bm{\pi}}_M\in\Pi_M$,
\begin{align*}
    {\rm{P}}_{\bm{z}}\biggl(\hat{G}_n({\bm{\pi}}_M)=G^*\biggr)
    \longrightarrow 1
\end{align*}
as $n\longrightarrow\infty$.
\end{theorem}
\noindent Next, we consider the estimation of the diffusion parameter.
For simplicity, hereafter we may write $\hat{G}_n({\bm{\pi}}_M)$ as $\hat{G}_n$. 
We consider the following estimator obtained by plugging $\hat{G}_n$ into $G$ in the quasi-likelihood:
\begin{align*}
    {\hat{\bm{\sigma}}}_{n}({\bm{\pi}}_M)
    \in\argsup_{{\bm{\sigma}}^{(\hat{G}_n)}\in\overline{\Sigma}^{(\hat{G}_n)}}{\bf{L}}^{(\hat{G}_n)}_{n}({\bm{\sigma}}^{(\hat{G}_n)};{\bm{\pi}}^{(\hat{G}_n)}).
\end{align*}
Set 
\begin{align*}
    I({\bm{\sigma}}^{(G^*)})=-\diag\biggl[|S_1|\bar{\mathbb{H}}''_{1}(\sigma^{(G^*)}_1),\ldots,|S_{G^*}|\bar{\mathbb{H}}''_{G^*}(\sigma^{(G^*)}_{G^*})\biggr]
\end{align*}
for ${\bm{\sigma}}^{(G^*)}\in\Sigma^{(G^*)}$. We obtain the following asymptotic result for the plug-in estimator.
\begin{theorem}\label{consis}
Suppose that ${[\bf{A1}]}$, ${[\bf{A2}]}$ and ${[\bf{B}]}$ hold. Then, for every ${\bm{\pi}}_M\in\Pi_M$,
\begin{align*}
    \begin{bmatrix}
    \hat{\bm{\sigma}}_n({\bm{\pi}}_M)\\
    {\bm{0}}_{M-\hat{G}_n({\bm{\pi}}_M)}
    \end{bmatrix}\xlongrightarrow{p}\begin{bmatrix}
    {\bm{\sigma}}^*\\
    {\bm{0}}_{M-G^*}
    \end{bmatrix}
\end{align*}
and
\begin{align*}
    &\sqrt{n}\Biggl(\,
    \begin{bmatrix}
    \hat{\bm{\sigma}}_n({\bm{\pi}}_M)\\
    {\bm{0}}_{M-\hat{G}_n({\bm{\pi}}_M)}
    \end{bmatrix}
    -\begin{bmatrix}
    {\bm{\sigma}}^*\\
    {\bm{0}}_{M-G^*}
    \end{bmatrix}\, \Biggr)\xlongrightarrow{d}N_{M}\left({\bm{0}}_M,\, \begin{bmatrix}
        I({\bm{\sigma}}^*)^{-1} & O_{G^*,M-G^*}\\
        O_{M-G^*,G^*} & O_{M-G^*,M-G^*}
    \end{bmatrix}\right)
\end{align*}
under ${\rm{P}}_{\bm{z}}$ as $n\longrightarrow\infty$.
\end{theorem}
\begin{remark}
While $\mathbb{H}_{\bm{z}}^{(G^*)}$ may not be differentiable everywhere on $\overline{\Sigma}^{(G^*)}$, it is twice continuously differentiable in a neighborhood of ${\bm{\sigma}}^*$. Noting that 
\begin{align*}
    \bar{\mathbb{H}}''_{g}(\sigma)=\frac{1}{\sigma^2}-\frac{3\sigma_{g}^{*2}}{\sigma^4}
\end{align*}
for $\sigma>0$, it is shown that
\begin{align*}
    -\partial^2_{\bm{\sigma}^{(G^*)}}\mathbb{H}_{\bm{z}}^{(G^*)}({\bm{\sigma}}^*)=
    I({\bm{\sigma}}^*)=2\diag\Biggl[\frac{|S_1|}{\sigma_1^{*2}},\ldots,\frac{|S_{G^*}|}{\sigma_{G^*}^{*2}}\Biggr].
\end{align*}
Since $S_g\neq\emptyset$ for each $g=1,\ldots,G^*$, $I({\bm{\sigma}}^*)$ is positive definite.
\end{remark}
\noindent Finally, we consider the estimation of the cluster labels. For each $j=1,\ldots,J$, we define
\begin{align*}
    \mathbb{Q}_{z_j}(g)=\left\{
    \begin{array}{ll}
        1 & (g=z_j),\\
        0 & (g\neq z_j)
    \end{array}
    \right.
\end{align*}
for $g=1,\ldots,G^*$. Furthermore, for each $j=1,\ldots,J$ and $G=1,\ldots,M$ and for every ${\bm{\sigma}}^{(G)}\in\overline{\Sigma}^{(G)}$ and ${\bm{\pi}}^{(G)}\in\Pi^{(G)}$, we set
\begin{align*}
    {\bf{Q}}^{(G)}_{j,n}(g;{\bm{\sigma}}^{(G)},{\bm{\pi}}^{(G)})=\frac{\pi^{(G)}_{g}\exp\Bigl\{\bar{\bf{H}}_{j,n}(\sigma^{(G)}_g)\Bigr\}
    }{\displaystyle\sum_{k=1}^{G}\pi^{(G)}_{k}\exp\Bigl\{\bar{\bf{H}}_{j,n}(\sigma^{(G)}_k)\Bigr\}}
\end{align*}
for $g=1,\ldots,G$. For each $j=1,\ldots,J$, since $\mathbb{Q}_{z_j}(g)$ is uniquely maximized at $g=z_j$, we define an estimator of the cluster label $z_j$ by
\begin{align*}
    \hat{Z}_{j,n}({\bm{\pi}}_M)
    \in\argmax_{g\in\{1,\ldots,\hat{G}_n\}}{\bf{Q}}^{(\hat{G}_n)}_{j,n}(g;\hat{\bm{\sigma}}_n({\bm\pi}_M),{\bm{\pi}}^{(\hat{G}_n)}).
\end{align*}
If the argmax contains more than one index, we select the smallest index. The following consistency result is obtained.
\begin{theorem}\label{zconsis}
Suppose that ${[\bf{A1}]}$, ${[\bf{A2}]}$ and ${[\bf{B}]}$ hold. Then, for every ${\bm{\pi}}_M\in\Pi_M$ and each $j=1,\ldots,J$,
\begin{align*}
    {\rm{P}}_{\bm{z}}\biggl(\hat{Z}_{j,n}({\bm{\pi}}_M)=z_j\biggr)\longrightarrow 1
\end{align*}
as $n\longrightarrow\infty$.
\end{theorem}
The above results imply that the choice of ${\bm{\pi}}_M$ does not asymptotically affect the estimation of the true number of clusters $G^*$, the true diffusion parameter vector ${\bm{\sigma}}^*$ and the vector of true cluster labels ${\bm{z}}$. Hence, if we are not interested in estimating ${\bm{\pi}}^*$, we may simply set, for example, 
\begin{align*}
    \pi^{(G)}_g=\frac{1}{G}\quad(g=1,\ldots,G)
\end{align*}
for each $G=1,\ldots,M$. Note that in general,  
${\bm{\pi}}^*$ cannot be consistently estimated when $J$ is fixed. Therefore, we define 
\begin{align*}
    \check{G}_n\in \argmin_{G\in\{1,\ldots,M\}}\Bigl\{-2\sup_{{\bm{\sigma}}^{(G)}\in\overline{\Sigma}^{(G)}}{\bf{H}}^{(G)}_{n}({\bm{\sigma}}^{(G)})+P^{(G)}_{n}\Bigr\}
\end{align*}
and 
\begin{align*}
    \check{\bm{\sigma}}_n\in
    \argsup_{{\bm{\sigma}}^{(\check{G}_n)}\in\overline{\Sigma}^{(\check{G}_n)}}
    {\bf L}^{(\check{G}_n)}_n({\bm{\sigma}}^{(\check{G}_n)}),
\end{align*}
where
\begin{align*}
    {\bf L}^{(G)}_n({\bm{\sigma}}^{(G)})=\prod_{j=1}^J\frac{1}{G}\sum_{g=1}^{G}\bar{\bf L}_{j,n}(\sigma^{(G)}_g)
\end{align*}
and
\begin{align*}
    {\bf{H}}^{(G)}_n({\bm{\sigma}}^{(G)})
    &=\sum_{j=1}^J\log\Biggl(\frac{1}{G}\sum_{g=1}^{G}\exp\Bigl\{\bar{\bf{H}}_{j,n}(\sigma^{(G)}_g)\Bigr\}\Biggr).
\end{align*}
For each $j=1,\ldots,J$, the estimator of $z_{j}$ is defined by
\begin{align*}
    \check{Z}_{j,n}
    \in\argmax_{g\in\{1,\ldots,\check{G}_n\}}{\bf{Q}}^{(\check{G}_n)}_{j,n}(g;\check{\bm{\sigma}}_n),
\end{align*}
where 
\begin{align*}
    {\bf{Q}}^{(G)}_{j,n}(g;{\bm{\sigma}}^{(G)})=\frac{\exp\Bigl\{\bar{\bf{H}}_{j,n}(\sigma^{(G)}_g)\Bigr\}
    }{\displaystyle\sum_{k=1}^{G}\exp\Bigl\{\bar{\bf{H}}_{j,n}(\sigma^{(G)}_k)\Bigr\}}
\end{align*}
for $g=1,\ldots,G$. Theorems \ref{Gconsis}--\ref{zconsis} can be simplified as follows.
\begin{corollary}\label{Gconsis2}
Suppose that ${[\bf{A1}]}$, ${[\bf{A2}]}$ and ${[\bf{B}]}$ hold. Then, 
\begin{align*}
    {\rm{P}}_{\bm{z}}\biggl(\check{G}_n=G^*\biggr)\longrightarrow 1
\end{align*}
as $n\longrightarrow\infty$.
\end{corollary}
\begin{corollary}\label{consis2}
Suppose that ${[\bf{A1}]}$, ${[\bf{A2}]}$ and ${[\bf{B}]}$ hold. Then, 
\begin{align*}
    \begin{bmatrix}
    \check{\bm{\sigma}}_n\\
    {\bm{0}}_{M-\check{G}_n}
    \end{bmatrix}\xlongrightarrow{p}\begin{bmatrix}
    {\bm{\sigma}}^*\\
    {\bm{0}}_{M-G^*}
    \end{bmatrix}
\end{align*}
and
\begin{align*}
    &\sqrt{n}\Biggl(\, \begin{bmatrix}
    \check{\bm{\sigma}}_n\\
    {\bm{0}}_{M-\check{G}_n}
    \end{bmatrix}-\begin{bmatrix}
    {\bm{\sigma}}^*\\
    {\bm{0}}_{M-G^*}
    \end{bmatrix}\, \Biggr)\xlongrightarrow{d}N_{M}\left({\bm{0}}_M,\, \begin{bmatrix}
    I({\bm{\sigma}}^*)^{-1} & O_{G^*,M-G^*}\\
    O_{M-G^*,G^*} & O_{M-G^*,M-G^*}
    \end{bmatrix}\right)
\end{align*}
under ${\rm{P}}_{\bm{z}}$ as $n\longrightarrow\infty$.
\end{corollary}
\begin{corollary}\label{zconsis2}
Suppose that ${[\bf{A1}]}$, ${[\bf{A2}]}$ and ${[\bf{B}]}$ hold. Then, for each $j=1,\ldots,J$,
\begin{align*}
    {\rm{P}}_{\bm{z}}\biggl(\check{Z}_{j,n}=z_j\biggr)\longrightarrow 1
\end{align*}
as $n\longrightarrow\infty$.
\end{corollary}
\begin{remark}
When $G^*$ is known, results analogous to Theorems \ref{consis}--\ref{zconsis} and Corollaries \ref{consis2}--\ref{zconsis2}
can be obtained; see Propositions \ref{consisprop}, \ref{asymprop} and \ref{Zstarconsis}. Therefore, the proposed method can be applied whether $G^*$ is known or unknown.
\end{remark}
\section{Numerical studies}\label{simulation}
\subsection{True model}
We consider the finite mixture model introduced in Section 2 with
$G^*=4$ and $J=20$. For $g=1,\ldots,4$, the drift functions are given by $a_g^{(G^*)}(x)=-2x$ and the initial values are set to $\xi^{(G^*)}_{g}=0$. The true mixing proportion vector is $\bm{\pi}^*=(0.3,0.2,0.3,0.2)^\top$, and the true diffusion coefficients are specified in the following two simulation settings:\vspace{2mm}
\begin{center}
\begin{tabular}{@{}l@{\quad}l@{\quad}l@{}}
    \textbf{Simulation 1:}
    & $\bm{\sigma}^*=(0.8,1.2,1.8,2.7)^\top$\\[4pt]

    \textbf{Simulation 2:}
    & $\bm{\sigma}^*=(1.0,1.2,1.4,1.6)^\top$.
\end{tabular}
\end{center}\vspace{2mm}
The $J$ cluster labels were generated only once and fixed as 
\begin{align*}
    \bm z=\bigl(2, 2, 1, 4, 4, 3, 2, 3, 1, 1, 4, 3, 3, 1, 1, 2, 4, 3, 3, 3\bigr)^\top
\end{align*}
throughout all replications in both simulation settings. The corresponding true clusters are
\begin{align*}
    S_1&=\bigl\{3,9,10,14,15\bigr\},\quad S_2=\bigl\{1,2,7,16\bigr\},\\
    S_3&=\bigl\{6,8,12,13,18,19,20\bigr\},\quad S_4=\bigl\{4,5,11,17\bigr\}.
\end{align*}
\subsection{Simulation settings}
Numerical optimization was performed by using the \texttt{optim()}
function in R with the BFGS algorithm. The sample sizes were set to
$n=10^2,\,10^3,\,10^4,\,10^5,\, 10^6$. For each simulation setting and each sample size, we conducted 10,000 Monte Carlo replications with $T=1$. The maximum number of clusters was set to $M=8$. 
We set the initial value to the true parameter value or an optimal parameter value. To select the number of clusters, we used the BIC-type penalty $P_n^{(G)}=G\log n$.
\subsection{Simulation results}
Tables \ref{1G} and \ref{2G} show the number of times 
$\check{G}_n=G$ for each $G=1,\ldots,8$ in Simulations 1 and 2, respectively. The true number of clusters $G^*=4$ was selected most frequently for sufficiently large $n$, and the correct selection rate approaches $1$ as the sample size increases, which is consistent with Corollary \ref{Gconsis2}. A high rate of correct selection was achieved at smaller sample sizes in Simulation 1 than in Simulation 2. This suggests that larger sample sizes are required to correctly
identify the number of clusters when the clusters are less separated. Tables \ref{1est} and \ref{2est} show the sample means and sample standard deviations (S.D.s) of the components of $\check{\bm{\sigma}}_n$ in Simulations 1 and 2, respectively. Figures \ref{1F} and \ref{2F} show histograms, normal Q--Q plots and empirical distribution functions of the components of $\sqrt{n}(\check{\bm{\sigma}}_{n}-{\bm{\sigma}}^{*})$ when $n=10^6$ in Simulations 1 and 2, respectively. These results are consistent with Corollary \ref{consis2}. Tables \ref{1Z} and \ref{2Z} show the 
proportions of replications in which $\check{Z}_{j,n}=z_j$ for each $j=1,\ldots,20$ in Simulations 1 and 2, respectively. The correct classification rates tend to increase as the sample size increases. These results are consistent with Corollary \ref{zconsis2}. Note that the results in Tables \ref{1est}, \ref{1Z}, \ref{2est} and \ref{2Z} and Figures \ref{1F} and \ref{2F} are based only on replications in which $\check{G}_n=G^*$.
\newpage
\ \vspace{0.5cm}
\begin{longtable}{c|cccccccc}
    & $G=1$ & $G=2$ & $G=3$ & $G=4$ & $G=5$ & $G=6$ & $G=7$ & $G=8$\\ \hline\hline
    $n=10^2$ & 0 & 0 & 0 & \textbf{9,780} & 218 & 2 & 0 & 0 \\
    $n=10^3$ & 0 & 0 & 0 & \textbf{9,943} & 56 & 1 & 0 & 0 \\
    $n=10^4$ & 0 & 0 & 0 & \textbf{9,979} & 21 & 0 & 0 & 0 \\
    $n=10^5$ & 0 & 0 & 0 & \textbf{9,996} & 4 & 0 & 0 & 0 \\
    $n=10^6$ & 0 & 0 & 0 & \textbf{10,000} & 0 & 0 & 0 & 0\\
    \caption{The number of times $\check{G}_n=G$ for each $G=1,\ldots,8$ in Simulation 1.}\label{1G}
\end{longtable}\
\begin{longtable}{c|cccc}
    & $\check{{\bm{\sigma}}}_n^{(1)}(0.8)$ & $\check{{\bm{\sigma}}}_n^{(2)}(1.2)$ & $\check{{\bm{\sigma}}}_n^{(3)}(1.8)$ & $\check{{\bm{\sigma}}}_n^{(4)}(2.7)$\\ \hline\hline
    $n=10^2$ & 0.8026 (0.0258) & 1.2059 (0.0460) & 1.8066 (0.0489) & 2.7073 (0.0979)\\
    $n=10^3$ & 0.8003 (0.0081) & 1.2003 (0.0136) & 1.8005 (0.0150) & 2.7011 (0.0303)\\
    $n=10^4$ & 0.8000 (0.0025) & 1.2000 (0.0042) & 1.8001 (0.0048) & 2.7000 (0.0096)\\
    $n=10^5$ & 0.8000 (0.0008) & 1.2000 (0.0013) & 1.8000 (0.0015) & 2.7000 (0.0030)\\
    $n=10^6$ & 0.8000 (0.0003) & 1.2000 (0.0004) & 1.8000 (0.0005) & 2.7000 (0.0010)\\
    \caption{Sample means and sample standard deviations (S.D.s) of the components of $\check{{\bm{\sigma}}}_n$ in Simulation 1.}\label{1est}
\end{longtable}\ 
\begin{longtable}{c|cccccccccc}
    & $z_1$ & $z_2$ & $z_3$ & $z_4$ & $z_5$ & $z_6$ & $z_7$ & $z_8$ & $z_9$ & $z_{10}$\\ \hline\hline
    $n=10^2$ & 1.0000 & 1.0000 & 1.0000 & 1.0000 & 1.0000 & 1.0000 & 1.0000 & 1.0000 & 1.0000 & 1.0000\\
    $n=10^3$ & 1.0000 & 1.0000 & 1.0000 & 1.0000 & 1.0000 & 1.0000 & 1.0000 & 1.0000 & 1.0000 & 1.0000\\
    $n=10^4$ & 1.0000 & 1.0000 & 1.0000 & 1.0000 & 1.0000 & 1.0000 & 1.0000 & 1.0000 & 1.0000 & 1.0000\\
    $n=10^5$ & 1.0000 & 1.0000 & 1.0000 & 1.0000 & 1.0000 & 1.0000 & 1.0000 & 1.0000 & 1.0000 & 1.0000\\
    $n=10^6$ & 1.0000 & 1.0000 & 1.0000 & 1.0000 & 1.0000 & 1.0000 & 1.0000 & 1.0000 & 1.0000 & 1.0000\\
    & $z_{11}$ & $z_{12}$ & $z_{13}$ & $z_{14}$ & $z_{15}$ & $z_{16}$ & $z_{17}$ & $z_{18}$ & $z_{19}$ & $z_{20}$\\ \hline\hline
    $n=10^2$ & 1.0000 & 1.0000 & 1.0000 & 1.0000 & 1.0000 & 1.0000 & 1.0000 & 1.0000 & 1.0000 & 1.0000\\
    $n=10^3$ & 1.0000 & 1.0000 & 1.0000 & 1.0000 & 1.0000 & 1.0000 & 1.0000 & 1.0000 & 1.0000 & 1.0000\\
    $n=10^4$ & 1.0000 & 1.0000 & 1.0000 & 1.0000 & 1.0000 & 1.0000 & 1.0000 & 1.0000 & 1.0000 & 1.0000\\
    $n=10^5$ & 1.0000 & 1.0000 & 1.0000 & 1.0000 & 1.0000 & 1.0000 & 1.0000 & 1.0000 & 1.0000 & 1.0000\\
    $n=10^6$ & 1.0000 & 1.0000 & 1.0000 & 1.0000 & 1.0000 & 1.0000 & 1.0000 & 1.0000 & 1.0000 & 1.0000\\
    \caption{Proportions of replications in which $\check{Z}_{j,n}=z_j$ for each $j=1,\ldots,20$ in Simulation 1.}\label{1Z}
\end{longtable}
\newpage
\ \vspace{20mm}\\
\begin{figure}[h]
    \centering
    \includegraphics[width=0.30\columnwidth]{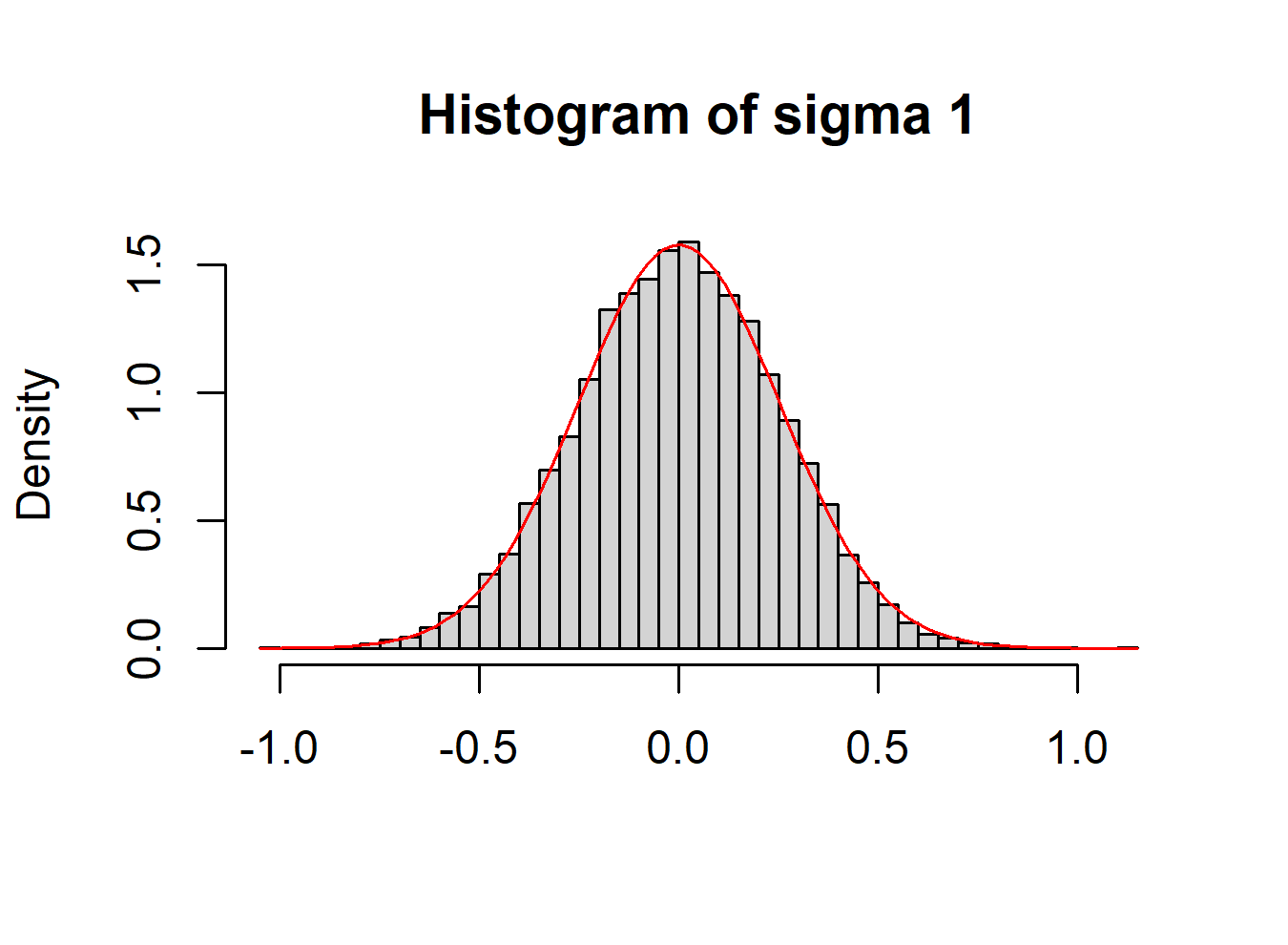}
    \includegraphics[width=0.30\columnwidth]{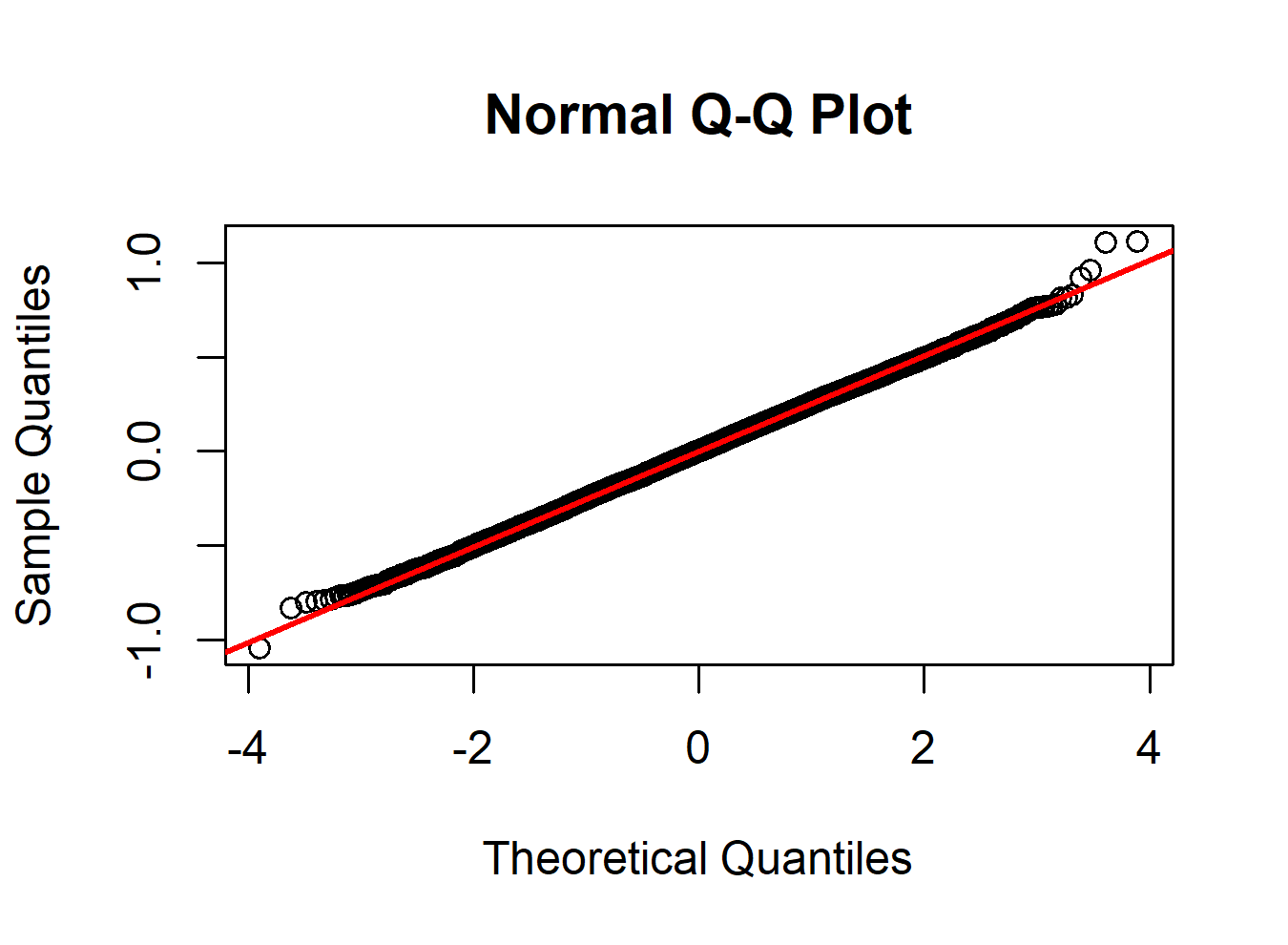}
    \includegraphics[width=0.30\columnwidth]{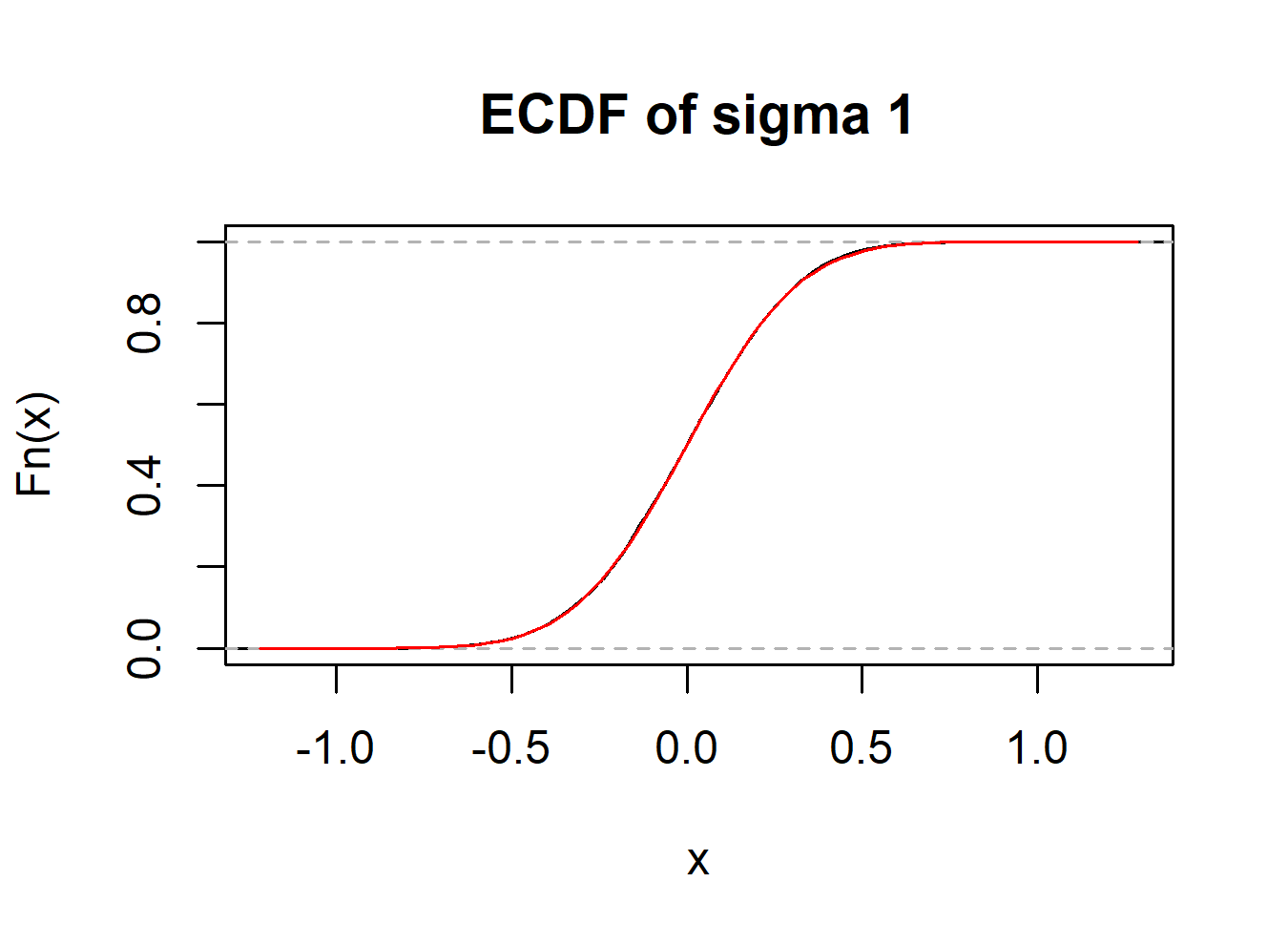}\vspace{2mm}\\
    \includegraphics[width=0.30\columnwidth]{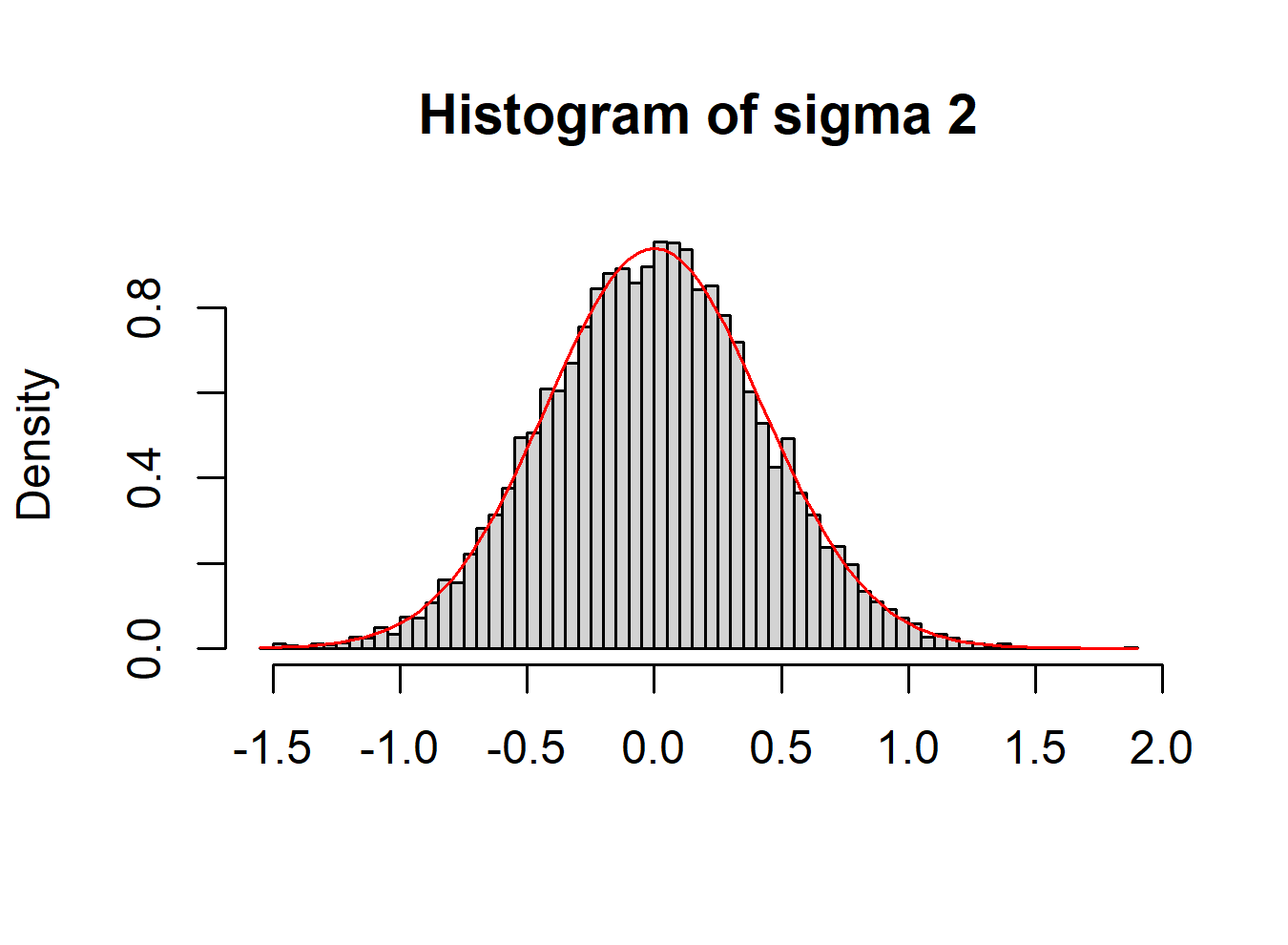}
    \includegraphics[width=0.30\columnwidth]{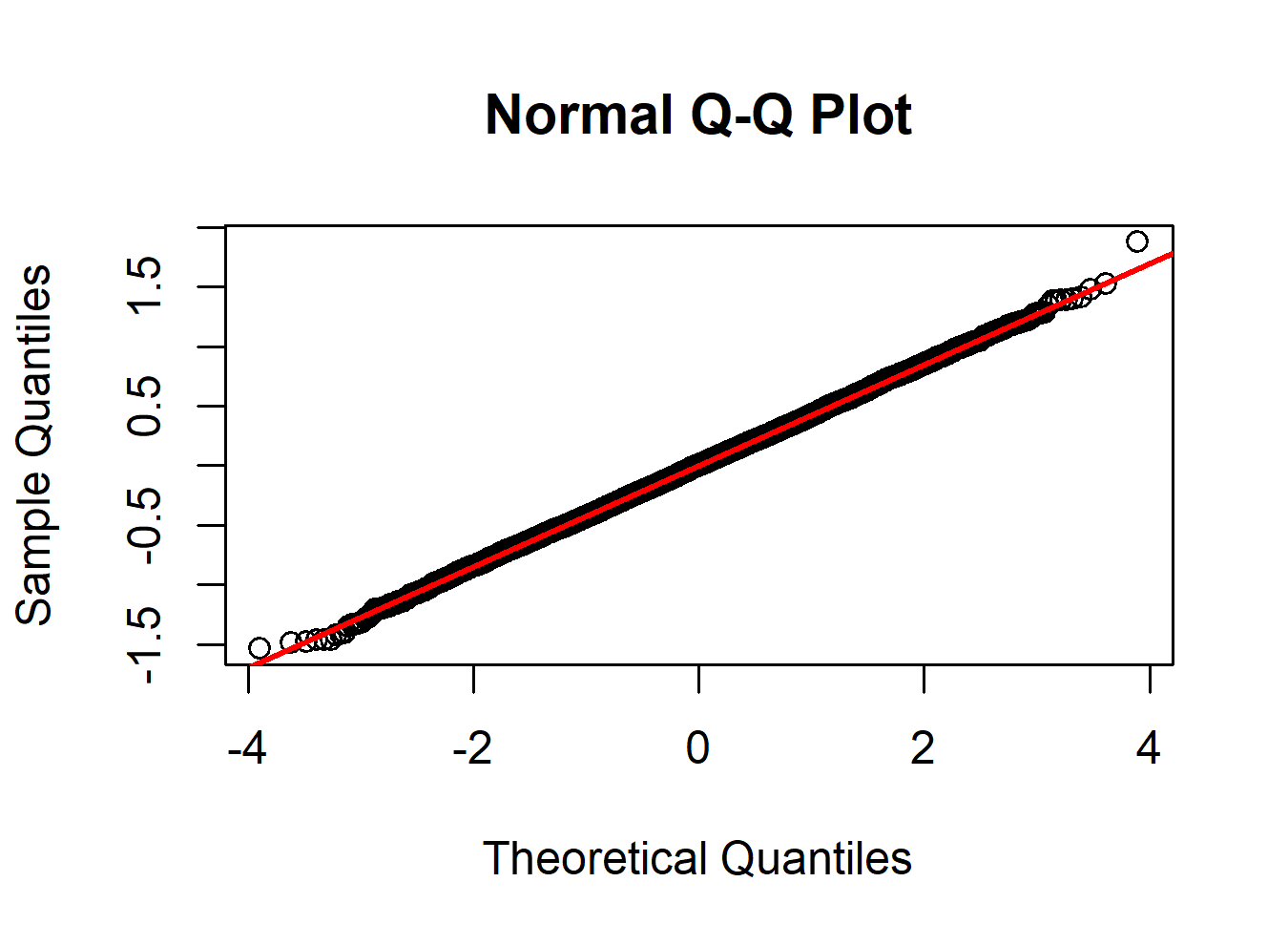}
    \includegraphics[width=0.30\columnwidth]{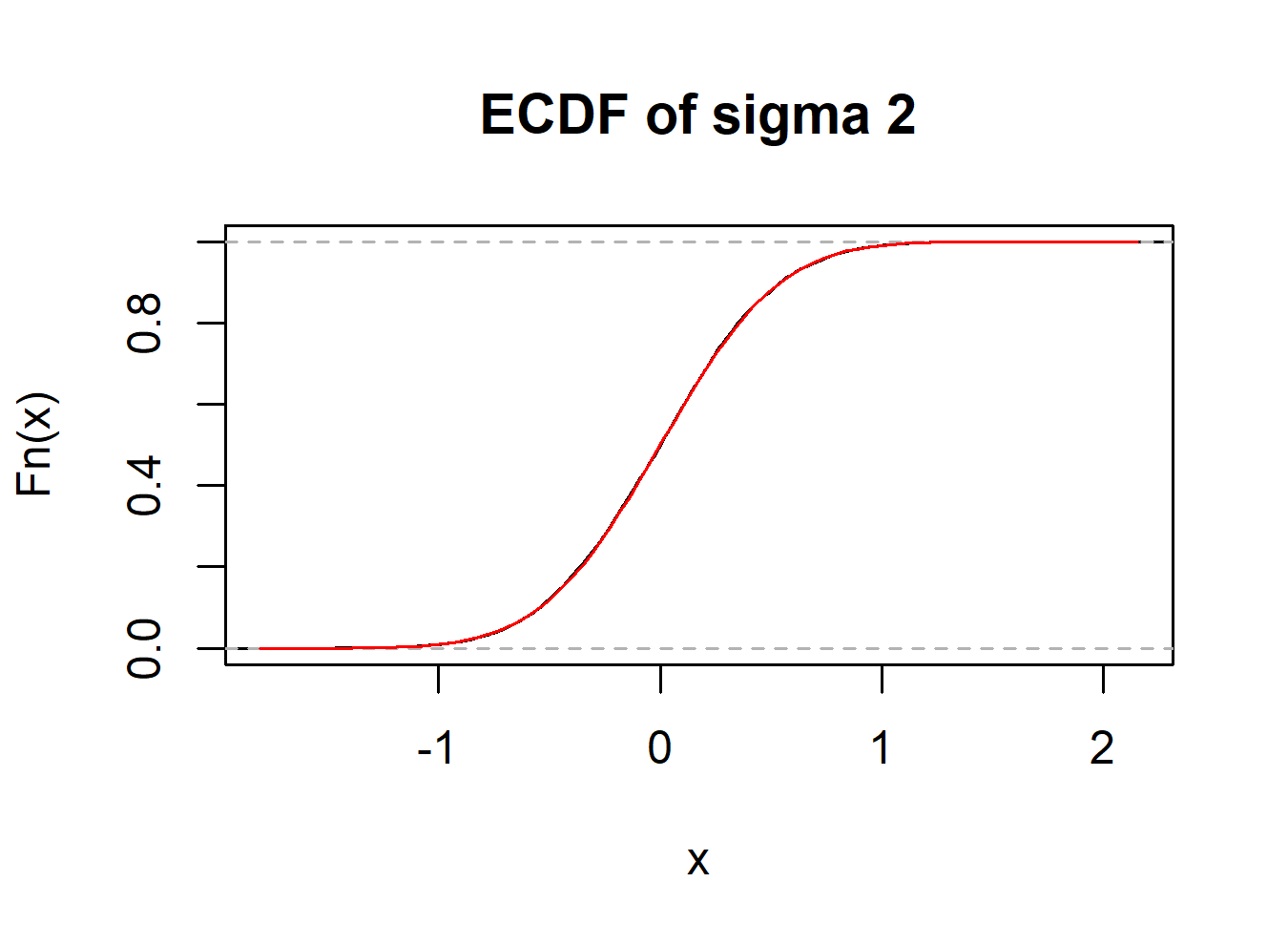}\vspace{2mm}\\
    \includegraphics[width=0.30\columnwidth]{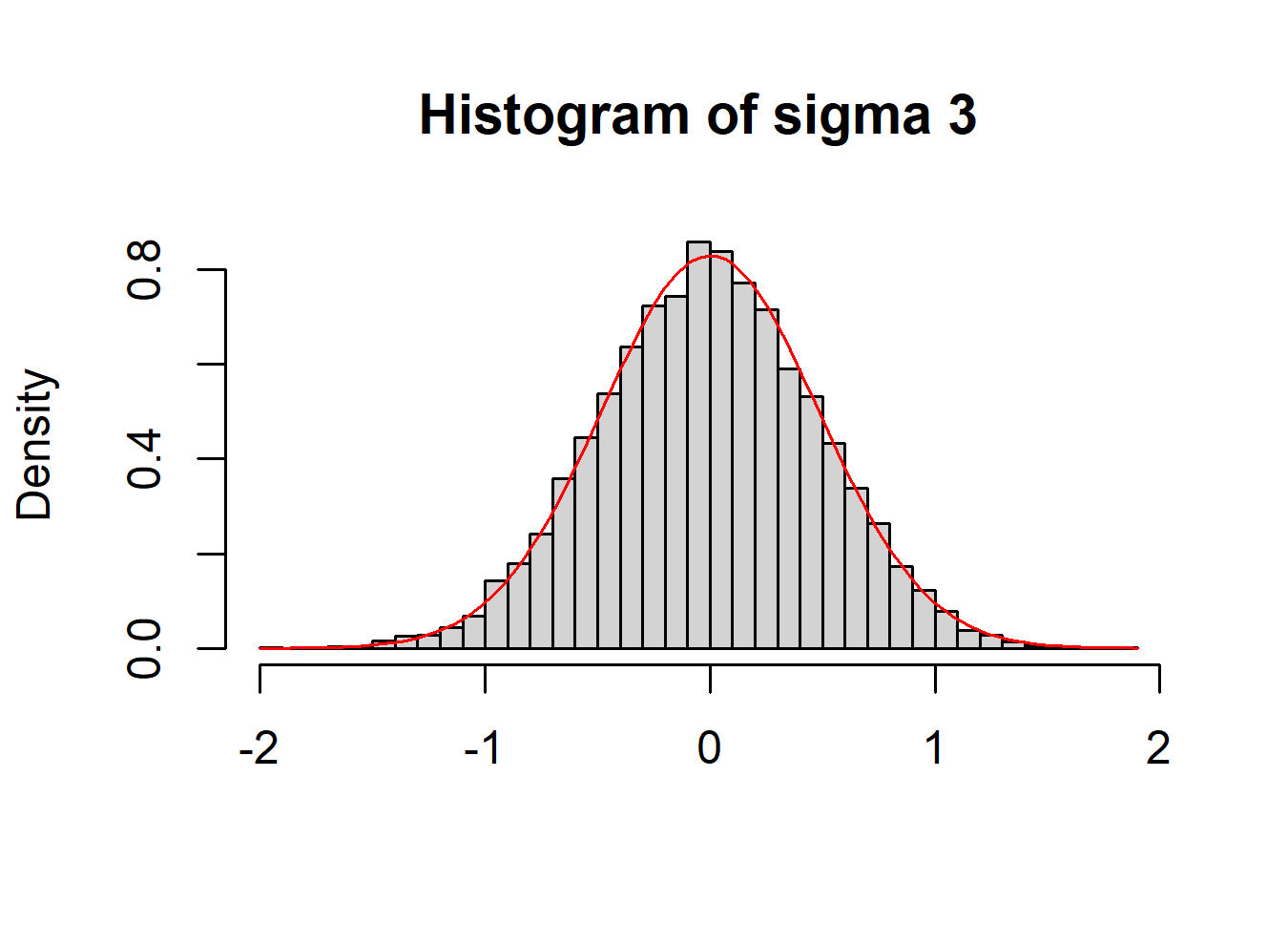}
    \includegraphics[width=0.30\columnwidth]{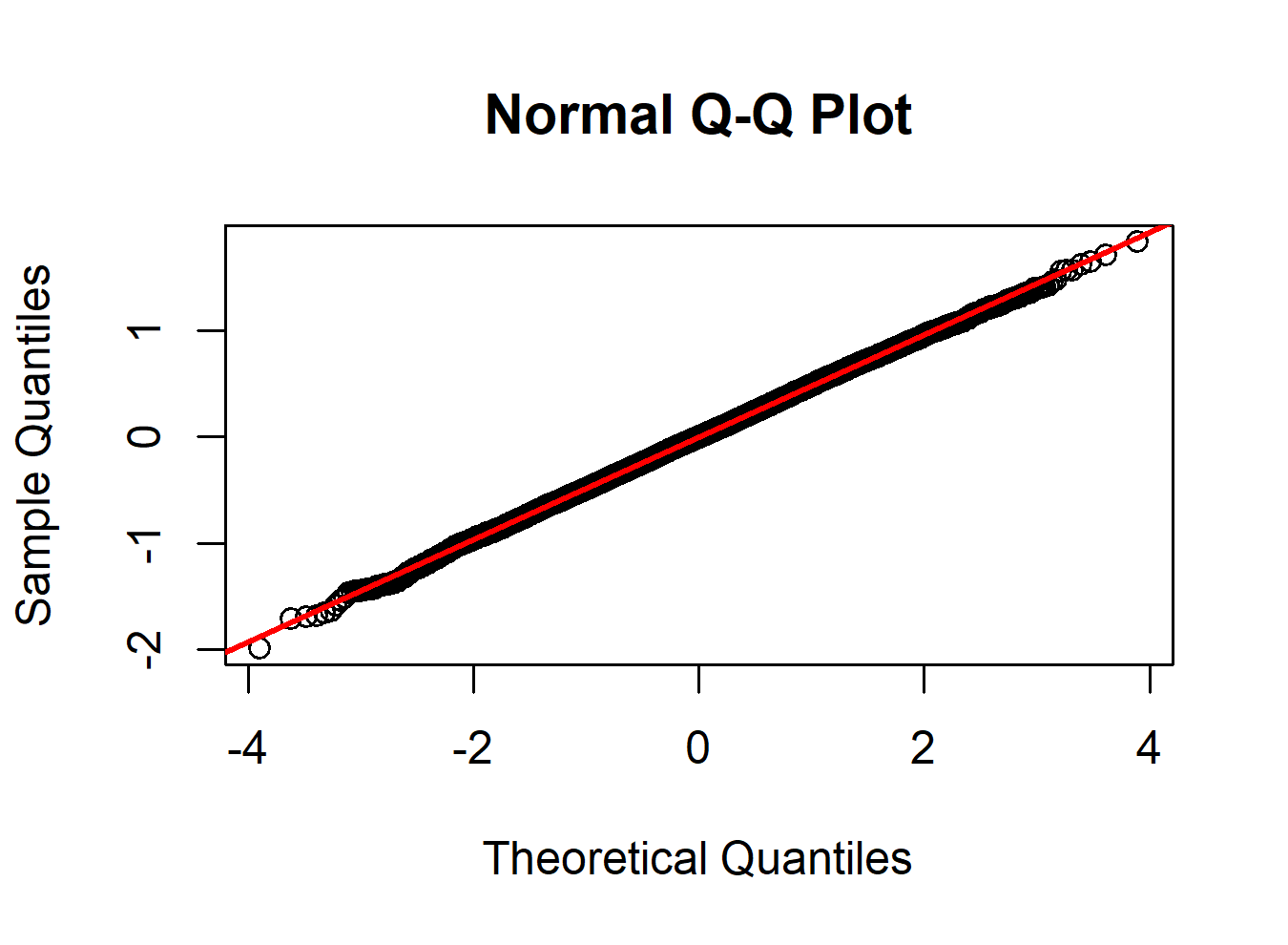}
    \includegraphics[width=0.30\columnwidth]{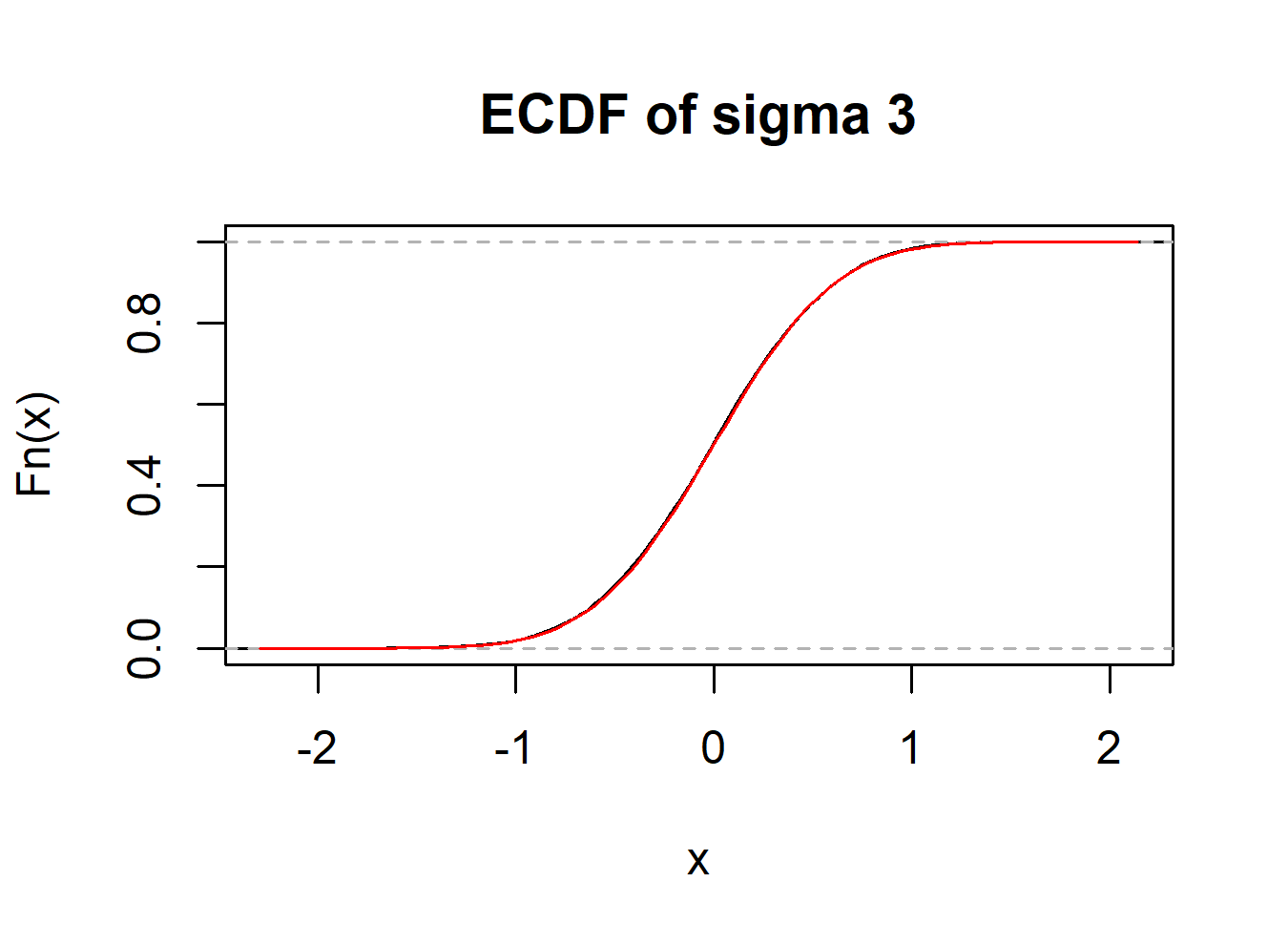}\vspace{2mm}\\
    \includegraphics[width=0.30\columnwidth]{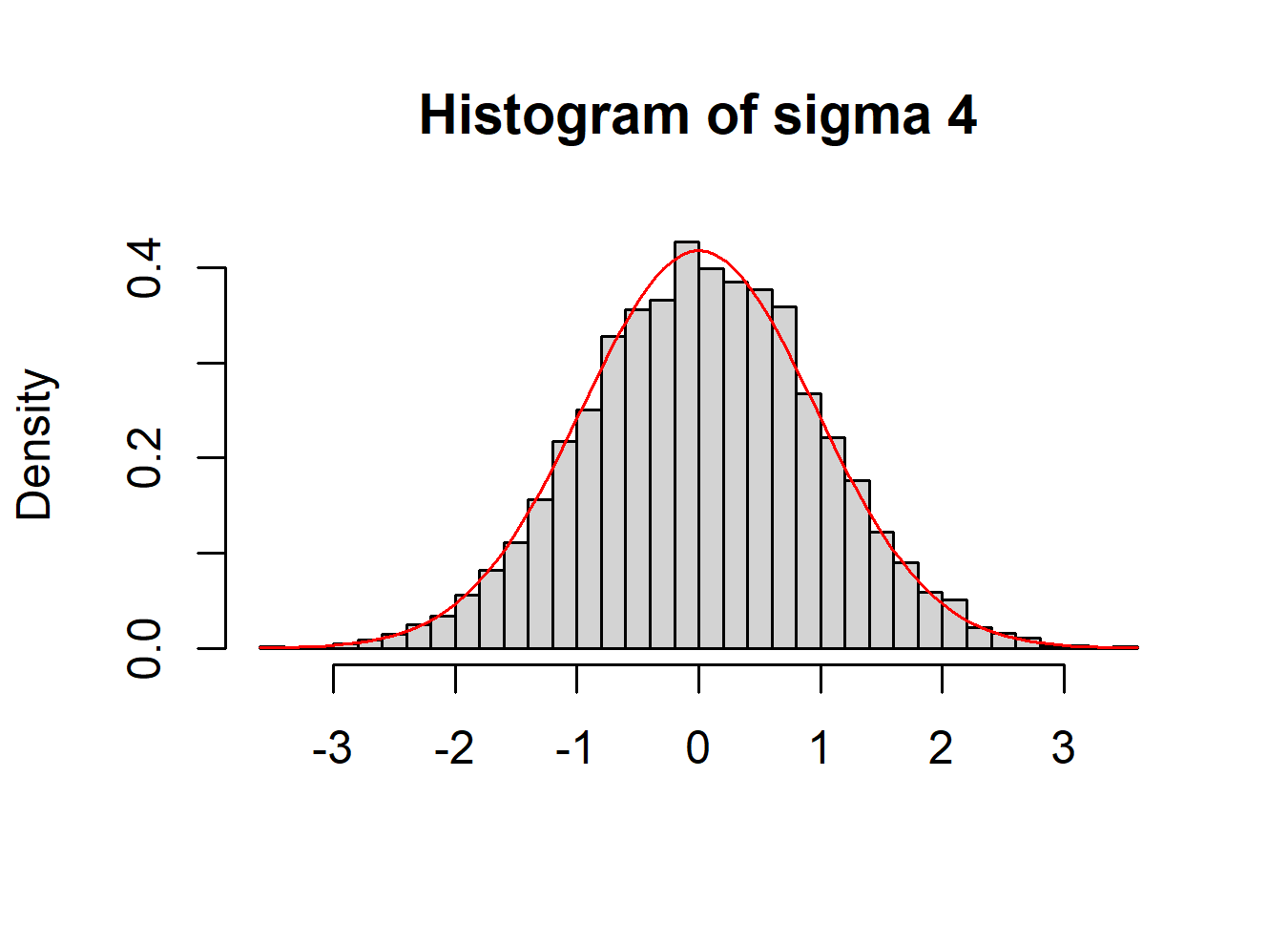}
    \includegraphics[width=0.30\columnwidth]{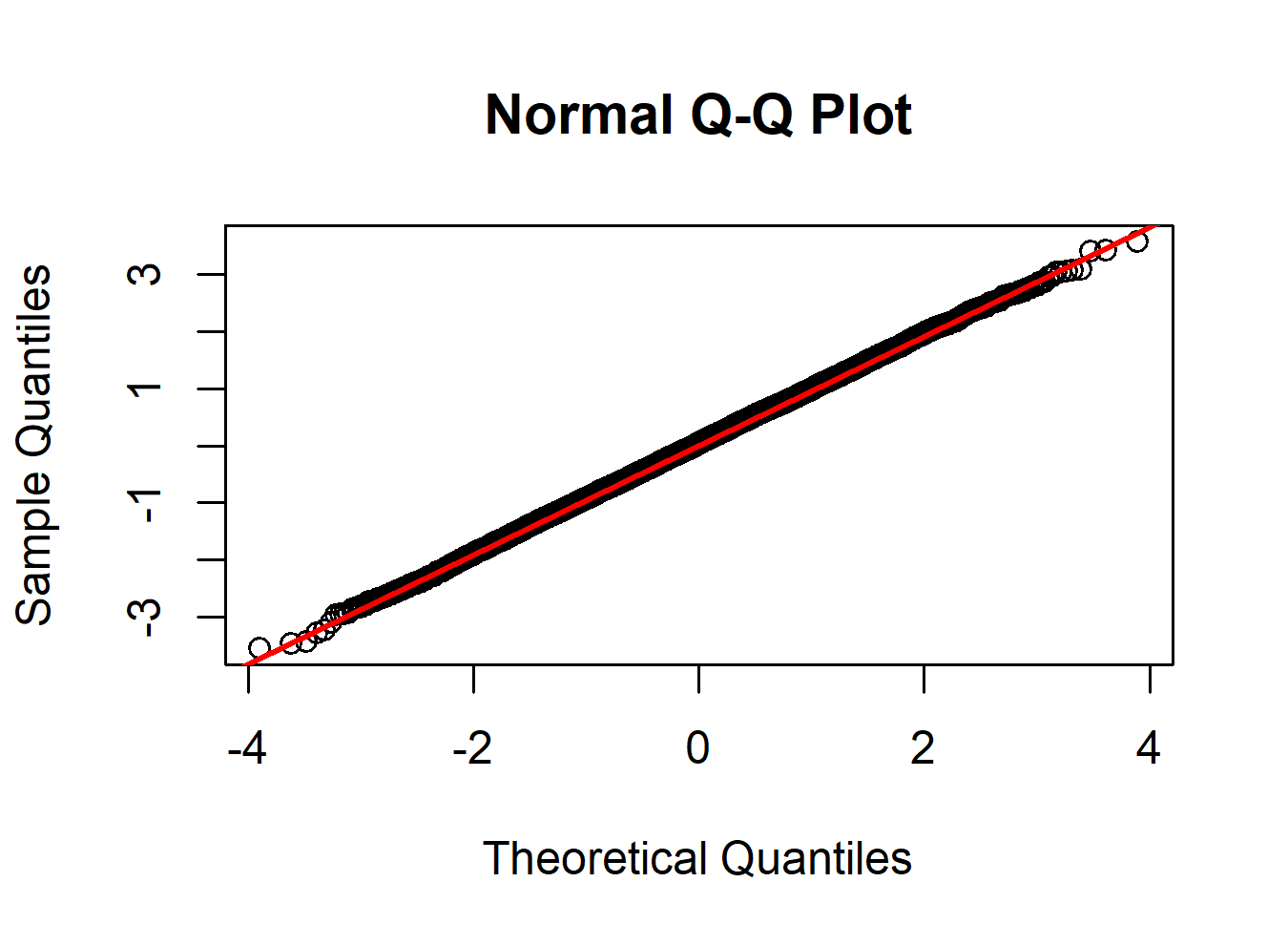}
    \includegraphics[width=0.30\columnwidth]{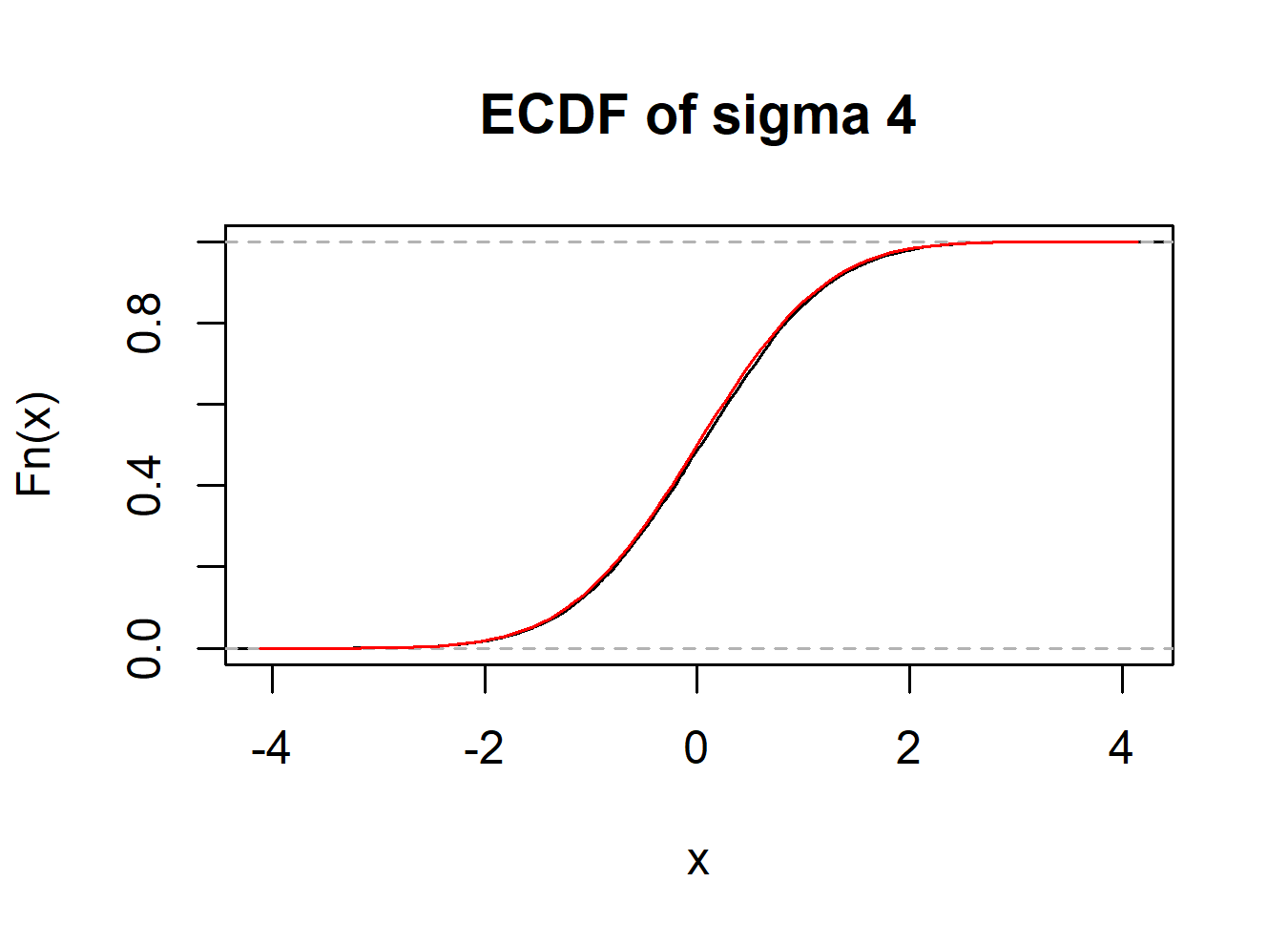}
    \caption{Histogram (left), Q-Q plot (middle) and empirical distribution function (right) of the components of $\sqrt{n}(\check{\bm{\sigma}}_{n}-{\bm{\sigma}}^{*})$ when $n=10^6$ in Simulation 1. The red lines are theoretical curves.}\label{1F}
\end{figure}
\newpage
\ \vspace{0.5mm}
\begin{longtable}{c|cccccccc}
    & $G=1$ & $G=2$ & $G=3$ & $G=4$ & $G=5$ & $G=6$ & $G=7$ & $G=8$\\ \hline\hline
    $n=10^2$ & 0 & 3,672 & 6,217 & \textbf{110} & 1 & 0 & 0 & 0 \\
    $n=10^3$ & 0 & 0 & 0 & \textbf{9,947} & 52 & 1 & 0 & 0 \\
    $n=10^4$ & 0 & 0 & 0 & \textbf{9,979} & 21 & 0 & 0 & 0 \\
    $n=10^5$ & 0 & 0 & 0 & \textbf{9,996} & 4 & 0 & 0 & 0 \\
    $n=10^6$ & 0 & 0 & 0 & \textbf{10,000} & 0 & 0 & 0 & 0\\
    \caption{The number of times $\check{G}_n=G$ for each $G=1,\ldots,8$ in Simulation 2.}\label{2G}
\end{longtable}\
\begin{longtable}{c|cccc}
    & $\check{{\bm{\sigma}}}_n^{(1)}(1.0)$ & $\check{{\bm{\sigma}}}_n^{(2)}(1.2)$ & $\check{{\bm{\sigma}}}_n^{(3)}(1.4)$ & $\check{{\bm{\sigma}}}_n^{(4)}(1.6)$\\ \hline\hline
    $n=10^2$ & 0.9509 (0.0472) & 1.2298 (0.0794) & 1.3870 (0.0542) & 1.7329 (0.1106)\\
    $n=10^3$ & 1.0004 (0.0101) & 1.2004 (0.0138) & 1.4003 (0.0117) & 1.6005 (0.0184)\\
    $n=10^4$ & 1.0001 (0.0032) & 1.2000 (0.0042) & 1.4001 (0.0037) & 1.6000 (0.0057)\\
    $n=10^5$ & 1.0000 (0.0010) & 1.2000 (0.0013) & 1.4000 (0.0012) & 1.6000 (0.0018)\\
    $n=10^6$ & 1.0000 (0.0003) & 1.2000 (0.0004) & 1.4000 (0.0004) & 1.6000 (0.0006) \\
    \caption{Sample means and sample standard deviations (S.D.s) of the components of $\check{{\bm{\sigma}}}_n$ in Simulation 2.}\label{2est}
\end{longtable}\ 
\begin{longtable}{c|cccccccccc}
    & $z_1$ & $z_2$ & $z_3$ & $z_4$ & $z_5$ & $z_6$ & $z_7$ & $z_8$ & $z_9$ & $z_{10}$\\ \hline\hline
    $n=10^2$ & 0.8727 & 0.8636 & 0.8273 & 0.6727 & 0.7000 & 0.7818 & 0.8000 & 0.7273 & 0.8273 & 0.8727\\
    $n=10^3$ & 0.9997 & 0.9998 & 1.0000 & 0.9984 & 0.9976 & 0.9973 & 0.9998 & 0.9965 & 1.0000 & 0.9999 \\
    $n=10^4$ & 1.0000 & 1.0000 & 1.0000 & 1.0000 & 1.0000 & 1.0000 & 1.0000 & 1.0000 & 1.0000 & 1.0000\\
    $n=10^5$ & 1.0000 & 1.0000 & 1.0000 & 1.0000 & 1.0000 & 1.0000 & 1.0000 & 1.0000 & 1.0000 & 1.0000\\
    $n=10^6$ & 1.0000 & 1.0000 & 1.0000 & 1.0000 & 1.0000 & 1.0000 & 1.0000 & 1.0000 & 1.0000 & 1.0000\\
    & $z_{11}$ & $z_{12}$ & $z_{13}$ & $z_{14}$ & $z_{15}$ & $z_{16}$ & $z_{17}$ & $z_{18}$ & $z_{19}$ & $z_{20}$\\ \hline\hline
    $n=10^2$ & 0.7091 &0.7545 &0.7273 &0.7727 &0.7818 &0.8545 &0.7455 &0.7182 &0.7909 &0.7545\\
    $n=10^3$ & 0.9975 & 0.9967 & 0.9971 & 1.0000 & 0.9999 & 1.0000 & 0.9983 & 0.9971 & 0.9956 & 0.9969\\
    $n=10^4$ & 1.0000 & 1.0000 & 1.0000 & 1.0000 & 1.0000 & 1.0000 & 1.0000 & 1.0000 & 1.0000 & 1.0000\\
    $n=10^5$ & 1.0000 & 1.0000 & 1.0000 & 1.0000 & 1.0000 & 1.0000 & 1.0000 & 1.0000 & 1.0000 & 1.0000\\
    $n=10^6$ & 1.0000 & 1.0000 & 1.0000 & 1.0000 & 1.0000 & 1.0000 & 1.0000 & 1.0000 & 1.0000 & 1.0000\\
    \caption{Proportions of replications in which $\check{Z}_{j,n}=z_j$ for each $j=1,\ldots,20$ in Simulation 2.}\label{2Z}
\end{longtable}
\newpage
\ \vspace{20mm}\\
\begin{figure}[h]
    \centering
    \includegraphics[width=0.3\columnwidth]{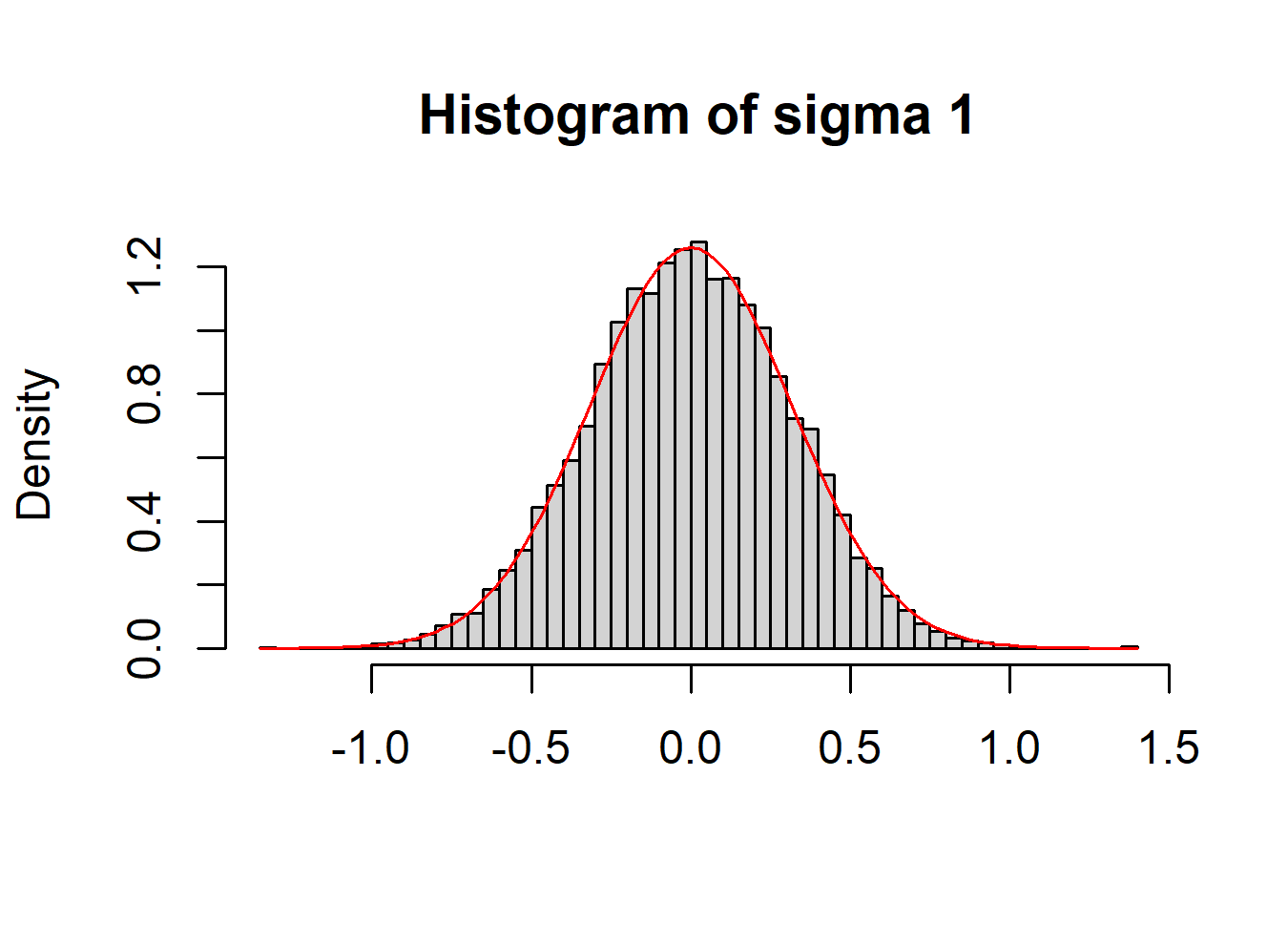}
    \includegraphics[width=0.3\columnwidth]{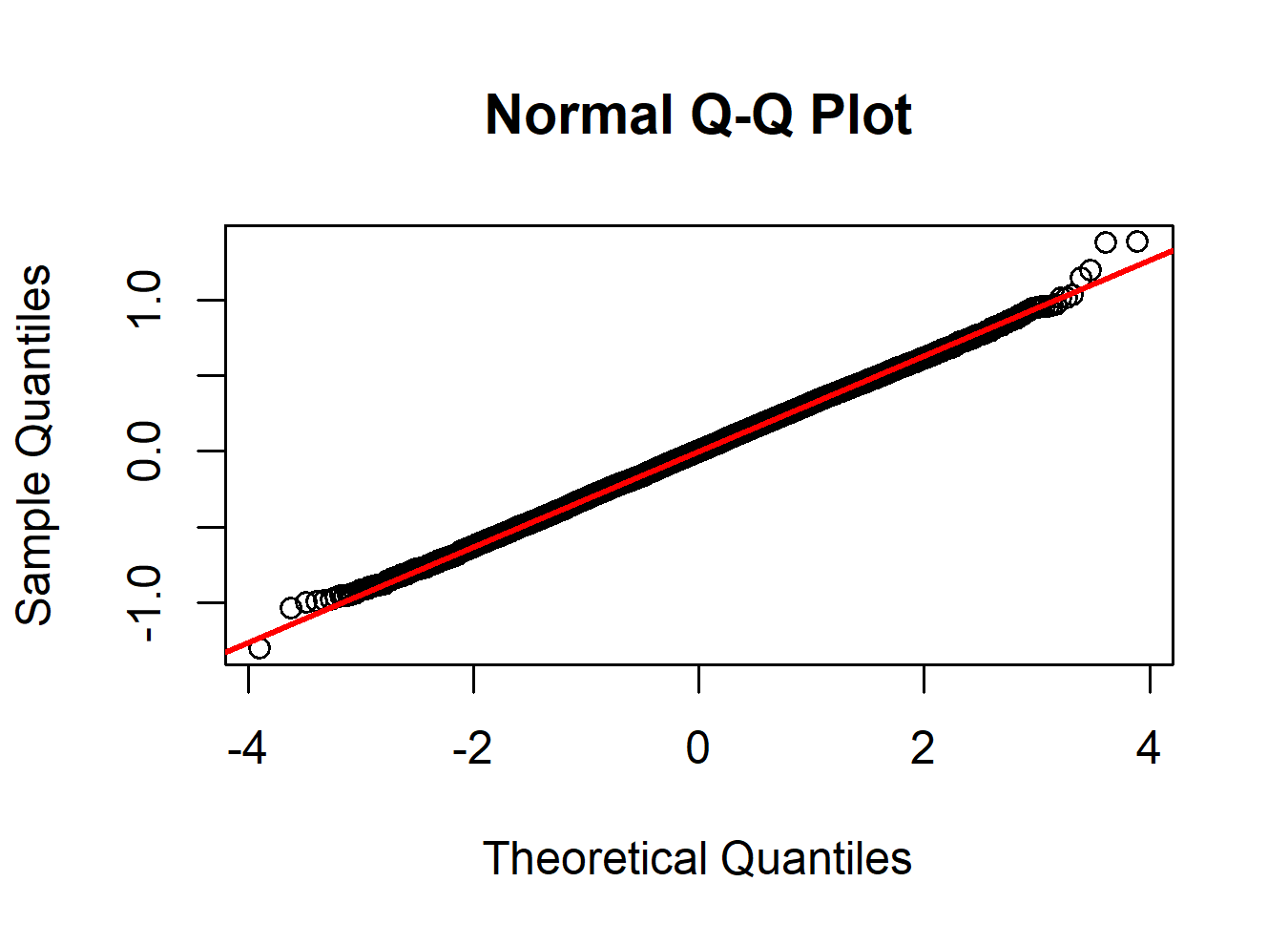}
    \includegraphics[width=0.3\columnwidth]{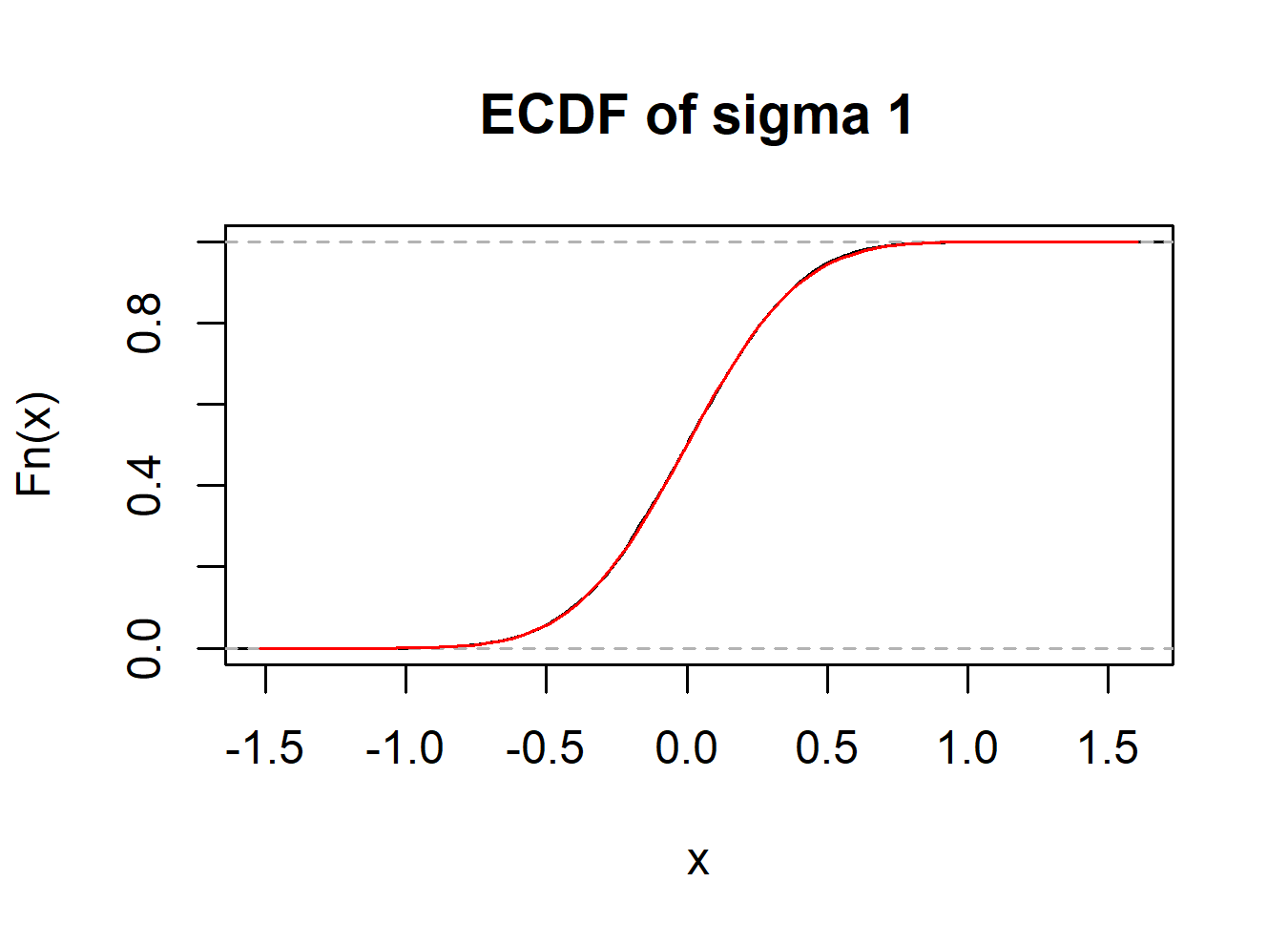}\vspace{2mm} \\
    \includegraphics[width=0.3\columnwidth]{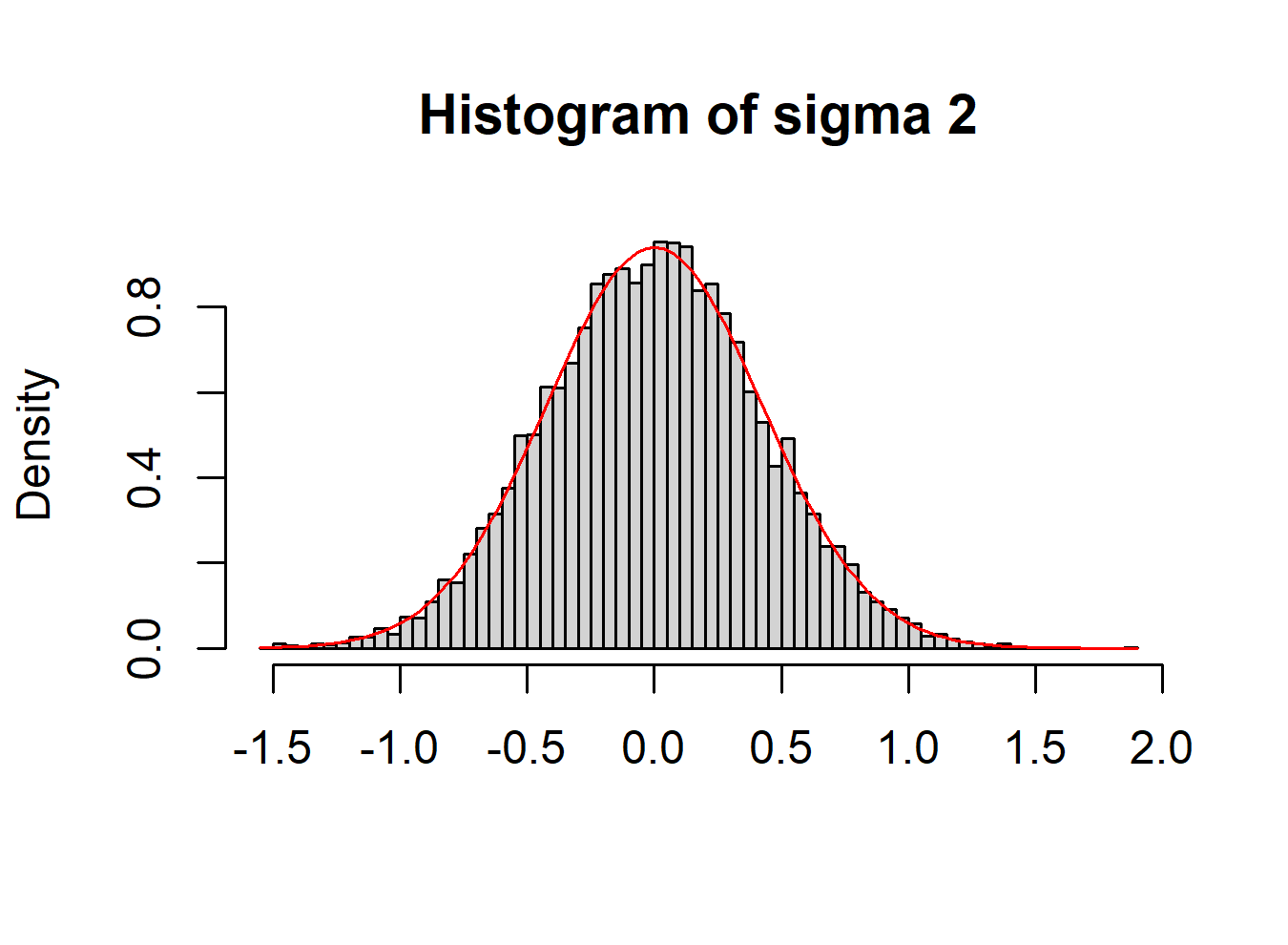}
    \includegraphics[width=0.3\columnwidth]{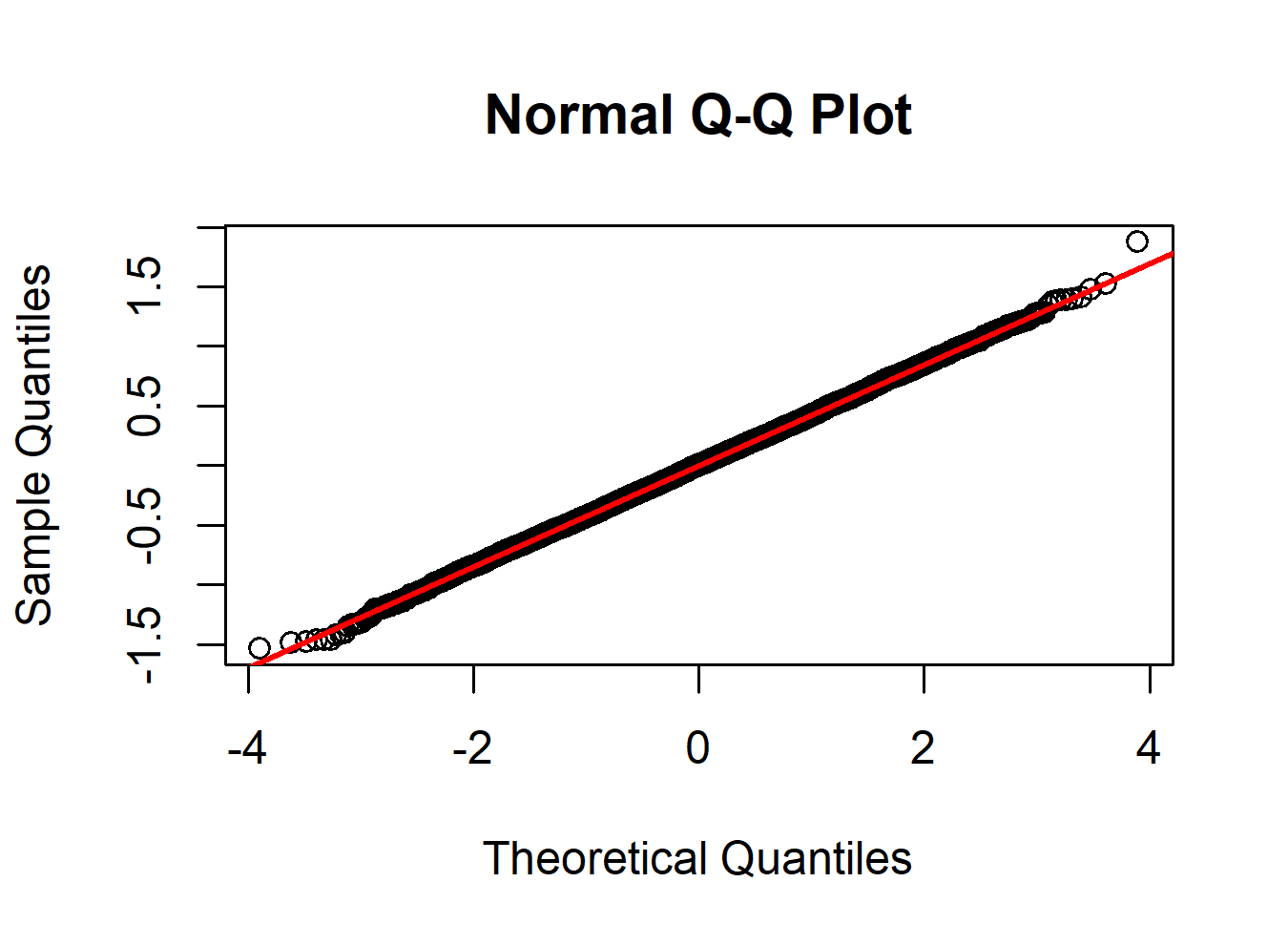}
    \includegraphics[width=0.3\columnwidth]{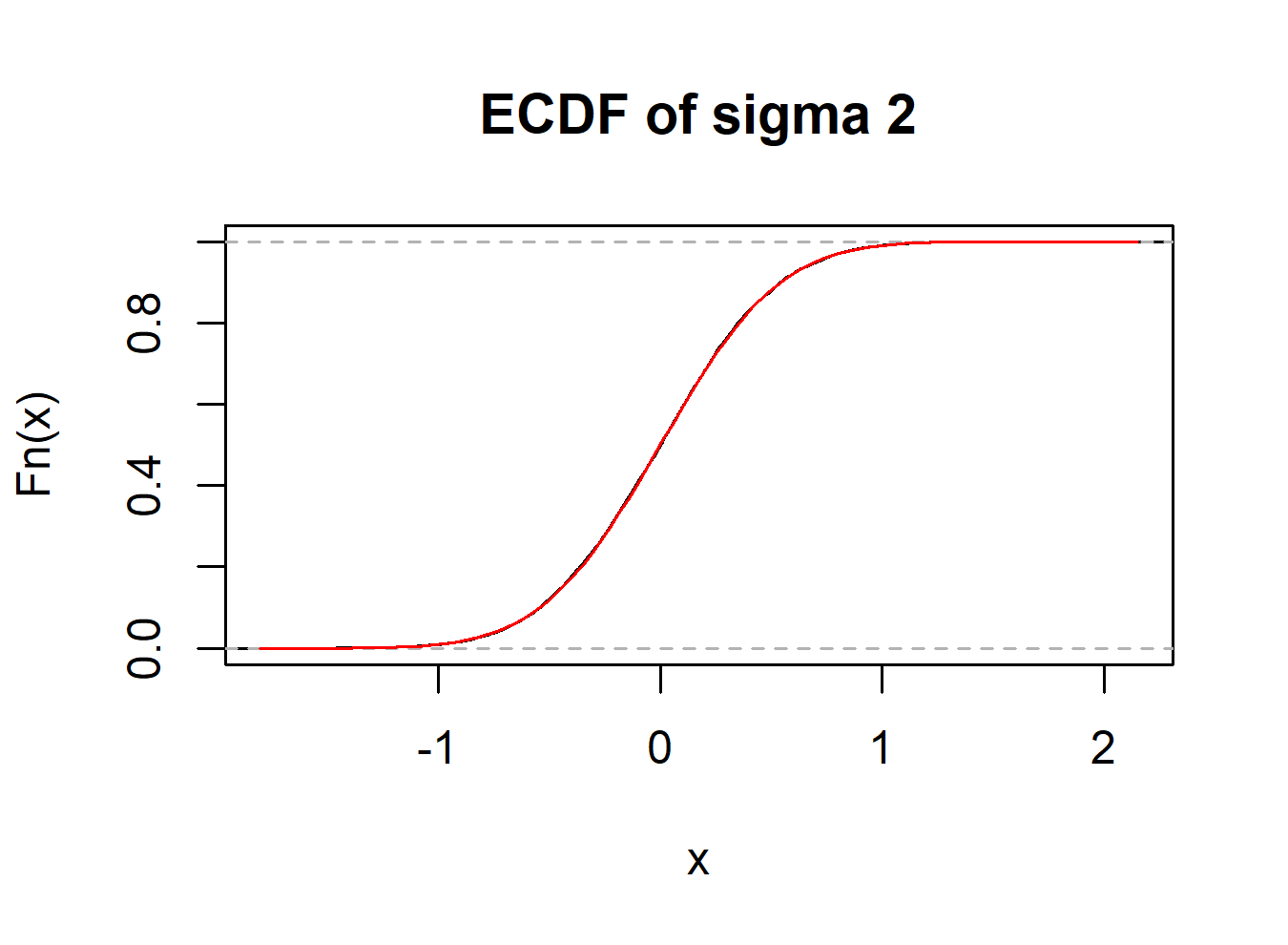}\vspace{2mm} \\
    \includegraphics[width=0.3\columnwidth]{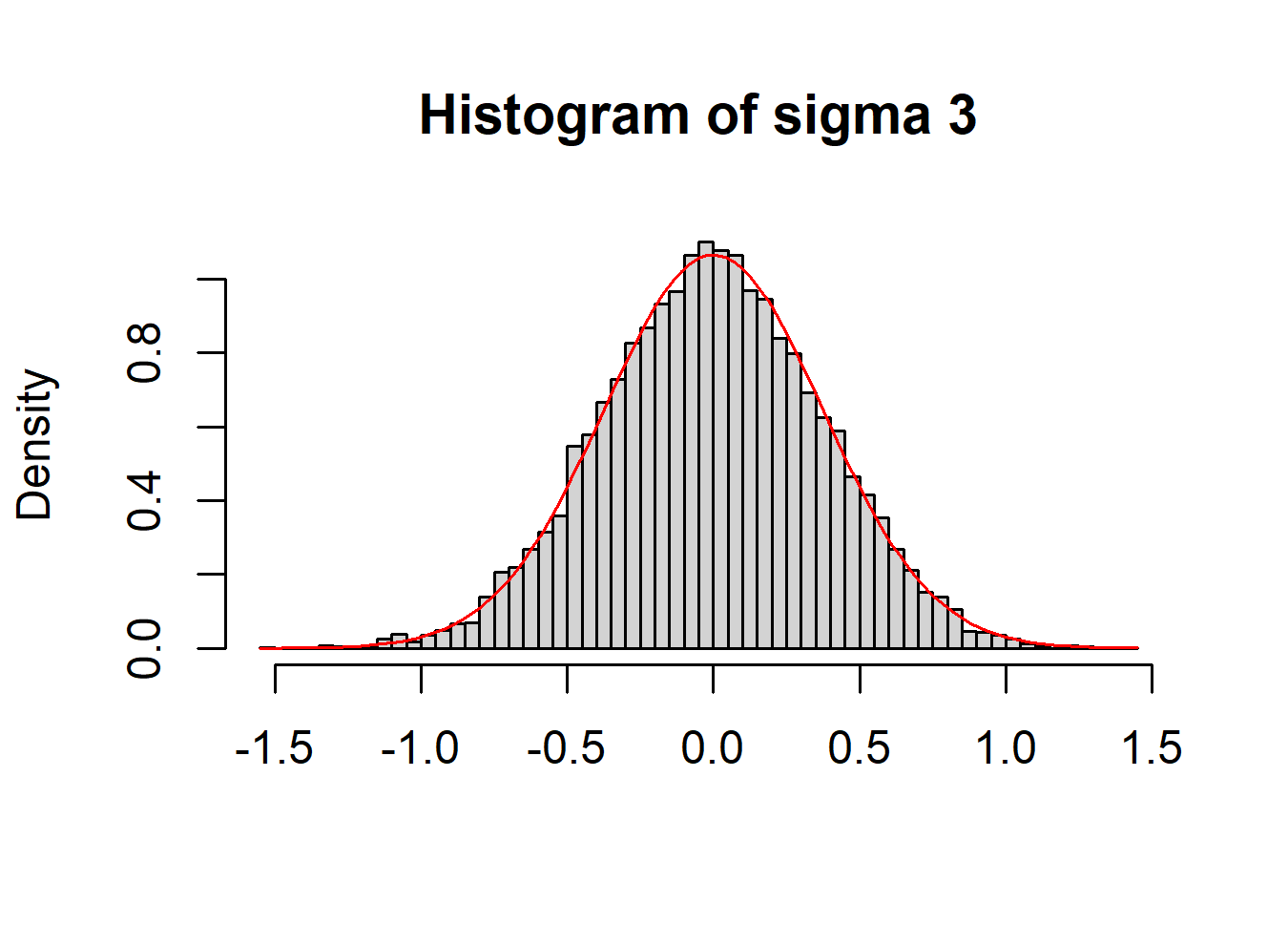}
    \includegraphics[width=0.3\columnwidth]{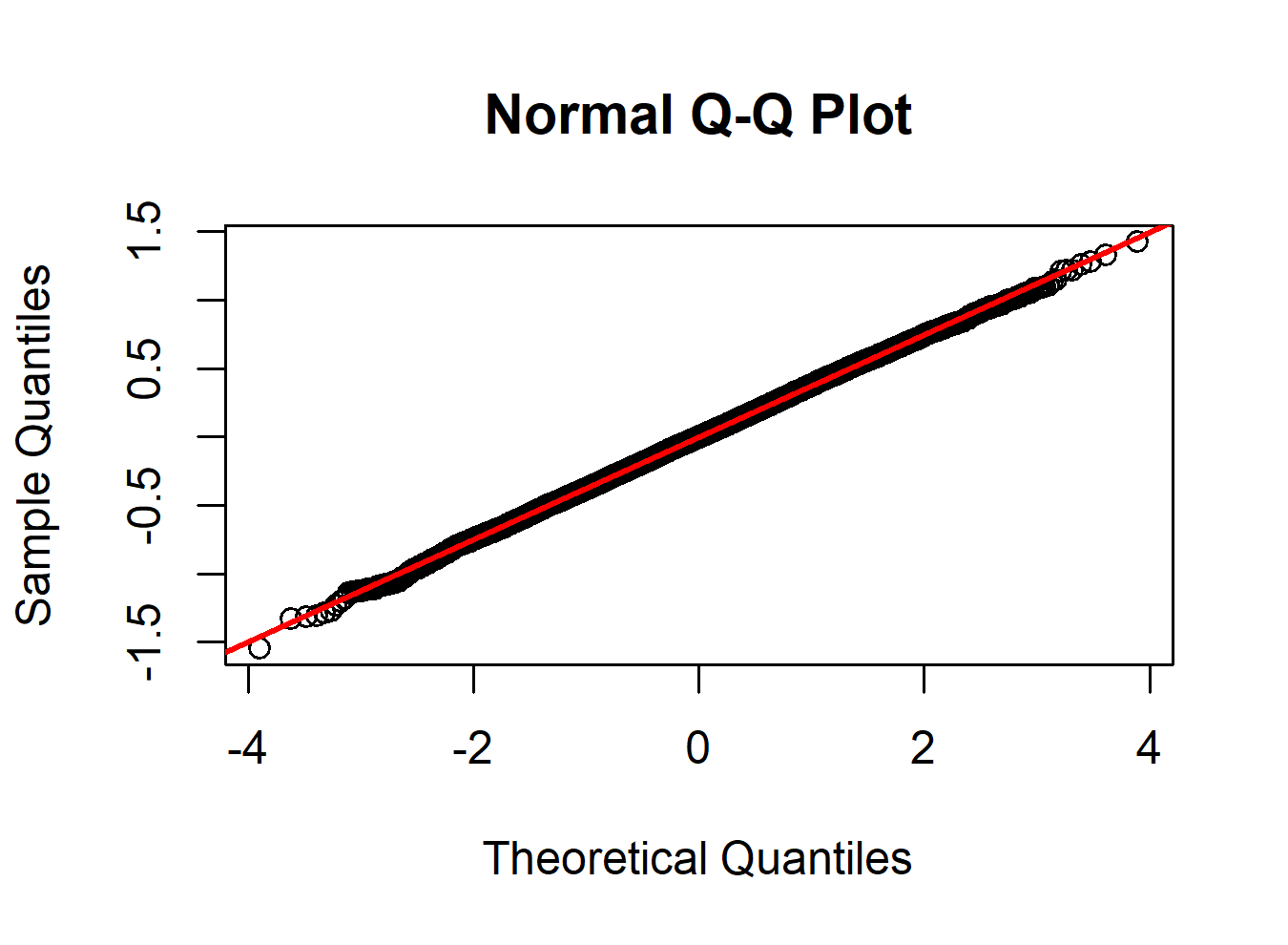}
    \includegraphics[width=0.3\columnwidth]{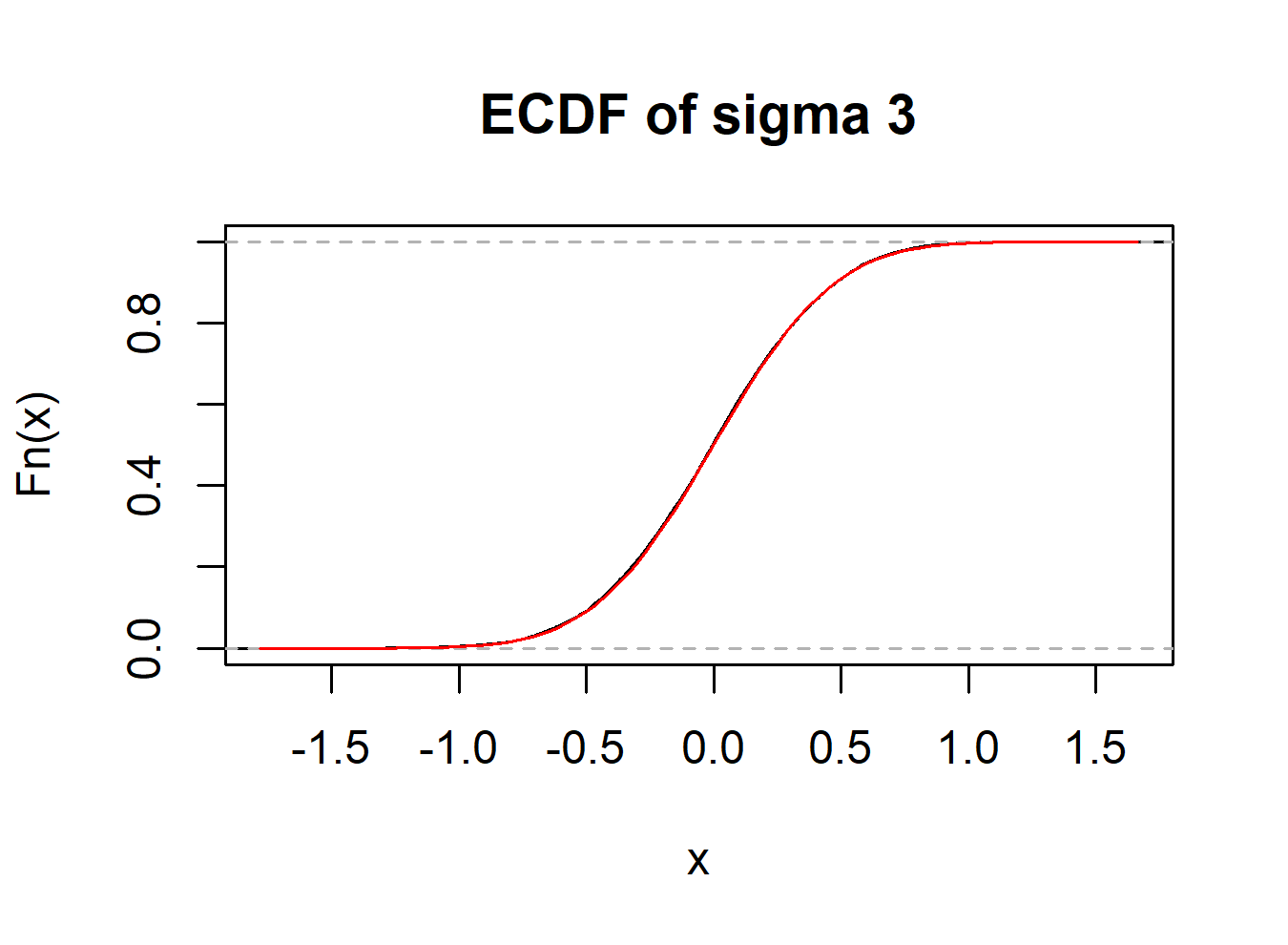}\vspace{2mm} \\
    \includegraphics[width=0.3\columnwidth]{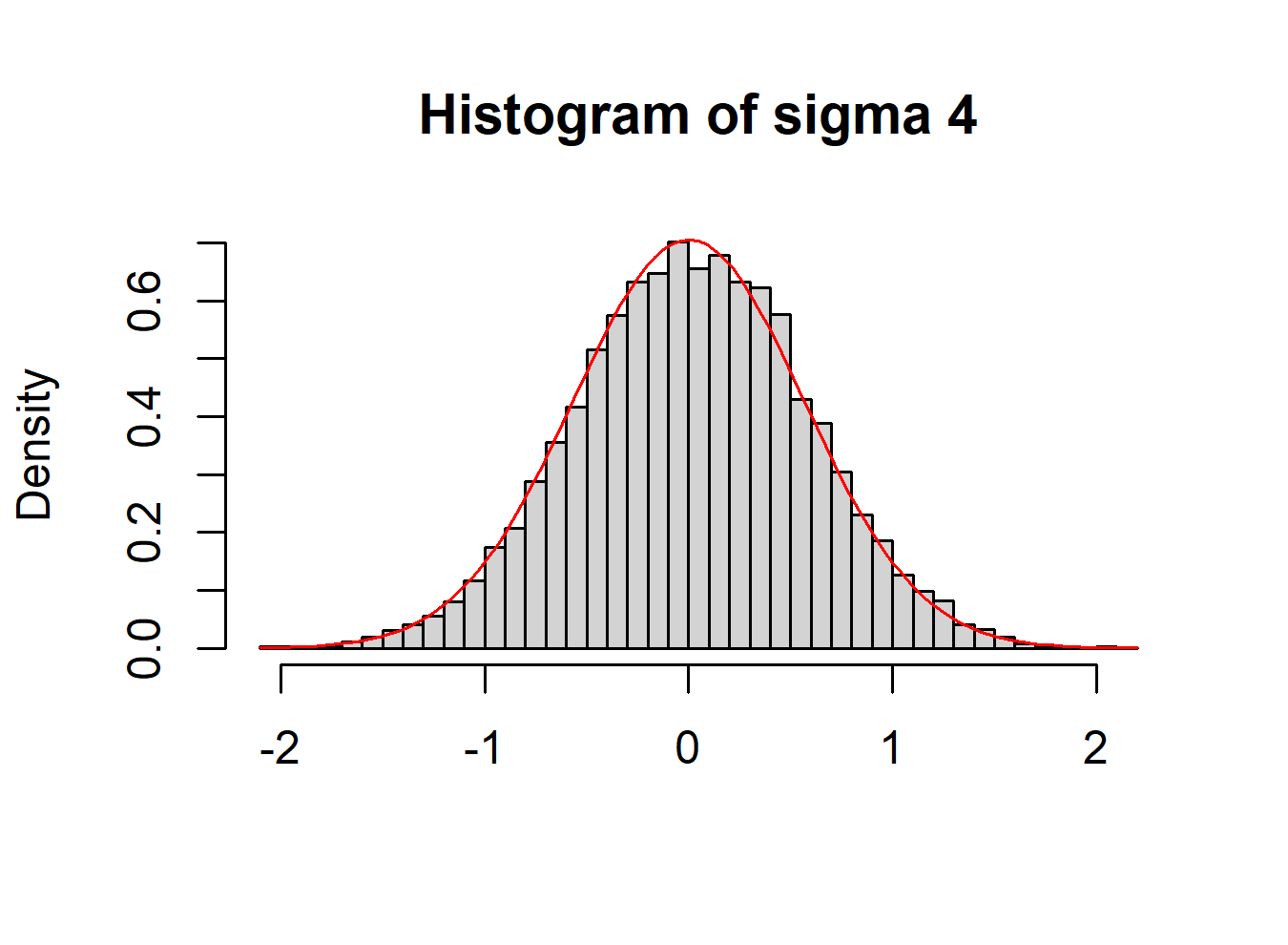}
    \includegraphics[width=0.3\columnwidth]{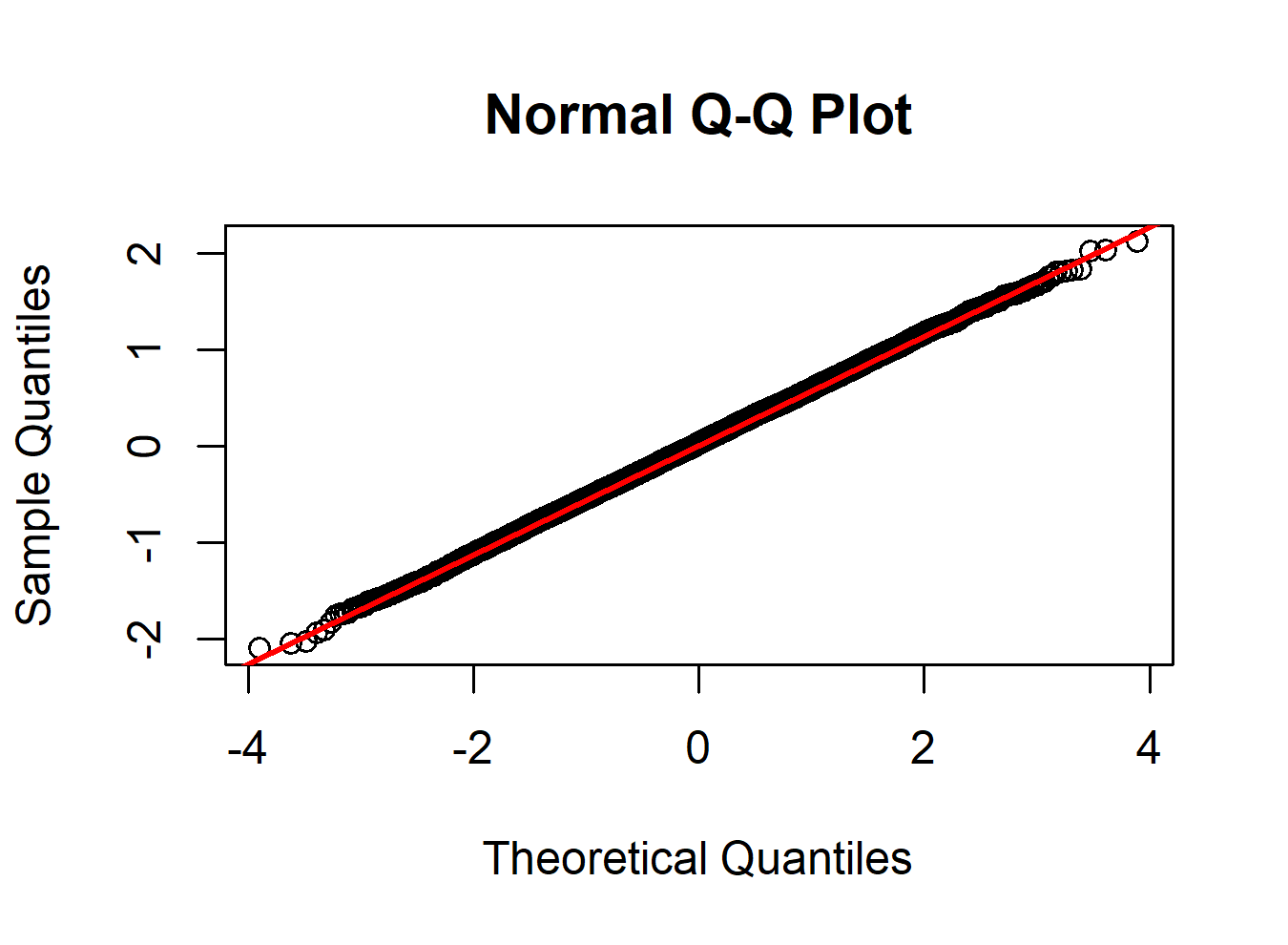}
    \includegraphics[width=0.3\columnwidth]{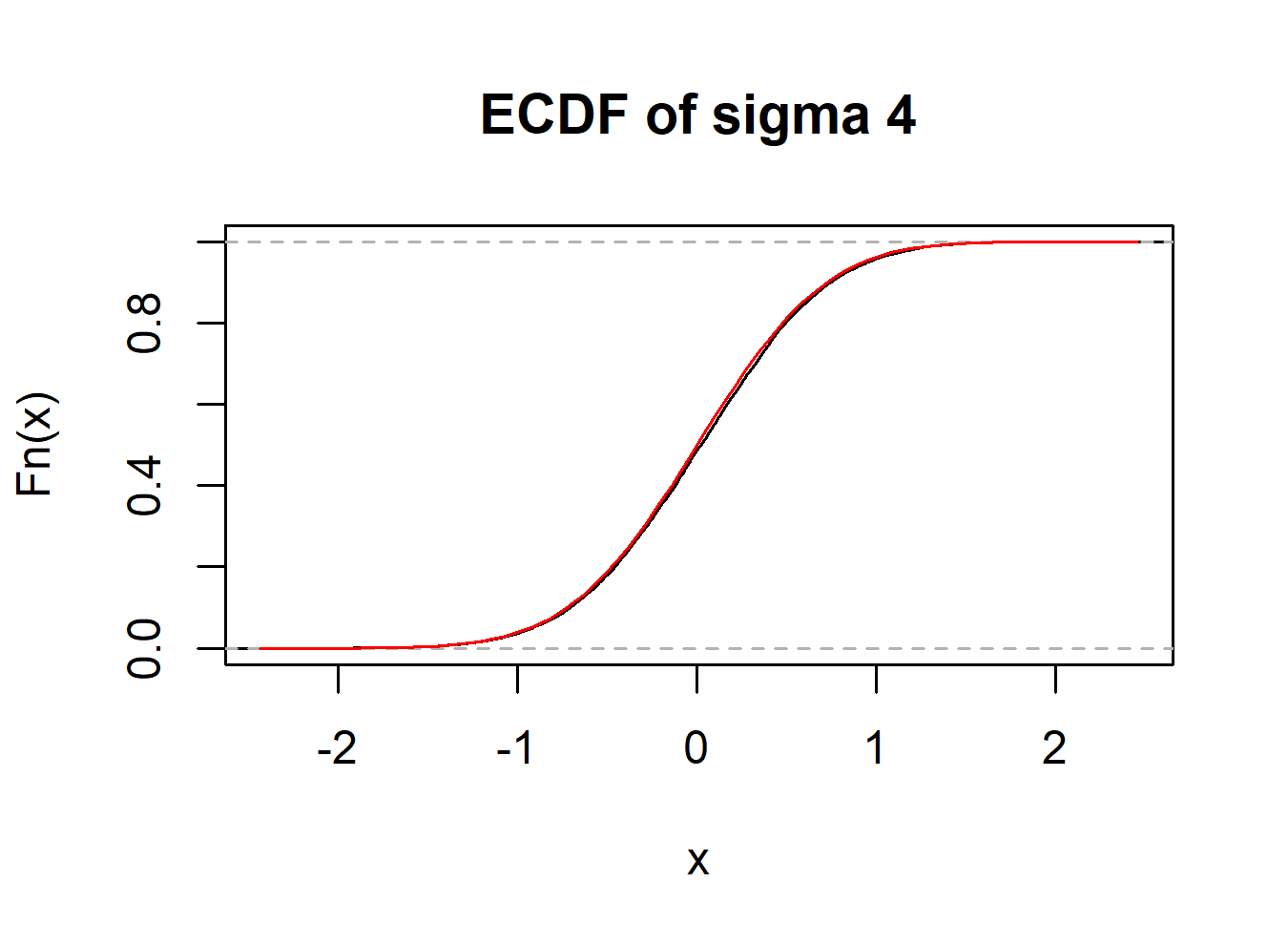}
    \caption{Histogram (left), Q-Q plot (middle) and empirical distribution function (right) of the components of $\sqrt{n}(\check{\bm{\sigma}}_{n}-{\bm{\sigma}}^{*})$ when $n=10^6$ in Simulation 2. The red lines are theoretical curves.}\label{2F}
\end{figure}
\newpage
\section{Proofs}\label{proof}
\setlength{\abovedisplayskip}{10pt}
\setlength{\belowdisplayskip}{10pt}
Set 
\begin{align*}
    \hat{\Sigma}_{j,n}=\sqrt{\frac{1}{nh_n}\sum_{i=1}^n (X_{j,t_i^n}-X_{j,t_{i-1}^n})^2}
\end{align*}
for $j=1,\ldots,J$. For each $j=1,\ldots,J$, we note that
\begin{align*}
    \bar{\bf H}_{j,n}(\sigma)=-\frac{n}{2}\log\sigma^2
    -\frac{n\hat{\Sigma}^2_{j,n}}{2\sigma^2}
\end{align*}
for $\sigma>0$.
\begin{lemma}\label{Sigmalemma}
Suppose that ${[\bf{A1}]}$ and ${[\bf{A2}]}$ hold. Then, under ${\rm{P}}_{\bm{z}}$, for each $j=1,\ldots,J$,
\begin{align}
    \hat{\Sigma}^2_{j,n}\xlongrightarrow{p}\sigma^{*2}_{z_j}
    \label{Sigmacons}
\end{align}
and
\begin{align}
    \sqrt{n}\Bigl(\hat{\Sigma}^2_{j,n}-\sigma^{*2}_{z_j}\Bigr)\xlongrightarrow{d}N\Bigl(0,2\sigma_{z_j}^{*4}\Bigr)\label{Sigmaasym}
\end{align}
as $n\longrightarrow\infty$.
\end{lemma}
\begin{proof}
    The proof is based on moment estimates and the martingale central limit theorem. See Appendix \ref{Lemma} for details.
\end{proof}
\begin{lemma}\label{Hbarlemma}
Suppose that ${[\bf{A1}]}$ and ${[\bf{A2}]}$ hold. Then, under ${\rm{P}}_{\bm{z}}$, for each $j=1,\ldots,J$,
\begin{align*}
    \sup_{\sigma\in[1/d,d]}\biggl|\frac{1}{n}\bar{\bf H}_{j,n}(\sigma)-\bar{\mathbb H}_{z_j}(\sigma)\biggr|&\xlongrightarrow{p}0,\\
    \sup_{\sigma\in[1/d,d]}\biggl|\frac{1}{n}\bar{\bf H}'_{j,n}(\sigma)-\bar{\mathbb H}'_{z_j}(\sigma)\biggr|&\xlongrightarrow{p}0
\end{align*}
and
\begin{align*}
    \sup_{\sigma\in[1/d,d]}\biggl|\frac{1}{n}\bar{\bf H}^{\prime\prime}_{j,n}(\sigma)-\bar{\mathbb H}^{\prime\prime}_{z_j}(\sigma)\biggr|\xlongrightarrow{p}0
\end{align*}
as $n\longrightarrow\infty$.
\end{lemma}
\begin{proof}
For each $j=1,\ldots,J$, since 
\begin{align*}
    \frac{1}{n}\bar{\bf H}_{j,n}(\sigma)-\bar{\mathbb H}_{z_j}(\sigma)&=-\frac{1}{2\sigma^2}\Bigl(\hat{\Sigma}^2_{j,n}-\sigma_{z_j}^{*2}\Bigr),\\
     \frac{1}{n}\bar{\bf H}'_{j,n}(\sigma)-\bar{\mathbb H}'_{z_j}(\sigma)&=\frac{1}{\sigma^3}\Bigl(\hat{\Sigma}^2_{j,n}-\sigma_{z_j}^{*2}\Bigr)
\end{align*}
and 
\begin{align*}
    \frac{1}{n}\bar{\bf H}^{\prime\prime}_{j,n}(\sigma)-\bar{\mathbb H}^{\prime\prime}_{z_j}(\sigma)=-\frac{3}{\sigma^4}\Bigl(\hat{\Sigma}^2_{j,n}-\sigma_{z_j}^{*2}\Bigr)
\end{align*}
for every $\sigma>0$, it holds from Lemma \ref{Sigmalemma} that
\begin{align*}
    \sup_{\sigma\in[1/d,d]}\biggl|\frac{1}{n}\bar{\bf H}_{j,n}(\sigma)-\bar{\mathbb H}_{z_j}(\sigma)\biggr|&\leq\frac{1}{2}\Biggl(\sup_{\sigma\in[1/d,d]}\frac{1}{\sigma^2}\Biggr) \Bigl|\hat{\Sigma}^2_{j,n}-\sigma_{z_j}^{*2}\Bigr|\xlongrightarrow{p}0,\\
    \sup_{\sigma\in[1/d,d]}\biggl|\frac{1}{n}\bar{\bf H}'_{j,n}(\sigma)-\bar{\mathbb H}'_{z_j}(\sigma)\biggr|&\leq\Biggl(\sup_{\sigma\in[1/d,d]}\frac{1}{\sigma^3}\Biggr)\Bigl|\hat{\Sigma}^2_{j,n}-\sigma_{z_j}^{*2}\Bigr|\xlongrightarrow{p}0
\end{align*}
and 
\begin{align*}
    \sup_{\sigma\in[1/d,d]}\biggl|\frac{1}{n}\bar{\bf H}^{\prime\prime}_{j,n}(\sigma)-\bar{\mathbb H}^{\prime\prime}_{z_j}(\sigma)\biggr|\leq 3\Biggl(\sup_{\sigma\in[1/d,d]}\frac{1}{\sigma^4}\Biggr) \Bigl|\hat{\Sigma}^2_{j,n}-\sigma_{z_j}^{*2}\Bigr|\xlongrightarrow{p}0\ 
\end{align*}
under ${\rm{P}}_{\bm{z}}$, which completes the proof.
\end{proof}
\begin{proposition}\label{unifprop}
Suppose that ${[\bf{A1}]}$ and ${[\bf{A2}]}$ hold. Then, under ${\rm{P}}_{\bm{z}}$, for each $G=1,\ldots,M$ and every ${\bm{\pi}}^{(G)}\in\Pi^{(G)}$,
\begin{align*}
    \sup_{{\bm{\sigma}}^{(G)}\in\overline{\Sigma}^{(G)}}\biggl|\frac{1}{n}{\bf{H}}^{(G)}_{n}({\bm{\sigma}}^{(G)};{\bm{\pi}}^{(G)})-\mathbb{H}_{\bm{z}}^{(G)}({\bm{\sigma}}^{(G)})\biggr|\xlongrightarrow{p}0
\end{align*}
as $n\longrightarrow\infty$.
\end{proposition}
\begin{proof}
Fix $G\in\{1,\ldots,M\}$ and ${\bm{\pi}}^{(G)}\in\Pi^{(G)}$ arbitrarily. Recall that
\begin{align*}
   \frac{1}{n}{\bf{H}}^{(G)}_{n}({\bm{\sigma}}^{(G)};{\bm{\pi}}^{(G)}) &=\sum_{j=1}^J\frac{1}{n}\log\Biggl(\sum_{g=1}^{G}\pi^{(G)}_g\exp\Bigl\{\bar{\bf H}_{j,n}(\sigma^{(G)}_g)\Bigr\}\Biggr)
\end{align*}
and 
\begin{align*}
    \mathbb{H}_{\bm{z}}^{(G)}({\bm{\sigma}}^{(G)})=\sum_{j=1}^J\max_{g\in\{1,\ldots,G\}}\bar{\mathbb H}_{z_j}(\sigma^{(G)}_g).
\end{align*}
For every ${\bm{\sigma}}^{(G)}\in\overline{\Sigma}^{(G)}$ and each $j=1,\ldots,J$, we have
\begin{align*}
    \sum_{g=1}^{G}\pi^{(G)}_g\exp\Bigl\{\bar{\bf H}_{j,n}(\sigma^{(G)}_g)\Bigr\}\geq
    c^{(G)}\exp\biggl\{\max_{g\in\{1,\ldots,G\}}\bar{\bf H}_{j,n}(\sigma^{(G)}_g)\biggr\},
\end{align*}
which yields 
\begin{align*}
    \log\Biggl(\sum_{g=1}^{G}\pi^{(G)}_g\exp\Bigl\{\bar{\bf H}_{j,n}(\sigma^{(G)}_g)\Bigr\}\Biggr)
    &\geq\log c^{(G)}+\max_{g\in\{1,\ldots,G\}}\bar{\bf H}_{j,n}(\sigma^{(G)}_g),
\end{align*}
where 
\begin{align*}
    c^{(G)}=\min_{g\in\{1,\ldots,G\}}\pi^{(G)}_g>0.
\end{align*}
On the other hand, for every ${\bm{\sigma}}^{(G)}\in\overline{\Sigma}^{(G)}$ and each $j=1,\ldots,J$, since 
\begin{align*}
   \sum_{g=1}^{G}\pi^{(G)}_g\exp\Bigl\{\bar{\bf H}_{j,n}(\sigma^{(G)}_g)\Bigr\}
   &\leq\exp\biggl\{
   \max_{g\in\{1,\ldots,G\}}
   \bar{\bf H}_{j,n}(\sigma^{(G)}_g)
   \biggr\}\sum_{g=1}^{G}\pi^{(G)}_g\\
   &=\exp\biggl\{\max_{g\in\{1,\ldots,G\}}\bar{\bf H}_{j,n}(\sigma^{(G)}_g)\biggr\},
\end{align*}
it holds that
\begin{align*}
    \log\Biggl(\sum_{g=1}^{G}\pi^{(G)}_g\exp\Bigl\{\bar{\bf H}_{j,n}(\sigma^{(G)}_g)\Bigr\}\Biggr)\leq\max_{g\in\{1,\ldots,G\}}\bar{\bf H}_{j,n}(\sigma^{(G)}_g).
\end{align*}
Consequently, for every ${\bm{\sigma}}^{(G)}\in\overline{\Sigma}^{(G)}$ and each $j=1,\ldots,J$, we obtain
\begin{align*}
    &\frac{1}{n}\log c^{(G)}+\max_{g\in\{1,\ldots,G\}}\biggl\{
    \frac{1}{n}\bar{\bf H}_{j,n}(\sigma^{(G)}_g)\biggr\}\\
    &\quad\leq\frac{1}{n}\log\Biggl(\sum_{g=1}^{G}\pi^{(G)}_g\exp\Bigl\{\bar{\bf H}_{j,n}(\sigma^{(G)}_g)\Bigr\}\Biggr)\leq
    \max_{g\in\{1,\ldots,G\}}\biggl\{
    \frac{1}{n}\bar{\bf H}_{j,n}(\sigma^{(G)}_g)\biggr\},
\end{align*}
and hence
\begin{align*}
    &\quad\ \Biggl|\frac{1}{n}\log\Biggl(\sum_{g=1}^{G}\pi^{(G)}_g\exp\Bigl\{\bar{\bf H}_{j,n}(\sigma^{(G)}_g)\Bigr\}\Biggr)-\max_{g\in\{1,\ldots,G\}}
    \bar{\mathbb H}_{z_j}(\sigma^{(G)}_g)\Biggr|\\
    &\leq\biggl|\max_{g\in\{1,\ldots,G\}}\biggl\{
    \frac{1}{n}\bar{\bf H}_{j,n}(\sigma^{(G)}_g)\biggr\}-\max_{g\in\{1,\ldots,G\}}
    \bar{\mathbb H}_{z_j}(\sigma^{(G)}_g)\biggr|+\frac{1}{n}\bigl|\log c^{(G)}\bigr|\\
    &\leq\max_{g\in\{1,\ldots,G\}}\biggl|
    \frac{1}{n}\bar{\bf H}_{j,n}(\sigma^{(G)}_g)-\bar{\mathbb H}_{z_j}(\sigma^{(G)}_g)\biggr|+\frac{1}{n}\bigl|\log c^{(G)}\bigr|\\
    &\leq\sup_{\sigma\in[1/d,d]}\biggl|
    \frac{1}{n}\bar{\bf H}_{j,n}(\sigma)-\bar{\mathbb H}_{z_j}(\sigma)\biggr|+\frac{1}{n}\bigl|\log c^{(G)}\bigr|.
\end{align*}
Therefore, Lemma \ref{Hbarlemma} implies 
\begin{align*}
    &\quad\ \sup_{{\bm{\sigma}}^{(G)}\in\overline{\Sigma}^{(G)}}\biggl|\frac{1}{n}{\bf{H}}^{(G)}_{n}({\bm{\sigma}}^{(G)};{\bm{\pi}}^{(G)})-\mathbb{H}_{\bm{z}}^{(G)}({\bm{\sigma}}^{(G)})\biggr|\\
    &\leq\sum_{j=1}^J\sup_{{\bm{\sigma}}^{(G)}\in\overline{\Sigma}^{(G)}}
    \Biggl|\frac{1}{n}\log\Biggl(\sum_{g=1}^{G}\pi^{(G)}_g\exp\Bigl\{\bar{\bf H}_{j,n}(\sigma^{(G)}_g)\Bigr\}\Biggr)-\max_{g\in\{1,\ldots,G\}}
    \bar{\mathbb H}_{z_j}(\sigma^{(G)}_g)\Biggr|\\
    &\leq \sum_{j=1}^J\sup_{\sigma\in[1/d,d]}\biggl|
    \frac{1}{n}\bar{\bf H}_{j,n}(\sigma)-\bar{\mathbb H}_{z_j}(\sigma)\biggr|+\frac{J}{n}\bigl|\log c^{(G)}\bigr|\xlongrightarrow{p}0
\end{align*}
under ${\rm{P}}_{\bm{z}}$, which completes the proof.
\end{proof}
\begin{corollary}\label{supGprob}
Suppose that ${[\bf{A1}]}$ and ${[\bf{A2}]}$ hold. Then, under ${\rm{P}}_{\bm{z}}$, for each $G=1,\ldots,M$ and every ${\bm{\pi}}^{(G)}\in\Pi^{(G)}$,
\begin{align*}
    \sup_{{\bm{\sigma}}^{(G)}\in\overline{\Sigma}^{(G)}}\frac{1}{n}{\bf{H}}^{(G)}_{n}({\bm{\sigma}}^{(G)};{\bm{\pi}}^{(G)})\xlongrightarrow{p}\sup_{{\bm{\sigma}}^{(G)}\in\overline{\Sigma}^{(G)}}\mathbb{H}_{\bm{z}}^{(G)}({\bm{\sigma}}^{(G)})
\end{align*}
as $n\longrightarrow\infty$.
\end{corollary}
\begin{proof}
Fix $G\in\{1,\ldots,M\}$ and ${\bm{\pi}}^{(G)}\in\Pi^{(G)}$ arbitrarily. By Proposition \ref{unifprop}, we have
\begin{align*}
    &\quad\ \Biggl|\sup_{{\bm{\sigma}}^{(G)}\in\overline{\Sigma}^{(G)}}\frac{1}{n}{\bf{H}}^{(G)}_{n}({\bm{\sigma}}^{(G)};{\bm{\pi}}^{(G)})-\sup_{{\bm{\sigma}}^{(G)}\in\overline{\Sigma}^{(G)}}\mathbb{H}_{\bm{z}}^{(G)}({\bm{\sigma}}^{(G)})\Biggr|\\
    &\leq \sup_{{\bm{\sigma}}^{(G)}\in\overline{\Sigma}^{(G)}}\biggl|\frac{1}{n}{\bf{H}}^{(G)}_{n}({\bm{\sigma}}^{(G)};{\bm{\pi}}^{(G)})-\mathbb{H}_{\bm{z}}^{(G)}({\bm{\sigma}}^{(G)})\biggr|\xlongrightarrow{p}0
\end{align*}
under ${\rm{P}}_{\bm{z}}$, which completes the proof.
\end{proof}
\begin{lemma}\label{HbarOp1}
Suppose that ${[\bf{A1}]}$ and ${[\bf{A2}]}$ hold. Then, under ${\rm{P}}_{\bm{z}}$, for each $j=1,\ldots,J$,
\begin{align*}
    {\bar{\bf{H}}}_{j,n}(\hat{\Sigma}_{j,n})-{\bar{\bf{H}}}_{j,n}(\sigma^{*}_{z_j})=O_p(1)
\end{align*}
as $n\longrightarrow\infty$.
\end{lemma}
\begin{proof}
Fix $j\in\{1,\ldots,J\}$ arbitrarily. By Taylor's theorem, we have 
\begin{align*} 
    {\bar{\bf{H}}}_{j,n}(\sigma^{*}_{z_j}) 
    &={\bar{\bf{H}}}_{j,n}(\hat{\Sigma}_{j,n})+{\bar{\bf{H}}}^{\prime}_{j,n}(\hat{\Sigma}_{j,n})\bigl(\sigma^{*}_{z_j}-\hat{\Sigma}_{j,n}\bigr)+\bar{\bf{R}}_{j,n},
\end{align*} 
where 
\begin{align*}
    \bar{\bf{R}}_{j,n}=\left(\int_0^1(1-\lambda){\bar{\bf{H}}}^{\prime\prime}_{j,n}\bigl(
    \hat{\Sigma}_{j,n}+\lambda(\sigma^{*}_{z_j}-\hat{\Sigma}_{j,n})\bigr)d\lambda\right)\bigl(\sigma^{*}_{z_j}-\hat{\Sigma}_{j,n}\bigr)^2.
\end{align*}
Since ${\bar{\bf{H}}}_{j,n}(\sigma)$ attains its maximum at $\sigma=\hat{\Sigma}_{j,n}$, we have
\begin{align*}
    {\bar{\bf{H}}}^{\prime}_{j,n}(\hat{\Sigma}_{j,n})=0,
\end{align*}
so that it suffices to show that 
\begin{align}
    \bar{\bf{R}}_{j,n}=O_p(1) \label{Rjn}
\end{align}
under ${\rm{P}}_{\bm{z}}$. Noting that
\begin{align*}
    \sup_{\sigma\in[1/d,d]}\biggl|\frac{1}{n}{\bar{\bf{H}}}^{\prime\prime}_{j,n}(\sigma)\biggr|\leq\sup_{\sigma\in[1/d,d]}\biggl|\frac{1}{n}{\bar{\bf{H}}}^{\prime\prime}_{j,n}(\sigma)-\bar{\mathbb{H}}^{\prime\prime}_{z_j}(\sigma)\biggr|+\sup_{\sigma\in[1/d,d]}\Bigl|\bar{\mathbb{H}}^{\prime\prime}_{z_j}(\sigma)\Bigr|,
\end{align*}
we see from Lemma \ref{Hbarlemma} that
\begin{align*}
    \sup_{\sigma\in[1/d,d]}\biggl|\frac{1}{n}{\bar{\bf{H}}}^{\prime\prime}_{j,n}(\sigma)\biggr|=O_p(1)
\end{align*}
under ${\rm{P}}_{\bm{z}}$. Set
\begin{align*}
    \varepsilon_j=\min\biggl\{ \sigma^*_{z_j}-\frac{1}{d},\, d-\sigma^*_{z_j} \biggr\},\quad A_{j,n}=\biggl\{\frac{1}{d}\leq\hat{\Sigma}_{j,n}\leq d\biggr\}.
\end{align*}
Since $1/d<\sigma^*_{z_j}<d$, we note that $\varepsilon_j>0$. On 
\begin{align*}
    \Bigl\{\bigl|\hat{\Sigma}_{j,n}-\sigma_{z_j}^*\bigr|\leq \varepsilon_j\Bigr\},
\end{align*}
we have 
\begin{align*} 
    \frac{1}{d}\leq \sigma_{z_j}^*-\varepsilon_j\leq \hat{\Sigma}_{j,n}\leq \sigma_{z_j}^*+\varepsilon_j\leq d,
\end{align*}
which yields
\begin{align*}
    \Bigl\{\bigl|\hat{\Sigma}_{j,n}-\sigma_{z_j}^*\bigr|\leq \varepsilon_j\Bigr\}\subset A_{j,n}.
\end{align*}
Consequently, it holds from Lemma \ref{Sigmalemma} that
\begin{align} 
    {\rm{P}}_{\bm{z}}\Bigl(A_{j,n}\Bigr)\longrightarrow 1\label{Ajn}
\end{align} 
as $n\longrightarrow\infty$. On $A_{j,n}$, for every $\lambda\in[0,1]$, we have 
\begin{align*}
    \frac{1}{d}\leq\hat{\Sigma}_{j,n}+\lambda\bigl({\sigma}^*_{z_j}-\hat{\Sigma}_{j,n}\bigr)\leq d
\end{align*} 
and
\begin{align*} 
\begin{split} 
    \bigl|\bar{\bf{R}}_{j,n}\bigr| 
    &\leq \int_0^1(1-\lambda)\biggl|\frac{1}{n}{\bar{\bf{H}}}^{\prime\prime}_{j,n} \bigl(
    \hat{\Sigma}_{j,n}+\lambda({\sigma}^*_{z_j}-\hat{\Sigma}_{j,n})\bigr) \biggr|d\lambda\, n\bigl(\sigma^{*}_{z_j}-\hat{\Sigma}_{j,n}\bigr)^2\\ 
    &\leq\frac{1}{2}\sup_{\sigma\in[1/d,d]}\, \biggl|\frac{1}{n}{\bar{\bf{H}}}^{\prime\prime}_{j,n}(\sigma)\biggr|\Bigl\{\sqrt{n}\bigl(\hat{\Sigma}_{j,n}-\sigma^*_{z_j}\bigr)\Bigr\}^2. 
\end{split} 
\end{align*} 
Moreover, by Lemma \ref{Sigmalemma} and the delta method,
\begin{align*}
    \sqrt{n}\bigl(\hat{\Sigma}_{j,n}-\sigma^*_{z_j}\bigr)=O_p(1),
\end{align*}
so that it follows from Lemma \ref{Hbarlemma} that
\begin{align} 
    \frac{1}{2}\sup_{\sigma\in[1/d,d]}\, \biggl|\frac{1}{n}{\bar{\bf{H}}}^{\prime\prime}_{j,n}(\sigma)\biggr|\Bigl\{\sqrt{n}\bigl(\hat{\Sigma}_{j,n}-\sigma^*_{z_j}\bigr)\Bigr\}^2=O_p(1) \label{supeq}
\end{align} 
under ${\rm{P}}_{\bm{z}}$. 
Since for every $K>0$, 
\begin{align*} 
    &\quad\ {\rm{P}}_{\bm{z}}\biggl(\bigl|\bar{\bf{R}}_{j,n}\bigr|>K\biggr)\\
    &={\rm{P}}_{\bm{z}}\biggl(\Bigl\{\bigl|\bar{\bf{R}}_{j,n}\bigr|>K\Bigr\}\cap A_{j,n}\biggr)+{\rm{P}}_{\bm{z}}\biggl(\Bigl\{\bigl|\bar{\bf{R}}_{j,n}\bigr|>K\Bigr\}\cap A_{j,n}^c\biggr)\\ 
    &\leq {\rm{P}}_{\bm{z}}\Biggl(\frac{1}{2}\sup_{\sigma\in[1/d,d]}\, \biggl|\frac{1}{n}{\bar{\bf{H}}}^{\prime\prime}_{j,n}(\sigma)\biggr|\Bigl\{\sqrt{n}\bigl(\hat{\Sigma}_{j,n}-\sigma^*_{z_j}\bigr)\Bigr\}^2>K\Biggr)+{\rm{P}}_{\bm{z}}\Bigl(A_{j,n}^c\Bigr),
\end{align*}
it holds from \eqref{Ajn} and \eqref{supeq} that
\begin{align*} 
    \lim_{K\to\infty} \limsup_{n\to\infty} {\rm{P}}_{\bm{z}}\Bigl(\bigl|\bar{\bf{R}}_{j,n}\bigr|>K\Bigr)=0,
\end{align*} 
which yields \eqref{Rjn}. This completes the proof.
\end{proof}
\begin{proposition}\label{supOp}
Suppose that ${[\bf{A1}]}$ and ${[\bf{A2}]}$ hold. Then, under ${\rm{P}}_{\bm{z}}$, for every $K,L\in\{G^*,\ldots,M\}$, ${\bm{\pi}}^{(K)}\in\Pi^{(K)}$ and ${\bm{\pi}}^{(L)}\in\Pi^{(L)}$,
\begin{align*}
    \sup_{{\bm{\sigma}}^{(K)}\in\overline{\Sigma}^{(K)}}{\bf{H}}^{(K)}_{n}({\bm{\sigma}}^{(K)};{\bm{\pi}}^{(K)})-\sup_{{\bm{\sigma}}^{(L)}\in\overline{\Sigma}^{(L)}}{\bf{H}}^{(L)}_{n}({\bm{\sigma}}^{(L)};{\bm{\pi}}^{(L)})=O_p(1)
\end{align*}
as $n\longrightarrow\infty$.
\end{proposition}
\begin{proof}
It is enough to show that, for each $G=G^*,\ldots,M$ and every ${\bm{\pi}}^{(G)}\in\Pi^{(G)}$, under ${\rm{P}}_{\bm{z}}$, 
\begin{align}
    \sup_{{\bm{\sigma}}^{(G)}\in\overline{\Sigma}^{(G)}}{\bf{H}}^{(G)}_{n}({\bm{\sigma}}^{(G)};{\bm{\pi}}^{(G)})=\sum_{j=1}^J\bar{\bf{H}}_{j,n}(\sigma_{z_j}^*)+O_p(1) \label{supOp1}
\end{align}
holds. Here, we note that the leading term does not depend on the number of clusters $G$. Indeed, if \eqref{supOp1} holds, then, for every $K,L\in\{G^*,\ldots,M\}$, ${\bm{\pi}}^{(K)}\in\Pi^{(K)}$ and ${\bm{\pi}}^{(L)}\in\Pi^{(L)}$,
\begin{align*}
    &\quad\ \sup_{{\bm{\sigma}}^{(K)}\in\overline{\Sigma}^{(K)}}{\bf{H}}^{(K)}_{n}({\bm{\sigma}}^{(K)};{\bm{\pi}}^{(K)})-\sup_{{\bm{\sigma}}^{(L)}\in\overline{\Sigma}^{(L)}}{\bf{H}}^{(L)}_{n}({\bm{\sigma}}^{(L)};{\bm{\pi}}^{(L)})\\
    &=\Biggl\{\sup_{{\bm{\sigma}}^{(K)}\in\overline{\Sigma}^{(K)}}{\bf{H}}^{(K)}_{n}({\bm{\sigma}}^{(K)};{\bm{\pi}}^{(K)})-\sum_{j=1}^J\bar{\bf{H}}_{j,n}(\sigma_{z_j}^*)\Biggr\}\\
    &\qquad\qquad-\Biggl\{\sup_{{\bm{\sigma}}^{(L)}\in\overline{\Sigma}^{(L)}}{\bf{H}}^{(L)}_{n}({\bm{\sigma}}^{(L)};{\bm{\pi}}^{(L)})-\sum_{j=1}^J\bar{\bf{H}}_{j,n}(\sigma_{z_j}^*)\Biggr\}=O_p(1)
\end{align*}
under ${\rm{P}}_{\bm{z}}$. It remains to prove \eqref{supOp1}. Fix $G\in\{G^*,\ldots,M\}$ and ${\bm{\pi}}^{(G)}\in\Pi^{(G)}$ arbitrarily. For every ${\bm{\sigma}}^{(G)}\in\overline{\Sigma}^{(G)}$, in a similar way to the proof of Proposition \ref{unifprop}, we have
\begin{align*}
    {\bf{H}}^{(G)}_n({\bm{\sigma}}^{(G)};{\bm{\pi}}^{(G)})
    &\leq\sum_{j=1}^J\max_{g\in\{1,\ldots,G\}}\bar{\bf H}_{j,n}(\sigma^{(G)}_g).
\end{align*}
Moreover, for each $j=1,\ldots,J$, since
\begin{align*}
    \bar{\bf{H}}_{j,n}(\sigma)\leq\bar{\bf{H}}_{j,n}(\hat{\Sigma}_{j,n})
\end{align*}
for every $\sigma>0$, we obtain
\begin{align*}
    \max_{g\in\{1,\ldots,G\}}\bar{\bf{H}}_{j,n}(\sigma^{(G)}_g)\leq\bar{\bf{H}}_{j,n}(\hat{\Sigma}_{j,n})
\end{align*}
for every ${\bm{\sigma}}^{(G)}\in\overline{\Sigma}^{(G)}$. Therefore, for every ${\bm{\sigma}}^{(G)}\in\overline{\Sigma}^{(G)}$, it holds that
\begin{align*}
    {\bf{H}}^{(G)}_n({\bm{\sigma}}^{(G)};{\bm{\pi}}^{(G)})
    &\leq\sum_{j=1}^J\bar{\bf{H}}_{j,n}(\hat{\Sigma}_{j,n}),
\end{align*}
which implies
\begin{align}
    \sup_{{\bm{\sigma}}^{(G)}\in\overline{\Sigma}^{(G)}}{\bf{H}}^{(G)}_n({\bm{\sigma}}^{(G)};{\bm{\pi}}^{(G)})
    \leq\sum_{j=1}^J\bar{\bf{H}}_{j,n}(\hat{\Sigma}_{j,n}). \label{upper}
\end{align}
On the other hand, noting that $z_j\in\{1,\ldots,G^*\}$ and $G^*\leq G$, for every ${\bm{\sigma}}^{(G)}\in\overline{\Sigma}^{(G)}$, we see
\begin{align*}
    {\bf{H}}^{(G)}_{n}({\bm{\sigma}}^{(G)};{\bm{\pi}}^{(G)}) &=\sum_{j=1}^J\log\Biggl(\sum_{g=1}^{G}\pi^{(G)}_g\exp\Bigl\{\bar{\bf H}_{j,n}(\sigma^{(G)}_g)\Bigr\}\Biggr)\\
    &\geq \sum_{j=1}^J\log\Biggl(\pi^{(G)}_{z_j}\exp\Bigl\{\bar{\bf{H}}_{j,n}(\sigma^{(G)}_{z_j})\Bigr\}\Biggr)\\
    &=\sum_{j=1}^J\log\pi^{(G)}_{z_j}+\sum_{j=1}^J{\bar{\bf{H}}}_{j,n}(\sigma^{(G)}_{z_j}).
\end{align*}
Define $\tilde{\bm{\sigma}}^{(G)*}$ by
\begin{align*}
    \tilde{\bm{\sigma}}^{(G)*}
    =\Bigl(\tilde{\sigma}^{(G)*}_{1},\ldots,
    \tilde{\sigma}^{(G)*}_{G}\Bigr)^{\top},
\end{align*}
where $\tilde{\bm{\sigma}}^{(G^*)*}={\bm{\sigma}}^*$, and if $G>G^*$,
\begin{align*}
    \tilde{\sigma}^{(G)*}_{g}=\begin{cases}
    \displaystyle \sigma^*_g & \bigl(g\in\{1,\ldots,G^*\}\bigr),\\
    \displaystyle\sigma^*_{G^*}+\dfrac{d-\sigma^*_{G^*}}{G-G^*+1}(g-G^*) 
    & \bigl(g\in\{G^*+1,\ldots,G\}\bigr)
    \end{cases}
\end{align*}
for $g=1,\ldots,G$. By construction, we have
\begin{align*}
    \frac{1}{d}<\tilde{\sigma}^{(G)*}_{1}<\cdots<\tilde{\sigma}^{(G)*}_{G}<d,
\end{align*}
which yields $\tilde{\bm{\sigma}}^{(G)*}\in\Sigma^{(G)}$. Noting that for each $j=1,\ldots,J$, $\tilde{\sigma}^{(G)*}_{z_j}=\sigma^*_{z_j}$ since $z_j\in\{1,\ldots,G^*\}$, we have
\begin{align*}
     \sup_{{\bm{\sigma}}^{(G)}\in\overline{\Sigma}^{(G)}}{\bf{H}}^{(G)}_{n}({\bm{\sigma}}^{(G)};{\bm{\pi}}^{(G)})
     &\geq{\bf{H}}^{(G)}_{n}(\tilde{\bm{\sigma}}^{(G)*};{\bm{\pi}}^{(G)})\\
     &\geq\sum_{j=1}^J\log\pi^{(G)}_{z_j}+\sum_{j=1}^J{\bar{\bf{H}}}_{j,n}(\tilde{\sigma}^{(G)*}_{z_j})\\
     &=\sum_{j=1}^J\log\pi^{(G)}_{z_j}+\sum_{j=1}^J{\bar{\bf{H}}}_{j,n}(\sigma^*_{z_j}).
\end{align*}
Combining this with \eqref{upper}, we obtain
\begin{align*}
    \sum_{j=1}^J\log\pi^{(G)}_{z_j}
    &\leq\sup_{{\bm{\sigma}}^{(G)}\in\overline{\Sigma}^{(G)}}{\bf{H}}^{(G)}_{n}({\bm{\sigma}}^{(G)};{\bm{\pi}}^{(G)})-\sum_{j=1}^J\bar{\bf{H}}_{j,n}(\sigma_{z_j}^*)\\
    &\leq\sum_{j=1}^J\Bigl\{{\bar{\bf{H}}}_{j,n}(\hat{\Sigma}_{j,n})-{\bar{\bf{H}}}_{j,n}(\sigma^{*}_{z_j})\Bigr\}.
\end{align*}
Consequently, Lemma \ref{HbarOp1} implies
\begin{align*}
    \sup_{{\bm{\sigma}}^{(G)}\in\overline{\Sigma}^{(G)}}{\bf{H}}^{(G)}_{n}({\bm{\sigma}}^{(G)};{\bm{\pi}}^{(G)})-\sum_{j=1}^J\bar{\bf{H}}_{j,n}(\sigma_{z_j}^*)=O_p(1)
\end{align*}
under ${\rm{P}}_{\bm{z}}$, which yields \eqref{supOp1}. This completes the proof.
\end{proof}
\begin{lemma}\label{maxlemma}
Let $G\in\{1,\ldots,G^*-1\}$. For every ${\bm{\sigma}}^{(G)}\in \overline{\Sigma}^{(G)}$, there exists $j\in\{1,\ldots,J\}$ such that
\begin{align*}
    \max_{g\in\{1,\ldots,G\}}\bar{\mathbb{H}}_{z_j}(\sigma^{(G)}_g)\neq \bar{\mathbb{H}}_{z_j}(\sigma^{*}_{z_j}). 
\end{align*}
\end{lemma}
\begin{proof}
We prove this by contradiction. Fix $G\in\{1,\ldots,G^*-1\}$ arbitrarily. Suppose that there exists ${\bm{\sigma}}^{(G)}\in\overline{\Sigma}^{(G)}$ such that, for each $j\in\{1,\ldots,J\}$,
\begin{align*}
    \max_{g\in\{1,\ldots,G\}}\bar{\mathbb{H}}_{z_j}(\sigma^{(G)}_g)=\bar{\mathbb{H}}_{z_j}(\sigma^{*}_{z_j}).
\end{align*}
Since 
\begin{align*}
    S_g=\Bigl\{j\in\{1,\ldots,J\}\ \big|\ z_j=g\Bigr\}\neq \emptyset\quad (g=1,\ldots,G^*),
\end{align*}
for each $g=1,\ldots,G^*$, there exists some $j\in\{1,\ldots,J\}$ such that $z_j=g$. Therefore, the above assumption implies that
\begin{align}
\begin{split}
    \max_{g\in\{1,\ldots,G\}}\bar{\mathbb{H}}_{1}(\sigma^{(G)}_g)&=\bar{\mathbb{H}}_{1}(\sigma^{*}_{1}),\\
    &\vdots\\
    \max_{g\in\{1,\ldots,G\}}\bar{\mathbb{H}}_{G}(\sigma^{(G)}_g)&=\bar{\mathbb{H}}_{G}(\sigma^{*}_{G}),\\
    &\vdots\\
    \max_{g\in\{1,\ldots,G\}}\bar{\mathbb{H}}_{G^*}(\sigma^{(G)}_g)&=\bar{\mathbb{H}}_{G^*}(\sigma^{*}_{G^*}).
\end{split}\label{maxeq}
\end{align}
Since $\bar{\mathbb{H}}_{1}(\sigma)$ attains its maximum only at $\sigma=\sigma^{*}_{1}$, it follows from \eqref{maxeq} that $\sigma^{*}_{1}$ must be contained among the components of ${\bm{\sigma}}^{(G)}$. Similarly, we see from \eqref{maxeq} that ${\bm{\sigma}}^{(G)}$ must contain all of $\sigma_1^*,\ldots,\sigma_{G^*}^*$. However, this is impossible because ${\bm{\sigma}}^{(G)}$ has only $G\, (<G^*)$ components whereas $\sigma_1^*,\ldots,\sigma_{G^*}^*$ are pairwise distinct. Hence, this contradicts the assumption. This completes the proof.
\end{proof}
\begin{proposition}\label{Hlemma}
For each $G=1,\ldots,G^*-1$,
\begin{align*}
    \sup_{{\bm{\sigma}}^{(G)}\in\overline{\Sigma}^{(G)}}\mathbb{H}_{\bm{z}}^{(G)}({\bm{\sigma}}^{(G)})<\mathbb{H}_{\bm{z}}^{(G^*)}({\bm{\sigma}}^{*}).
\end{align*}
\end{proposition}
\begin{proof}
Fix an arbitrary $G\in\{1,\ldots,G^*-1\}$. For each $j=1,\ldots,J$, $\bar{\mathbb{H}}_{z_j}(\sigma)$ attains its unique maximum at $\sigma=\sigma^{*}_{z_j}$, so that
\begin{align*}
    \bar{\mathbb{H}}_{z_j}(\sigma)\leq\bar{\mathbb{H}}_{z_j}(\sigma_{z_j}^{*})
\end{align*}
for every $\sigma>0$. Hence, for every ${\bm{\sigma}}^{(G)}\in\overline{\Sigma}^{(G)}$, we have
\begin{align*}
    \mathbb{H}_{\bm{z}}^{(G)}\bigl({\bm{\sigma}}^{(G)}\bigr)
    &=\sum_{j=1}^J\max_{g\in\{1,\ldots,G\}}\bar{\mathbb{H}}_{z_j}(\sigma^{(G)}_g)\\
    &\leq\sum_{j=1}^J \bar{\mathbb{H}}_{z_j}(\sigma^{*}_{z_j})
    =\mathbb{H}_{\bm{z}}^{(G^*)}({\bm{\sigma}}^{*}).
\end{align*}
By Lemma \ref{maxlemma}, for every ${\bm{\sigma}}^{(G)}\in\overline{\Sigma}^{(G)}$, there exists $j\in\{1,\ldots,J\}$ such that
\begin{align*}
    \max_{g\in\{1,\ldots,G\}}\bar{\mathbb{H}}_{z_j}(\sigma^{(G)}_g)\neq \bar{\mathbb{H}}_{z_j}(\sigma^{*}_{z_j}), 
\end{align*}
so that 
\begin{align*}
    \mathbb{H}_{\bm{z}}^{(G)}\bigl({\bm{\sigma}}^{(G)}\bigr)<\mathbb{H}_{\bm{z}}^{(G^*)}({\bm{\sigma}}^{*}).
\end{align*}
Consequently, by the compactness of $\overline{\Sigma}^{(G)}$ and the continuity of $\mathbb{H}_{\bm{z}}^{(G)}$, we have
\begin{align*}
    \sup_{{\bm{\sigma}}^{(G)}\in\overline{\Sigma}^{(G)}}\mathbb{H}_{\bm{z}}^{(G)}({\bm{\sigma}}^{(G)})<\mathbb{H}_{\bm{z}}^{(G^*)}({\bm{\sigma}}^{*}).
\end{align*}
This completes the proof.
\end{proof}
\begin{proof}[\textbf{Proof of Theorem \ref{Gconsis}}]
The proof is based on the approach used in the proof of Theorem $1$ in Nguyen \cite{Nguyen(2024)}. Fix an arbitrary ${\bm{\pi}}_M\in\Pi_M$. First, we show that
\begin{align}
    {\rm{P}}_{\bm{z}}\biggl(\hat{G}_n(\bm{\pi}_M)<G^*\biggr)\longrightarrow 0\label{<G}
\end{align}
as $n\longrightarrow\infty$. Since the definition of $\hat{G}_n(\bm{\pi}_M)$ implies
\begin{align*}
    {\rm{P}}_{\bm{z}}\biggl(\hat{G}_n(\bm{\pi}_M)<G^*\biggr)=0
\end{align*}
when $G^*=1$, it is enough to consider the case $G^*\geq 2$. By Corollary \ref{supGprob}, for each $G=1,\ldots,M$, we have
\begin{align}
    \sup_{{\bm{\sigma}}^{(G)}\in\overline{\Sigma}^{(G)}}\frac{1}{n}{\bf{H}}^{(G)}_{n}({\bm{\sigma}}^{(G)};{\bm{\pi}}^{(G)})\xlongrightarrow{p}\sup_{{\bm{\sigma}}^{(G)}\in\overline{\Sigma}^{(G)}}\mathbb{H}_{\bm{z}}^{(G)}({\bm{\sigma}}^{(G)})\label{supH}
\end{align}
under ${\rm{P}}_{\bm{z}}$. In particular, it holds that
\begin{align}
    \sup_{{\bm{\sigma}}^{(G^*)}\in\overline{\Sigma}^{(G^*)}}\frac{1}{n}{\bf{H}}^{(G^*)}_{n}({\bm{\sigma}}^{(G^*)};{\bm{\pi}}^{(G^*)})\xlongrightarrow{p}\mathbb{H}_{\bm{z}}^{(G^*)}({\bm{\sigma}}^{*})\label{Hstar}
\end{align}
under ${\rm{P}}_{\bm{z}}$. By Proposition \ref{Hlemma}, for each $G=1,\ldots,G^*-1$,
\begin{align*}
    \sup_{{\bm{\sigma}}^{(G)}\in\overline{\Sigma}^{(G)}}\mathbb{H}_{\bm{z}}^{(G)}({\bm{\sigma}}^{(G)})<\mathbb{H}_{\bm{z}}^{(G^*)}({\bm{\sigma}}^{*}).
\end{align*}
Thus, we set
\begin{align*}
    r^{(G)}=-2\Biggl\{\sup_{{\bm{\sigma}}^{(G)}\in\overline{\Sigma}^{(G)}}\mathbb{H}_{\bm{z}}^{(G)}({\bm{\sigma}}^{(G)})-\mathbb{H}_{\bm{z}}^{(G^*)}({\bm{\sigma}}^{*})\Biggr\}>0.
\end{align*}
For each $G=1,\ldots,G^*-1$, it follows from \eqref{supH}, \eqref{Hstar} and [{\bf{B}}] (a) that
\begin{align*}
    &\quad\ \frac{1}{n}{\bf{F}}_n(G^*;{\bm{\pi}}^{(G^*)})-\frac{1}{n}{\bf{F}}_n(G;{\bm{\pi}}^{(G)})\\
    &=-2\sup_{{\bm{\sigma}}^{(G^*)}\in\overline{\Sigma}^{(G^*)}}\frac{1}{n}{\bf{H}}^{(G^*)}_{n}({\bm{\sigma}}^{(G^*)};{\bm{\pi}}^{(G^*)})+\frac{P^{(G^*)}_{n}}{n}\\
    &\qquad\qquad\quad+2\sup_{{\bm{\sigma}}^{(G)}\in\overline{\Sigma}^{(G)}}\frac{1}{n}{\bf{H}}^{(G)}_{n}({\bm{\sigma}}^{(G)};{\bm{\pi}}^{(G)})-\frac{P_n^{(G)}}{n}\xlongrightarrow{p}-r^{(G)}
\end{align*}
under ${\rm{P}}_{\bm{z}}$, so that 
\begin{align*}
    &\quad\ {\rm{P}}_{\bm{z}}\biggl({\bf{F}}_n(G^*;{\bm{\pi}}^{(G^*)})\geq{\bf{F}}_n(G;{\bm{\pi}}^{(G)})\biggr)\\
    &={\rm{P}}_{\bm{z}}\Biggl(\frac{1}{n}{\bf{F}}_n(G^*;{\bm{\pi}}^{(G^*)})-\frac{1}{n}{\bf{F}}_n(G;{\bm{\pi}}^{(G)})+r^{(G)}\geq r^{(G)}\Biggr)\\
    &\leq {\rm{P}}_{\bm{z}}\Biggl(\biggl|\frac{1}{n}{\bf{F}}_n(G^*;{\bm{\pi}}^{(G^*)})-\frac{1}{n}{\bf{F}}_n(G;{\bm{\pi}}^{(G)})+r^{(G)}\biggr|\geq r^{(G)}\Biggr)\longrightarrow 0
\end{align*}
as $n\longrightarrow\infty$. Therefore, we obtain
\begin{align*}
    {\rm{P}}_{\bm{z}}\biggl(\hat{G}_n({\bm{\pi}}_M)<G^*\biggr)
    &\leq {\rm{P}}_{\bm{z}}\Biggl(\bigcup_{G=1}^{G^*-1}\Bigl\{{\bf{F}}_n(G^*;{\bm{\pi}}^{(G^*)})\geq {\bf{F}}_n(G;{\bm{\pi}}^{(G)})\Bigr\}\Biggr)\\
    &\leq\sum_{G=1}^{G^*-1}{\rm{P}}_{\bm{z}}\biggl({\bf{F}}_n(G^*;{\bm{\pi}}^{(G^*)})\geq{\bf{F}}_n(G;{\bm{\pi}}^{(G)})\biggr)\longrightarrow 0
\end{align*}
as $n\longrightarrow\infty$, which yields \eqref{<G}. Next, we show that
\begin{align}
    {\rm{P}}_{\bm{z}}\biggl(\hat{G}_n({\bm{\pi}}_M)>G^*\biggr)\longrightarrow 0\label{>G}
\end{align}
as $n\longrightarrow\infty$. Since the definition of $\hat{G}_n(\bm{\pi}_M)$ implies
\begin{align*}
    {\rm{P}}_{\bm{z}}\biggl(\hat{G}_n(\bm{\pi}_M)>G^*\biggr)=0
\end{align*}
when $G^*=M$, it is enough to consider the case $G^*\leq M-1$. For all sufficiently large $n$, by [{\bf{B}}] (b), it follows that
\begin{align*}
    &\quad\ {\rm{P}}_{\bm{z}}\biggl(\hat{G}_n({\bm{\pi}}_M)>G^*\biggr)\\
    &\leq {\rm{P}}_{\bm{z}}\Biggl(\bigcup_{G=G^*+1}^{M}\Bigl\{{\bf{F}}_n(G^*;{\bm{\pi}}^{(G^*)})\geq {\bf{F}}_n(G;{\bm{\pi}}^{(G)})\Bigr\}\Biggr)\\
    &\leq \sum_{G=G^*+1}^{M}{\rm{P}}_{\bm{z}}\Biggl(-2\sup_{{\bm{\sigma}}^{(G^*)}\in\overline{\Sigma}^{(G^*)}}{\bf{H}}^{(G^*)}_{n}({\bm{\sigma}}^{(G^*)};{\bm{\pi}}^{(G^*)})+P^{(G^*)}_{n}\\
    &\qquad\qquad\qquad\qquad\qquad\quad\geq
    -2\sup_{{\bm{\sigma}}^{(G)}\in\overline{\Sigma}^{(G)}}{\bf{H}}^{(G)}_{n}({\bm{\sigma}}^{(G)};{\bm{\pi}}^{(G)})+P^{(G)}_{n}\Biggr)\\
    &\leq\sum_{G=G^*+1}^{M}{\rm{P}}_{\bm{z}}\Biggl(-\frac{2}{P_n^{(G)}-P^{(G^*)}_{n}}\biggl\{
    \sup_{{\bm{\sigma}}^{(G^*)}\in\overline{\Sigma}^{(G^*)}}{\bf{H}}^{(G^*)}_{n}({\bm{\sigma}}^{(G^*)};{\bm{\pi}}^{(G^*)})\\
    &\qquad\qquad\qquad\qquad\qquad\qquad\qquad-\sup_{{\bm{\sigma}}^{(G)}\in\overline{\Sigma}^{(G)}}{\bf{H}}^{(G)}_{n}({\bm{\sigma}}^{(G)};{\bm{\pi}}^{(G)})\biggr\}\geq 1\Biggr).
\end{align*}
By Proposition \ref{supOp} and [{\bf{B}}] (b), for each $G=G^*+1,\ldots,M$, we have
\begin{align*}
    &-\frac{2}{P_n^{(G)}-P^{(G^*)}_{n}}\Biggl\{\sup_{{\bm{\sigma}}^{(G^*)}\in\overline{\Sigma}^{(G^*)}}{\bf{H}}^{(G^*)}_{n}({\bm{\sigma}}^{(G^*)};{\bm{\pi}}^{(G^*)})\\
    &\qquad\qquad\qquad\qquad\qquad\qquad-\sup_{{\bm{\sigma}}^{(G)}\in\overline{\Sigma}^{(G)}}{\bf{H}}^{(G)}_{n}({\bm{\sigma}}^{(G)};{\bm{\pi}}^{(G)})\Biggr\}=o_p(1)
\end{align*}
under ${\rm{P}}_{\bm{z}}$, which implies \eqref{>G}. Therefore, it holds from $\eqref{<G}$ and $\eqref{>G}$ that
\begin{align*}
    {\rm{P}}_{\bm{z}}\biggl(\hat{G}_n({\bm{\pi}}_M)\neq G^*\biggr)
    &={\rm{P}}_{\bm{z}}\biggl(\hat{G}_n({\bm{\pi}}_M)<G^*\biggr)\\
    &\qquad+{\rm{P}}_{\bm{z}}\biggl(\hat{G}_n({\bm{\pi}}_M)>G^*\biggr)\longrightarrow 0
\end{align*}
as $n\longrightarrow\infty$, which completes the proof.
\end{proof}
\noindent Next, we prove Theorem \ref{consis}. For this purpose, we first consider the case where the true number of clusters $G^*$ is known, and define an estimator of ${\bm{\sigma}}^*$ by
\begin{align*}
    {\hat{\bm{\sigma}}}^*_{n}({\bm{\pi}}^{(G^*)})
    \in\argsup_{{\bm{\sigma}}^{(G^*)}\in\overline{\Sigma}^{(G^*)}}{\bf{L}}^{(G^*)}_{n}({\bm{\sigma}}^{(G^*)};{\bm{\pi}}^{(G^*)}).
\end{align*}
For notational simplicity, in the proofs of Lemmas \ref{idenlemma}--\ref{sigmaunif} and Propositions \ref{consisprop}--\ref{asymprop}, we may suppress the model index $G^*$. 
\begin{lemma}\label{idenlemma}
For every ${\bm{\sigma}}^{(G^*)}\in\overline{\Sigma}^{(G^*)}$, if
\begin{align*}
   \max_{g\in\{1,\ldots,G^*\}}
   \bar{\mathbb H}_{z_j}(\sigma^{(G^*)}_g)
   =\bar{\mathbb H}_{z_j}(\sigma_{z_j}^*)
\end{align*}
for each $j=1,\ldots,J$, then
${\bm{\sigma}}^{(G^*)}={\bm{\sigma}}^*$.
\end{lemma}
\begin{proof}
Fix ${\bm{\sigma}}\in\overline{\Sigma}$ arbitrarily and suppose that
\begin{align}
   \max_{g\in\{1,\ldots,G^*\}}\bar{\mathbb{H}}_{z_j}(\sigma_g)=\bar{\mathbb{H}}_{z_j}(\sigma_{z_j}^*)\label{idenassum}
\end{align}
for each $j=1,\ldots,J$. For each $k=1,\ldots,G^*$, since
\begin{align*}
    S_k=\Bigl\{j\in\{1,\ldots,J\}\mid z_j=k\Bigr\}\neq\emptyset,
\end{align*}
we can choose an index $j_k\in S_k$. Hence, for each $k=1,\ldots,G^*$, by \eqref{idenassum}, we have
\begin{align*}
   \max_{g\in\{1,\ldots,G^*\}}\bar{\mathbb{H}}_{z_{j_k}}(\sigma_g)=\bar{\mathbb{H}}_{z_{j_k}}(\sigma_{z_{j_k}}^*),
\end{align*}
so that $z_{j_k}=k$ yields
\begin{align}
   \max_{g\in\{1,\ldots,G^*\}}
   \bar{\mathbb{H}}_{k}(\sigma_g)
   =\bar{\mathbb{H}}_{k}(\sigma_k^*). \label{maxk}
\end{align}
For each $k=1,\ldots,G^*$, by \eqref{maxk}, there exists an index $r_k\in\{1,\ldots,G^*\}$ such that
\begin{align*}
   \bar{\mathbb H}_{k}(\sigma_{r_k})
   =\bar{\mathbb H}_{k}(\sigma_k^*).
\end{align*}
Consequently, for each $k=1,\ldots,G^*$, since $\bar{\mathbb{H}}_{k}(\sigma)$ attains its unique maximum at $\sigma=\sigma_k^*$, this equality implies
\begin{align}
    \sigma_{r_k}=\sigma_k^*. \label{sigmark}
\end{align}
Recalling that $\sigma_1^*<\sigma_2^*<\cdots<\sigma_{G^*}^*$, it follows from \eqref{sigmark} that
\begin{align*}
    \sigma_{r_1}<\sigma_{r_2}<\cdots<\sigma_{r_{G^*}}.
\end{align*}
On the other hand, since ${\bm{\sigma}}\in\overline{\Sigma}$, we see
\begin{align*}
    \sigma_1\leq\sigma_2\leq\cdots\leq\sigma_{G^*},
\end{align*}
which implies
\begin{align*}
    r_1<r_2<\cdots<r_{G^*}.
\end{align*}
Since $r_k\in\{1,\ldots,G^*\}$ for each $k=1,\ldots,G^*$, one gets
\begin{align}
    r_k=k \label{rkk}
\end{align}
for each $k=1,\ldots,G^*$. Therefore, by \eqref{sigmark} and \eqref{rkk}, we obtain
\begin{align*}
    \sigma_k=\sigma_k^*
\end{align*}
for each $k=1,\ldots,G^*$, which yields ${\bm{\sigma}}={\bm{\sigma}}^*$. This completes the proof.
\end{proof}
\begin{lemma}\label{sigmaiden}
For every ${\bm{\sigma}}^{(G^*)}\in\overline{\Sigma}^{(G^*)}$, if
\begin{align*}
    \mathbb{H}^{(G^*)}_{\bm{z}}({\bm{\sigma}}^{(G^*)})=\mathbb{H}^{(G^*)}_{\bm{z}}({\bm{\sigma}}^*),
\end{align*}
then ${\bm{\sigma}}^{(G^*)}={\bm{\sigma}}^*$.
\end{lemma}
\begin{proof}
Fix an arbitrary ${\bm{\sigma}}\in\overline{\Sigma}$ and suppose that
\begin{align*}
    \mathbb{H}_{\bm{z}}({\bm{\sigma}})=\mathbb{H}_{\bm{z}}({\bm{\sigma}}^*).
\end{align*}
Then, we have
\begin{align}
    0=\mathbb{H}_{\bm{z}}({\bm{\sigma}})-\mathbb{H}_{\bm{z}}({\bm{\sigma}}^*)
    &=\sum_{j=1}^J\biggl\{
    \max_{g\in\{1,\ldots,G^*\}}\bar{\mathbb H}_{z_j}(\sigma_g)
    -\bar{\mathbb H}_{z_j}(\sigma_{z_j}^*)\biggr\}. \label{sum}
\end{align}
For each $j=1,\ldots,J$, since $\bar{\mathbb H}_{z_j}(\sigma)$ is maximized at $\sigma=\sigma_{z_j}^*$, we have
\begin{align*}
    \max_{g\in\{1,\ldots,G^*\}}\bar{\mathbb H}_{z_j}(\sigma_g)\leq\bar{\mathbb H}_{z_j}(\sigma_{z_j}^*), 
\end{align*}
so that each term in the sum in \eqref{sum} is non-positive. Since the sum is zero, each term must be zero. In other words, for each $j=1,\ldots,J$, it follows that
\begin{align*}
    \max_{g\in\{1,\ldots,G^*\}}
    \bar{\mathbb H}_{z_j}(\sigma_g)
    -\bar{\mathbb H}_{z_j}(\sigma_{z_j}^*)=0.
\end{align*}
Therefore, by Lemma \ref{idenlemma}, we obtain ${\bm{\sigma}}={\bm{\sigma}}^*$.
\end{proof}
\begin{proposition}\label{consisprop}
Suppose that ${[\bf A1]}$ and ${[\bf A2]}$ hold. Then, for every ${\bm{\pi}}^{(G^*)}\in\Pi^{(G^*)}$,
\begin{align*}
    \hat{\bm{\sigma}}^{*}_n({\bm{\pi}}^{(G^*)})\xlongrightarrow{p}{\bm{\sigma}}^*
\end{align*}
under ${\rm{P}}_{\bm{z}}$ as $n\longrightarrow\infty$.
\end{proposition}
\begin{proof}
Fix ${\bm{\pi}}\in\Pi$ arbitrarily. By Lemma \ref{sigmaiden} and the continuity of $\mathbb{H}_{\bm{z}}$, for every $\varepsilon>0$, we have
\begin{align*}
    \sup_{{\bm{\sigma}}\in\bigl\{{\bm{\sigma}}\in\overline{\Sigma}\, ;\, |{\bm{\sigma}}-{\bm{\sigma}}^*|\geq\varepsilon\bigr\}}\mathbb{H}_{\bm{z}}({\bm{\sigma}})<\mathbb{H}_{\bm{z}}({\bm{\sigma}}^*).
\end{align*}
Consequently, there exists $\delta>0$ such that for every ${\bm{\sigma}}\in\overline{\Sigma}$,
\begin{align*}
    |{\bm{\sigma}}-{\bm{\sigma}}^*|>\varepsilon
    \Longrightarrow
    \mathbb{H}_{\bm{z}}({\bm{\sigma}}^*)-\mathbb{H}_{\bm{z}}({\bm{\sigma}})>\delta.
\end{align*}
Hence, it follows from Proposition \ref{unifprop} and the definition of $\hat{\bm{\sigma}}^*_n({\bm{\pi}})$ that
\begin{align*}
    {\rm{P}}_{\bm{z}}\biggl(|\hat{\bm{\sigma}}^*_n({\bm{\pi}})-{\bm{\sigma}}^*|>\varepsilon\biggr)
    &\leq{\rm{P}}_{\bm{z}}\biggl(\mathbb{H}_{\bm{z}}({\bm{\sigma}}^*)-\mathbb{H}_{\bm{z}}(\hat{\bm{\sigma}}^*_n({\bm{\pi}}))>\delta
    \biggr)\\
    &\leq{\rm{P}}_{\bm{z}}\biggl(\mathbb{H}_{\bm{z}}({\bm{\sigma}}^*)-\frac{1}{n}{\bf H}_n({\bm{\sigma}}^*;{\bm{\pi}})>\frac{\delta}{3}\biggr)\\
    &\quad+{\rm{P}}_{\bm{z}}\Biggl(\frac{1}{n}{\bf H}_n({\bm{\sigma}}^*;{\bm{\pi}})
    -\frac{1}{n}{\bf H}_n(\hat{\bm{\sigma}}^*_n({\bm{\pi}});{\bm{\pi}})>\frac{\delta}{3}
    \Biggr)\\
    &\quad+{\rm{P}}_{\bm{z}}\Biggl(\frac{1}{n}{\bf H}_n(\hat{\bm{\sigma}}^*_n({\bm{\pi}});{\bm{\pi}})-\mathbb{H}_{\bm{z}}(\hat{\bm{\sigma}}^*_n({\bm{\pi}}))>\frac{\delta}{3}
    \Biggr)\\
    &\leq 2{\rm{P}}_{\bm{z}}\Biggl(\sup_{{\bm{\sigma}}\in\overline{\Sigma}}\biggl|\frac{1}{n}{\bf H}_n({\bm{\sigma}};{\bm{\pi}})-\mathbb{H}_{\bm{z}}({\bm{\sigma}})\biggr|>\frac{\delta}{3}
    \Biggr)\longrightarrow 0
\end{align*}
as $n\longrightarrow\infty$, which completes the proof.
\end{proof}
\noindent For $\rho>0$, we set 
\begin{align*}
    B^*_g(\rho)=\Bigl\{\sigma\in[1/d,d]\ \big|\ \bigl|\sigma-\sigma^*_{g}\bigr|\leq\rho\Bigr\}
\end{align*}
for $g=1,\ldots,G^*$ and
\begin{align*}
    B^*(\rho)=\Bigl\{{\bm{\sigma}}^{(G^*)}\in\overline{\Sigma}^{(G^*)}\, \big|\ \bigl|{\bm{\sigma}}^{(G^*)}-{\bm{\sigma}}^*\bigr|\leq\rho\Bigr\}.
\end{align*}
\begin{lemma}\label{infsuplemma}
Suppose that ${[\bf{A1}]}$ and ${[\bf{A2}]}$ hold and $G^*\geq 2$. Then, there exist $\rho>0$ and $K>0$ such that for each $j=1,\ldots,J$ and each $g\in\{1,\ldots,G^*\}\setminus\{z_j\}$, 
\begin{align*}
    {\rm{P}}_{\bm{z}}\Biggl(\inf_{\sigma\in B^*_{z_j}(\rho)}\frac{1}{n}\bar{\bf{H}}_{j,n}(\sigma)-\sup_{\sigma\in B_g^*(\rho)}\frac{1}{n}\bar{\bf{H}}_{j,n}(\sigma)\geq K\Biggr)\longrightarrow 1
\end{align*}
as $n\longrightarrow\infty$.
\end{lemma}
\begin{proof}
For each $j=1,\ldots,J$ and $g\in\{1,\ldots,G^*\}\setminus\{z_j\}$, by
\begin{align*}
    \bar{\mathbb{H}}_{z_j}(\sigma^*_{z_j})>\bar{\mathbb{H}}_{z_j}(\sigma^*_g)
\end{align*}
and the continuity of $\bar{\mathbb{H}}_{z_j}$, there exists $\rho_{j,g}>0$ such that
\begin{align*}
    \inf_{\sigma\in B_{z_j}^*(\rho_{j,g})}\bar{\mathbb{H}}_{z_j}(\sigma)>\sup_{\sigma\in B_g^*(\rho_{j,g})}\bar{\mathbb{H}}_{z_j}(\sigma).
\end{align*}
Let 
\begin{align*}
    \rho=\min_{j\in\{1,\ldots,J\}}\min_{g\in\{1,\ldots,G^*\}\setminus\{z_j\}}\rho_{j,g}
\end{align*}
and 
\begin{align*}
    K=\frac{1}{2}\min_{j\in\{1,\ldots,J\}}\min_{g\in\{1,\ldots,G^*\}\setminus\{z_j\}}K_{j,g},
\end{align*}
where
\begin{align*}
    K_{j,g}=\inf_{\sigma\in B_{z_j}^*(\rho)}\bar{\mathbb{H}}_{z_j}(\sigma)-\sup_{\sigma\in B_g^*(\rho)}\bar{\mathbb{H}}_{z_j}(\sigma)
\end{align*}
for $j\in\{1,\ldots,J\}$ and $g\in\{1,\ldots,G^*\}\setminus\{z_j\}$. Fix $j\in\{1,\ldots,J\}$ and $g\in\{1,\ldots,G^*\}\setminus\{z_j\}$ arbitrarily. Since
\begin{align*}
    \inf_{\sigma\in B_{z_j}^*(\rho)}\bar{\mathbb{H}}_{z_j}(\sigma)>\sup_{\sigma\in B_g^*(\rho)}\bar{\mathbb{H}}_{z_j}(\sigma),
\end{align*}
we note  that $K>0$ and $K_{j,g}\geq 2K$. By Lemma \ref{Hbarlemma}, we have
\begin{align*}
    &\inf_{\sigma\in B^*_{z_j}(\rho)}\frac{1}{n}\bar{\bf{H}}_{j,n}(\sigma)-\sup_{\sigma\in B_g^*(\rho)}\frac{1}{n}\bar{\bf{H}}_{j,n}(\sigma)\\
    &\qquad\qquad\quad\xlongrightarrow{p}\inf_{\sigma\in B_{z_j}^*(\rho)}\bar{\mathbb{H}}_{z_j}(\sigma)-\sup_{\sigma\in B_g^*(\rho)}\bar{\mathbb{H}}_{z_j}(\sigma)=K_{j,g}
\end{align*}
under ${\rm{P}}_{\bm{z}}$, so that 
\begin{align*}
    &\quad\ {\rm{P}}_{\bm{z}}\Biggl(\inf_{\sigma\in B^*_{z_j}(\rho)}\frac{1}{n}\bar{\bf{H}}_{j,n}(\sigma)-\sup_{\sigma\in B_g^*(\rho)}\frac{1}{n}\bar{\bf{H}}_{j,n}(\sigma)<K\Biggr)\\
    &={\rm{P}}_{\bm{z}}\Biggl(\inf_{\sigma\in B^*_{z_j}(\rho)}\frac{1}{n}\bar{\bf{H}}_{j,n}(\sigma)-\sup_{\sigma\in B_g^*(\rho)}\frac{1}{n}\bar{\bf{H}}_{j,n}(\sigma)-K_{j,g}<K-K_{j,g}\Biggr)\\
    &\leq{\rm{P}}_{\bm{z}}\Biggl(\inf_{\sigma\in B^*_{z_j}(\rho)}\frac{1}{n}\bar{\bf{H}}_{j,n}(\sigma)-\sup_{\sigma\in B_g^*(\rho)}\frac{1}{n}\bar{\bf{H}}_{j,n}(\sigma)-K_{j,g}<-K\Biggr)\\
    &\leq{\rm{P}}_{\bm{z}}\Biggl(\biggl|\inf_{\sigma\in B^*_{z_j}(\rho)}\frac{1}{n}\bar{\bf{H}}_{j,n}(\sigma)-\sup_{\sigma\in B_g^*(\rho)}\frac{1}{n}\bar{\bf{H}}_{j,n}(\sigma)-K_{j,g}\biggr|>K\Biggr)\longrightarrow 0
\end{align*}
as $n\longrightarrow\infty$. This completes the proof.
\end{proof}
\begin{lemma}\label{elemma}
Suppose that ${[\bf{A1}]}$ and ${[\bf{A2}]}$ hold and $G^*\geq 2$. Then, for every ${\bm{\pi}}^{(G^*)}\in\Pi^{(G^*)}$, there exist $\rho>0$, $K_1>0$ and $K_2>0$ such that for each $j=1,\ldots,J$ and each
$g\in\{1,\ldots,G^*\}\setminus\{z_j\}$,
\begin{align*}
    {\rm{P}}_{\bm{z}}\Biggl(\sup_{{\bm{\sigma}}^{(G^*)}\in B^*(\rho)}{\bf{Q}}^{(G^*)}_{j,n}(g;\, {\bm{\sigma}}^{(G^*)},{\bm{\pi}}^{(G^*)})\leq K_1e^{-K_2n}\Biggr)\longrightarrow 1
\end{align*}
as $n\longrightarrow\infty$.
\end{lemma}
\begin{proof}
Fix an arbitrary ${\bm{\pi}}\in\Pi$. Let
\begin{align*}
    K_1=\biggl(\min_{k\in\{1,\ldots,G^*\}}\pi_{k}\biggr)^{-1}.
\end{align*}
By Lemma \ref{infsuplemma}, there exist $\rho>0$ and $K_2>0$ such that for each $j=1,\ldots,J$ and $g\in\{1,\ldots,G^*\}\setminus\{z_j\}$,
\begin{align}
    {\rm{P}}_{\bm{z}}\Bigl(B_{g,j,n}\Bigr)\longrightarrow 1 \label{probB}
\end{align}
as $n\longrightarrow\infty$, where
\begin{align*}
    B_{g,j,n}=\Biggl\{\inf_{\sigma\in B^*_{z_j}(\rho)}\frac{1}{n}\bar{\bf{H}}_{j,n}(\sigma)-\sup_{\sigma\in B_g^*(\rho)}\frac{1}{n}\bar{\bf{H}}_{j,n}(\sigma)\geq K_2\Biggr\}.
\end{align*}
Fix $j\in\{1,\ldots,J\}$ and $g\in\{1,\ldots,G^*\}\setminus
\{z_j\}$ arbitrarily. Note that
\begin{align*}
    {\bf{Q}}_{j,n}(g;\, {\bm{\sigma}},{\bm{\pi}})&=\frac{\pi_{g}\exp\Bigl\{\bar{\bf{H}}_{j,n}(\sigma_g)\Bigr\}
    }{\displaystyle\sum_{k=1}^{G^*}\pi_{k}\exp\Bigl\{\bar{\bf{H}}_{j,n}(\sigma_k)\Bigr\}}\\
    &=\frac{1
    }{1+\displaystyle\sum_{k\neq g}\frac{\pi_k}{\pi_g}\exp\Bigl\{\bar{\bf{H}}_{j,n}(\sigma_k)-\bar{\bf{H}}_{j,n}(\sigma_g)\Bigr\}}.
\end{align*}
Moreover, ${\bm{\sigma}}\in B^*(\rho)$ implies that $\sigma_{g}\in B^*_{g}(\rho)$ and 
$\sigma_{z_j}\in B^*_{z_j}(\rho)$ since 
\begin{align*}
    \bigl|\sigma_g-\sigma_g^*\bigr|\leq\bigl|{\bm{\sigma}}-{\bm{\sigma}}^*\bigr|\leq\rho,\quad \bigl|\sigma_{z_j}-\sigma_{z_j}^*\bigr|\leq\bigl|{\bm{\sigma}}-{\bm{\sigma}}^*\bigr|\leq\rho.
\end{align*}
Recalling that $0<\pi_g\leq 1$ and $z_j\neq g$, on $B_{g,j,n}$, for every ${\bm{\sigma}}\in B^*(\rho)$, we have
\begin{align*}
     &\quad\ 1+\displaystyle\sum_{k\neq g}\frac{\pi_k}{\pi_g}\exp\Bigl\{\bar{\bf{H}}_{j,n}(\sigma_k)-\bar{\bf{H}}_{j,n}(\sigma_g)\Bigr\}\\
    &\geq 1+\displaystyle\frac{\pi_{z_j}}{\pi_g}\exp\Bigl\{\bar{\bf{H}}_{j,n}(\sigma_{z_j})-\bar{\bf{H}}_{j,n}(\sigma_g)\Bigr\}\\
    &\geq 1+\frac{1}{\pi_g}\biggl(\min_{k\in\{1,\ldots,G^*\}}\pi_{k}\biggr)\exp\Biggl\{n\Biggl(\inf_{\sigma\in B^*_{z_j}(\rho)}\frac{1}{n}\bar{\bf{H}}_{j,n}(\sigma)-\sup_{\sigma\in B_g^*(\rho)}\frac{1}{n}\bar{\bf{H}}_{j,n}(\sigma)\Biggr)\Biggr\}\\
    &\geq 1+\biggl(\min_{k\in\{1,\ldots,G^*\}}\pi_{k}\biggr)e^{K_2n}\geq\biggl(\min_{k\in\{1,\ldots,G^*\}}\pi_{k}\biggr)e^{K_2n},
\end{align*}
so that
\begin{align*}
     {\bf{Q}}_{j,n}(g;\, {\bm{\sigma}},{\bm{\pi}})\leq K_1e^{-K_2n},
\end{align*}
which implies
\begin{align*}
    \sup_{{\bm{\sigma}}\in B^*(\rho)}{\bf{Q}}_{j,n}(g;\, {\bm{\sigma}},{\bm{\pi}})\leq K_1e^{-K_2n}.
\end{align*}
Therefore, since 
\begin{align*}
    {\rm{P}}_{\bm{z}}
    \Bigl(B_{g,j,n}\Bigr)\leq{\rm{P}}_{\bm{z}}\Biggl(\sup_{{\bm{\sigma}}\in B^*(\rho)}{\bf{Q}}_{j,n}(g;\, {\bm{\sigma}},{\bm{\pi}})\leq K_1e^{-K_2n}\Biggr),
\end{align*}
it holds from \eqref{probB} that
\begin{align*}
    {\rm{P}}_{\bm{z}}\Biggl(\sup_{{\bm{\sigma}}\in B^*(\rho)}{\bf{Q}}_{j,n}(g;\, {\bm{\sigma}},{\bm{\pi}})\leq K_1e^{-K_2n}\Biggr)\longrightarrow 1
\end{align*}
as $n\longrightarrow\infty$. This completes the proof.
\end{proof}
\begin{lemma}\label{Qsuplemma}
Suppose that ${[\bf{A1}]}$ and ${[\bf{A2}]}$ hold. Then, under ${\rm{P}}_{\bm{z}}$, for every ${\bm{\pi}}^{(G^*)}\in\Pi^{(G^*)}$ and $L\geq 0$, there exists $\rho>0$ such that for each $j=1,\ldots,J$ and $g=1,\ldots,G^*$,
\begin{align*}
    n^L\sup_{{\bm{\sigma}}^{(G^*)}\in B^*(\rho)}\Bigl|{\bf{Q}}^{(G^*)}_{j,n}(g;\, {\bm{\sigma}}^{(G^*)},{\bm{\pi}}^{(G^*)})-\mathbb{Q}_{z_j}(g)\Bigr|\xlongrightarrow{p}0
\end{align*}
as $n\longrightarrow\infty$.
\end{lemma}
\begin{proof}
Fix ${\bm{\pi}}\in\Pi$ and $L\geq 0$ arbitrarily. When $G^*=1$, since
\begin{align*}
    {\bf{Q}}_{j,n}(1;\, {\bm{\sigma}},{\bm{\pi}})=1,\quad \mathbb{Q}_{z_j}(1)=1
\end{align*}
for each $j=1,\ldots,J$, this is trivial. Hence, it suffices to consider the case where $G^*\geq 2$. By Lemma \ref{elemma}, there exist $\rho>0$, $K_1>0$ and $K_2>0$ such that for each $j=1,\ldots,J$ and each $g\in\{1,\ldots,G^*\}\setminus
\{z_j\}$,
\begin{align}
    {\rm{P}}_{\bm{z}}\Bigl(C_{g,j,n}({\bm{\pi}})\Bigr)\longrightarrow 1 \label{Cprob}
\end{align}
as $n\longrightarrow\infty$, where 
\begin{align*}
    C_{g,j,n}({\bm{\pi}})=\Biggl\{\sup_{{\bm{\sigma}}\in B^*(\rho)}{\bf{Q}}_{j,n}(g;\, {\bm{\sigma}},{\bm{\pi}})\leq K_1e^{-K_2n}\Biggr\}.
\end{align*}
Fix an arbitrary $j\in\{1,\ldots,J\}$. For every $\varepsilon>0$ and each $g\in\{1,\ldots,G^*\}\setminus\{z_j\}$, by \eqref{Cprob}, we have
\begin{align*}
    &\quad\ {\rm{P}}_{\bm{z}}\Biggl(n^L\sup_{{\bm{\sigma}}\in B^*(\rho)}{\bf{Q}}_{j,n}(g;\, {\bm{\sigma}},{\bm{\pi}})>\varepsilon\Biggr)\\
    &={\rm{P}}_{\bm{z}}\Biggl(\biggl\{n^L\sup_{{\bm{\sigma}}\in B^*(\rho)}{\bf{Q}}_{j,n}(g;\, {\bm{\sigma}},{\bm{\pi}})>\varepsilon\biggr\}\cap C_{g,j,n}({\bm{\pi}})\Biggr)\\
    &\qquad+{\rm{P}}_{\bm{z}}\Biggl(\biggl\{n^L\sup_{{\bm{\sigma}}\in B^*(\rho)}{\bf{Q}}_{j,n}(g;\, {\bm{\sigma}},{\bm{\pi}})>\varepsilon\biggr\}\cap C^c_{g,j,n}({\bm{\pi}})\Biggr)\\
    &\leq {\rm{P}}_{\bm{z}}\biggl(K_1n^Le^{-K_2n}>\varepsilon\biggr)+{\rm{P}}_{\bm{z}}\Bigl(C^c_{g,j,n}({\bm{\pi}})\Bigr)\longrightarrow 0
\end{align*}
as $n\longrightarrow\infty$, which yields
\begin{align}
    n^L\sup_{{\bm{\sigma}}\in B^*(\rho)}{\bf{Q}}_{j,n}(g;\, {\bm{\sigma}},{\bm{\pi}})\xlongrightarrow{p}0\label{nLQ1}
\end{align}
under ${\rm{P}}_{\bm{z}}$. Moreover, since
\begin{align*}
    0&\leq 1-{\bf{Q}}_{j,n}(z_j;\, {\bm{\sigma}},{\bm{\pi}})\\
    &=\sum_{g\neq z_j}\left(\frac{\pi_{g}\exp\Bigl\{\bar{\bf{H}}_{j,n}(\sigma_g)\Bigr\}
    }{\displaystyle\sum_{k=1}^{G^*}\pi_{k}\exp\Bigl\{\bar{\bf{H}}_{j,n}(\sigma_k)\Bigr\}}\right)=\sum_{g\neq z_j}{\bf{Q}}_{j,n}(g;\, {\bm{\sigma}},{\bm{\pi}})
\end{align*}
and $\mathbb{Q}_{z_j}(z_j)=1$, under ${\rm{P}}_{\bm{z}}$, it follows from \eqref{nLQ1} that
\begin{align*}
    0&\leq n^L\sup_{{\bm{\sigma}}\in B^*(\rho)}\Bigl|{\bf{Q}}_{j,n}(z_j;\, {\bm{\sigma}},{\bm{\pi}})-\mathbb{Q}_{z_j}(z_j)\Bigr|\\
    &=n^L\sup_{{\bm{\sigma}}\in B^*(\rho)}\Bigl\{1-{\bf{Q}}_{j,n}(z_j;\, {\bm{\sigma}},{\bm{\pi}})\Bigr\}\\
    &\leq \sum_{g\neq z_j}n^L\sup_{{\bm{\sigma}}\in B^*(\rho)}{\bf{Q}}_{j,n}(g;\, {\bm{\sigma}},{\bm{\pi}})\xlongrightarrow{p}0,
\end{align*}
which gives
\begin{align}
    n^L\sup_{{\bm{\sigma}}\in B^*(\rho)}\Bigl|{\bf{Q}}_{j,n}(z_j;\, {\bm{\sigma}},{\bm{\pi}})-\mathbb{Q}_{z_j}(z_j)\Bigr|\xlongrightarrow{p}0.\label{nLQ2}
\end{align}
By noting that $\mathbb{Q}_{z_j}(g)=0$ for each $g\in\{1,\ldots,G^*\}\setminus
\{z_j\}$, \eqref{nLQ1} and \eqref{nLQ2} complete the proof.
\end{proof}
\begin{corollary}\label{Qcor}
Suppose that ${[\bf{A1}]}$ and ${[\bf{A2}]}$ hold. Then, under ${\rm{P}}_{\bm{z}}$, for every ${\bm{\pi}}^{(G^*)}\in\Pi^{(G^*)}$ and $L\geq 0$ and each $j=1,\ldots,J$ and $g=1,\ldots,G^*$,
\begin{align*}
    n^L\Bigl|{\bf{Q}}^{(G^*)}_{j,n}(g;\, {\bm{\sigma}}^{*},{\bm{\pi}}^{(G^*)})-\mathbb{Q}_{z_j}(g)\Bigr|\xlongrightarrow{p}0
\end{align*}
as $n\longrightarrow\infty$.
\end{corollary}
\begin{proof}
Since ${\bm{\sigma}}^*\in B^*(\rho)$ for every $\rho>0$, by Lemma \ref{Qsuplemma}, we can easily obtain the result.
\end{proof}
\begin{lemma}\label{H1lemma}
Suppose that ${[\bf{A1}]}$ and ${[\bf{A2}]}$ hold. Then, under ${\rm{P}}_{\bm{z}}$, for each $j=1,\ldots,J$,
\begin{align*}
   \frac{1}{\sqrt{n}}\bar{\bf{H}}'_{j,n}(\sigma^*_{z_j})\xlongrightarrow{d}N\Biggl(0,\frac{2}{\sigma^{*2}_{z_j}}\Biggr)
\end{align*}
as $n\longrightarrow\infty$.
\end{lemma}
\begin{proof}
For each $j=1,\ldots,J$, since 
\begin{align*}
    \frac{1}{\sqrt{n}}\bar{\bf H}'_{j,n}(\sigma)=-\frac{\sqrt{n}}{\sigma}
    +\frac{\sqrt{n}\hat{\Sigma}^2_{j,n}}{\sigma^3}
\end{align*}
for $\sigma>0$, Lemma \ref{Sigmalemma} implies
\begin{align*}
    \frac{1}{\sqrt{n}}\bar{\bf{H}}'_{j,n}(\sigma^*_{z_j})=\frac{1}{\sigma_{z_j}^{*3}}\sqrt{n}\Bigl(
    \hat{\Sigma}^2_{j,n}-\sigma^{*2}_{z_j}\Bigr)\xlongrightarrow{d}N\Biggl(0,\frac{2}{\sigma^{*2}_{z_j}}\Biggr)
\end{align*}
under ${\rm{P}}_{\bm{z}}$, which completes the proof.
\end{proof}
\begin{lemma}\label{difflemma}
For every ${\bm{\pi}}^{(G^*)}\in\Pi^{(G^*)}$ and ${\bm{\sigma}}^{(G^*)}\in\overline{\Sigma}^{(G^*)}$,
\begin{align}
    \partial_{\sigma_g^{(G^*)}}{\bf{H}}^{(G^*)}_n({\bm{\sigma}}^{(G^*)};{\bm{\pi}}^{(G^*)})=\sum_{j=1}^J{\bf{Q}}^{(G^*)}_{j,n}(g;\, {\bm{\sigma}}^{(G^*)},{\bm{\pi}}^{(G^*)})\bar{\bf{H}}'_{j,n}(\sigma_g^{(G^*)}) \label{d1}
\end{align}
and 
\begin{align}
\begin{split}
    &\quad\ \partial^2_{\sigma_g^{(G^*)}}{\bf{H}}^{(G^*)}_n({\bm{\sigma}}^{(G^*)};{\bm{\pi}}^{(G^*)})\\
    &=\sum_{j=1}^J{\bf{Q}}^{(G^*)}_{j,n}(g;\, {\bm{\sigma}}^{(G^*)},{\bm{\pi}}^{(G^*)})\bar{\bf{H}}''_{j,n}(\sigma_g^{(G^*)})\\
    &+\sum_{j=1}^J{\bf{Q}}^{(G^*)}_{j,n}(g;\, {\bm{\sigma}}^{(G^*)},{\bm{\pi}}^{(G^*)})\Bigl\{1-{\bf{Q}}^{(G^*)}_{j,n}(g;\, {\bm{\sigma}}^{(G^*)},{\bm{\pi}}^{(G^*)})\Bigr\}\Bigl\{\bar{\bf{H}}'_{j,n}(\sigma_g^{(G^*)})\Bigr\}^2 \label{d2}
\end{split}
\end{align}
for each $g=1,\ldots,G^*$, and 
\begin{align}
\begin{split}
    &\quad\ \partial_{\sigma_{g_1}^{(G^*)}}\partial_{\sigma_{g_2}^{(G^*)}}{\bf{H}}^{(G^*)}_n({\bm{\sigma}}^{(G^*)};{\bm{\pi}}^{(G^*)})\\
    &=-\sum_{j=1}^J{\bf{Q}}^{(G^*)}_{j,n}(g_1;\, {\bm{\sigma}}^{(G^*)},{\bm{\pi}}^{(G^*)}){\bf{Q}}^{(G^*)}_{j,n}(g_2;\, {\bm{\sigma}}^{(G^*)},{\bm{\pi}}^{(G^*)})\bar{\bf{H}}'_{j,n}(\sigma_{g_1}^{(G^*)})\bar{\bf{H}}'_{j,n}(\sigma_{g_2}^{(G^*)}) \label{d2d}
\end{split}
\end{align}
for each $g_1,g_2\in\{1,\ldots,G^*\}$ with $g_1\neq g_2$.
\end{lemma}
\begin{proof}
Fix an arbitrary ${\bm{\pi}}\in\Pi$. For each $g=1,\ldots,G^*$, since 
\begin{align*}
    \partial_{\sigma_g}\Biggl(\sum_{k=1}^{G^*}\pi_{k}\exp\Bigl\{\bar{\bf{H}}_{j,n}(\sigma_k)\Bigr\}\Biggr)&=\pi_g\partial_{\sigma_g}\exp\Bigl\{\bar{\bf{H}}_{j,n}(\sigma_g)\Bigr\}\\
    &=\pi_g\exp\Bigl\{\bar{\bf{H}}_{j,n}(\sigma_g)\Bigr\}\bar{\bf{H}}'_{j,n}(\sigma_g),
\end{align*}
we have
\begin{align*}
    \partial_{\sigma_g}{\bf{H}}_n({\bm{\sigma}};{\bm{\pi}})
    &=\sum_{j=1}^J\partial_{\sigma_g}\log\Biggl(\sum_{k=1}^{G^*}\pi_{k}\exp\Bigl\{\bar{\bf{H}}_{j,n}(\sigma_k)\Bigr\}\Biggr)\\
    &=\sum_{j=1}^J\left(\frac{1}{\displaystyle\sum_{k=1}^{G^*}\pi_{k}\exp\Bigl\{\bar{\bf{H}}_{j,n}(\sigma_k)\Bigr\}}\right)\partial_{\sigma_g}\Biggl(\sum_{k=1}^{G^*}\pi_{k}\exp\Bigl\{\bar{\bf{H}}_{j,n}(\sigma_k)\Bigr\}\Biggr)\\
    &=\sum_{j=1}^J{\bf{Q}}_{j,n}(g;\, {\bm{\sigma}},{\bm{\pi}})\bar{\bf{H}}'_{j,n}(\sigma_g),
\end{align*}
which yields (\ref{d1}). For each $g=1,\ldots,G^*$, we see
\begin{align*}
    \partial_{\sigma_g}\left(\frac{1}{\displaystyle\sum_{k=1}^{G^*}\pi_{k}\exp\Bigl\{\bar{\bf{H}}_{j,n}(\sigma_k)\Bigr\}}\right)&=-\frac{\displaystyle\partial_{\sigma_g}\Biggl(\sum_{k=1}^{G^*}\pi_{k}\exp\Bigl\{\bar{\bf{H}}_{j,n}(\sigma_k)\Bigr\}\Biggr)}{\displaystyle\Biggl(\sum_{k=1}^{G^*}\pi_{k}\exp\Bigl\{\bar{\bf{H}}_{j,n}(\sigma_k)\Bigr\}\Biggr)^2}\\
    &=-\frac{\pi_g\exp\Bigl\{\bar{\bf{H}}_{j,n}(\sigma_g)\Bigr\}}{\displaystyle\Biggl(\sum_{k=1}^{G^*}\pi_{k}\exp\Bigl\{\bar{\bf{H}}_{j,n}(\sigma_k)\Bigr\}\Biggr)^2}\bar{\bf{H}}'_{j,n}(\sigma_g),
\end{align*}
which gives 
\begin{align*}
    \partial_{\sigma_{g}}{\bf{Q}}_{j,n}(g;\, {\bm{\sigma}},{\bm{\pi}})&=\partial_{\sigma_g}\left(\frac{1}{\displaystyle\sum_{k=1}^{G^*}\pi_{k}\exp\Bigl\{\bar{\bf{H}}_{j,n}(\sigma_k)\Bigr\}}\right)\pi_g\exp\Bigl\{\bar{\bf{H}}_{j,n}(\sigma_g)\Bigr\}\\
    &\qquad+\left(\frac{1}{\displaystyle\sum_{k=1}^{G^*}\pi_{k}\exp\Bigl\{\bar{\bf{H}}_{j,n}(\sigma_k)\Bigr\}}\right)\pi_g\partial_{\sigma_g}\exp\Bigl\{\bar{\bf{H}}_{j,n}(\sigma_g)\Bigr\}\\
    &=-\Bigl\{{\bf{Q}}_{j,n}(g;\, {\bm{\sigma}},{\bm{\pi}})\Bigr\}^2\bar{\bf{H}}'_{j,n}(\sigma_g)+{\bf{Q}}_{j,n}(g;\, {\bm{\sigma}},{\bm{\pi}})\bar{\bf{H}}'_{j,n}(\sigma_g).
\end{align*}
Hence, for each $g=1,\ldots,G^*$, it follows that
\begin{align*}
    \partial^2_{\sigma_g}{\bf{H}}_n({\bm{\sigma}};{\bm{\pi}})
    &=\sum_{j=1}^J\partial_{\sigma_g}\Bigl\{{\bf{Q}}_{j,n}(g;\, {\bm{\sigma}},{\bm{\pi}})\bar{\bf{H}}'_{j,n}(\sigma_g)\Bigr\}\\
    &=\sum_{j=1}^J\Bigl\{\partial_{\sigma_g}{\bf{Q}}_{j,n}(g;\, {\bm{\sigma}},{\bm{\pi}})\Bigr\}\bar{\bf{H}}'_{j,n}(\sigma_g)+\sum_{j=1}^J{\bf{Q}}_{j,n}(g;\, {\bm{\sigma}},{\bm{\pi}})\bar{\bf{H}}''_{j,n}(\sigma_g)\\
    &=\sum_{j=1}^J{\bf{Q}}_{j,n}(g;\, {\bm{\sigma}},{\bm{\pi}})\bar{\bf{H}}''_{j,n}(\sigma_g)\\
    &\quad+\sum_{j=1}^J{\bf{Q}}_{j,n}(g;\, {\bm{\sigma}},{\bm{\pi}})\Bigl\{1-{\bf{Q}}_{j,n}(g;\, {\bm{\sigma}},{\bm{\pi}})\Bigr\}\Bigl\{\bar{\bf{H}}'_{j,n}(\sigma_g)\Bigr\}^2,
\end{align*}
so that (\ref{d2}) holds. Furthermore, for each $g_1,g_2\in\{1,\ldots,G^*\}$ with $g_1\neq g_2$, since
\begin{align*}
    \partial_{\sigma_{g_1}}{\bf{Q}}_{j,n}(g_2;\, {\bm{\sigma}},{\bm{\pi}})
    &=\pi_{g_2}\exp\Bigl\{\bar{\bf{H}}_{j,n}(\sigma_{g_2})\Bigr\}\partial_{\sigma_{g_1}}\left(\frac{1}{\displaystyle\sum_{k=1}^{G^*}\pi_{k}\exp\Bigl\{\bar{\bf{H}}_{j,n}(\sigma_k)\Bigr\}}\right)\\
    &=-{\bf{Q}}_{j,n}(g_1;\, {\bm{\sigma}},{\bm{\pi}}){\bf{Q}}_{j,n}(g_2;\, {\bm{\sigma}},{\bm{\pi}})\bar{\bf{H}}'_{j,n}(\sigma_{g_1}),
\end{align*}
we obtain
\begin{align*}
    \partial_{\sigma_{g_1}}\partial_{\sigma_{g_2}}
    {\bf{H}}_n({\bm{\sigma}};{\bm{\pi}})
    &=\sum_{j=1}^J \partial_{\sigma_{g_1}}\Bigl\{{\bf{Q}}_{j,n}(g_2;\, {\bm{\sigma}},{\bm{\pi}})\bar{\bf{H}}'_{j,n}(\sigma_{g_2})\Bigr\}\\
    &=-\sum_{j=1}^J{\bf{Q}}_{j,n}(g_1;\, {\bm{\sigma}},{\bm{\pi}}){\bf{Q}}_{j,n}(g_2;\, {\bm{\sigma}},{\bm{\pi}})\bar{\bf{H}}'_{j,n}(\sigma_{g_1})\bar{\bf{H}}'_{j,n}(\sigma_{g_2}).
\end{align*}
Therefore, $(\ref{d2d})$ holds. This completes the proof.
\end{proof}
\begin{lemma}\label{Hnormal}
Suppose that ${[\bf{A1}]}$ and ${[\bf{A2}]}$ hold. Then, under ${\rm{P}}_{\bm{z}}$, for every ${\bm{\pi}}^{(G^*)}\in\Pi^{(G^*)}$,
\begin{align*}
    \frac{1}{\sqrt{n}}\partial_{\bm{\sigma}^{(G^*)}}{\bf{H}}^{(G^*)}_n({\bm{\sigma}}^*;{\bm{\pi}}^{(G^*)})\xlongrightarrow{d}N_{G^*}\Bigl({\bm{0}}_{G^*},I({\bm{\sigma}}^*)\Bigr)
\end{align*}
as $n\longrightarrow\infty$.
\end{lemma}
\begin{proof}
Fix an arbitrary ${\bm{\pi}}\in\Pi$. By Lemma \ref{difflemma}, we have
\begin{align*}
    \frac{1}{\sqrt{n}}\partial_{\bm{\sigma}}{\bf{H}}_n({\bm{\sigma}}^*;{\bm{\pi}})
    =\sum_{j=1}^J\begin{bmatrix}
        M_{1,j,n}({\bm{\pi}})\\
        M_{2,j,n}({\bm{\pi}})\\
        \vdots\\
        M_{G^*,j,n}({\bm{\pi}})
    \end{bmatrix},
\end{align*}
where
\begin{align*}
    M_{g,j,n}({\bm{\pi}})={\bf{Q}}_{j,n}(g;\, {\bm{\sigma}}^*,{\bm{\pi}})\frac{1}{\sqrt{n}}\bar{\bf{H}}'_{j,n}(\sigma^*_g)
\end{align*}
for $g=1,\ldots,G^*$ and $j=1,\ldots,J$. For each $g=1,\ldots,G^*$, under ${\rm{P}}_{\bm{z}}$, it follows from Lemma \ref{Hbarlemma}, Corollary \ref{Qcor} and Lemma \ref{H1lemma} that
\begin{align*}
    M_{g,j,n}({\bm{\pi}})\xlongrightarrow{d}N\Biggl(0,\frac{2}{\sigma^{*2}_{g}}\Biggr)
\end{align*}
for each $j\in S_g$ and
\begin{align*}
    M_{g,j,n}({\bm{\pi}})=\sqrt{n}{\bf{Q}}_{j,n}(g;\, {\bm{\sigma}}^*,{\bm{\pi}})\biggl\{\frac{1}{n}\bar{\bf{H}}'_{j,n}(\sigma^*_g)\biggr\}\xlongrightarrow{p}0
\end{align*}
for each $j\notin S_g$. Consequently, by Slutsky's theorem, for each $g=1,\ldots,G^*$ and each $j\in S_g$, we have
\begin{align*}
    {\bf{M}}_{j,n}({\bm{\pi}})=\begin{bmatrix}
        M_{1,j,n}({\bm{\pi}})\\
        M_{2,j,n}({\bm{\pi}})\\
        \vdots\\
        M_{G^*,j,n}({\bm{\pi}})
    \end{bmatrix}\xlongrightarrow{d}N_{G^*}\Biggl({\bm{0}}_{G^*},\, \frac{2}{\sigma^{*2}_{g}}{\bm{e}}_g{\bm{e}}_g^{\top}\Biggr)
\end{align*}
under ${\rm{P}}_{\bm{z}}$, where ${\bm e}_g$ is the $g$-th unit vector in $\mathbb{R}^{G^*}$. Noting that $S_g\neq\emptyset$ for each $g=1,\ldots,G^*$ and
\begin{align*}
    \{1,\ldots,J\}=\bigcup_{g=1}^{G^*}S_g,\quad S_{g_1}\cap S_{g_2}=\emptyset\quad(g_1\neq g_2),
\end{align*}
we consider a decomposition:
\begin{align*}
    &\quad\ \frac{1}{\sqrt{n}}\partial_{\bm{\sigma}}{\bf{H}}_n({\bm{\sigma}}^*;{\bm{\pi}})\\
    &=\sum_{j\in S_1}{\bf{M}}_{j,n}({\bm{\pi}})+\sum_{j\in S_2}{\bf{M}}_{j,n}({\bm{\pi}})+\cdots+\sum_{j\in S_{G^*}}{\bf{M}}_{j,n}({\bm{\pi}}).
\end{align*}
Since ${\bf{M}}_{j,n}({\bm{\pi}})\ (j=1,\ldots,J)$ are mutually independent, under ${\rm{P}}_{\bm{z}}$, it holds that
\begin{align*}
    \frac{1}{\sqrt{n}}\partial_{\bm{\sigma}}{\bf{H}}_n({\bm{\sigma}}^*;{\bm{\pi}})
    &\xlongrightarrow{d}N_{G^*}\left({\bm{0}}_{G^*},\ 
    2\diag\Biggl[\frac{|S_1|}{\sigma_1^{*2}},\ldots,\frac{|S_{G^*}|}{\sigma_{G^*}^{*2}}\Biggr]
    \right),
\end{align*}
which completes the proof.
\end{proof}
\begin{lemma}\label{sigmaunif}
Suppose that ${[\bf{A1}]}$ and ${[\bf{A2}]}$ hold. Then, under ${\rm{P}}_{\bm{z}}$, for every ${\bm{\pi}}^{(G^*)}\in\Pi^{(G^*)}$, there exists $\rho>0$ such that 
\begin{align*}
    \sup_{{\bm{\sigma}}^{(G^*)}\in B^*(\rho)}\biggl|\frac{1}{n}\partial^2_{{\bm{\sigma}}^{(G^*)}}{\bf{H}}^{(G^*)}_n({\bm{\sigma}}^{(G^*)};{\bm{\pi}}^{(G^*)})+I({\bm{\sigma}}^{(G^*)})\biggr|\xlongrightarrow{p}0
\end{align*}
as $n\longrightarrow\infty$.
\end{lemma}
\begin{proof}
Fix an arbitrary $\bm{\pi}\in\Pi$. By Lemma \ref{Qsuplemma}, there exists $\rho>0$ such that for each $g=1,\ldots,G^*$,
\begin{align}
    n\sup_{{\bm{\sigma}}\in B^*(\rho)}\Bigl|{\bf{Q}}_{j,n}(g;\, {\bm{\sigma}},{\bm{\pi}})-1\Bigr|\xlongrightarrow{p}0\label{Qj1}
\end{align}
for each $j\in S_g$ and 
\begin{align}
    n\sup_{{\bm{\sigma}}\in B^*(\rho)}{\bf{Q}}_{j,n}(g;\, {\bm{\sigma}},{\bm{\pi}})\xlongrightarrow{p}0\label{Qj2}
\end{align}
for each $j\notin S_g$. Since
\begin{align*}
    &\quad\ \sup_{{\bm{\sigma}}\in B^*(\rho)}\biggl|\frac{1}{n}\partial^2_{\bm{\sigma}}{\bf{H}}_n({\bm{\sigma}};{\bm{\pi}})+I({\bm{\sigma}})\biggr|\\
    &\leq\sum_{g_1=1}^{G^*}
    \sum_{g_2=1}^{G^*}\sup_{{\bm{\sigma}}\in B^*(\rho)}\biggl|\frac{1}{n}\partial_{\sigma_{g_1}}\partial_{\sigma_{g_2}}{\bf{H}}_n({\bm{\sigma}};{\bm{\pi}})+I({\bm{\sigma}})_{g_1g_2}\biggr|
\end{align*}
and $I({\bm{\sigma}})$ is diagonal, under ${\rm{P}}_{\bm{z}}$, it suffices to prove that
\begin{align}
    \sup_{{\bm{\sigma}}\in B^*(\rho)}\biggl|\frac{1}{n}\partial^2_{\sigma_{g}}{\bf{H}}_n({\bm{\sigma}};{\bm{\pi}})+I({\bm{\sigma}})_{gg}\biggr|\xlongrightarrow{p}0\label{K1}
\end{align}
for each $g=1,\ldots,G^*$, and 
\begin{align}
    \sup_{{\bm{\sigma}}\in B^*(\rho)}\biggl|\frac{1}{n}\partial_{\sigma_{g_1}}\partial_{\sigma_{g_2}}{\bf{H}}_n({\bm{\sigma}};{\bm{\pi}})\biggr|\xlongrightarrow{p}0\label{K2}
\end{align}
for each $g_1,g_2\in\{1,\ldots,G^*\}$ with $g_1\neq g_2$. First, we prove \eqref{K1}. For each $g=1,\ldots,G^*$, a decomposition is given by
\begin{align*}
    \frac{1}{n}\partial^2_{\sigma_g}{\bf{H}}_n({\bm{\sigma}};{\bm{\pi}})
    &=\frac{1}{n}\sum_{j\in S_g}\partial^2_{\sigma_g}\log\Biggl(\sum_{k=1}^{G^*}\pi_k\exp\Bigl\{\bar{\bf{H}}_{j,n}(\sigma_k)\Bigr\}\Biggr)\\
    &\qquad+\frac{1}{n}
    \sum_{j\notin S_g}
    \partial^2_{\sigma_g}
    \log\Biggl(\sum_{k=1}^{G^*}
    \pi_k\exp\Bigl\{\bar{\bf{H}}_{j,n}(\sigma_k)\Bigr\}
    \Biggr).
\end{align*}
For each $g=1,\ldots,G^*$, since
\begin{align*}
    I({\bm{\sigma}})_{gg}=-\sum_{j\in S_g}\bar{\mathbb{H}}''_{z_j}(\sigma_g),
\end{align*}
we see from Lemma \ref{difflemma} that
\begin{align*}
    &\quad\ \frac{1}{n}\sum_{j\in S_g}\partial^2_{\sigma_g}\log\Biggl(\sum_{k=1}^{G^*}\pi_k\exp\Bigl\{\bar{\bf{H}}_{j,n}(\sigma_k)\Bigr\}\Biggr)+I({\bm{\sigma}})_{gg}\\
    &=\sum_{j\in S_g}{\bf{Q}}_{j,n}(g;\, {\bm{\sigma}},{\bm{\pi}})\biggl\{\frac{1}{n}\bar{\bf{H}}''_{j,n}(\sigma_g)-\bar{\mathbb{H}}''_{z_j}(\sigma_g)\biggr\}\\
    &\quad-\sum_{j\in S_g}\Bigl\{1-{\bf{Q}}_{j,n}(g;\, {\bm{\sigma}},{\bm{\pi}})\Bigr\}\bar{\mathbb{H}}''_{z_j}(\sigma_g)\\
    &\quad+\sum_{j\in S_g}{\bf{Q}}_{j,n}(g;\, {\bm{\sigma}},{\bm{\pi}})n\Bigl\{1-{\bf{Q}}_{j,n}(g;\, {\bm{\sigma}},{\bm{\pi}})\Bigr\}\biggl\{\frac{1}{n}\bar{\bf{H}}'_{j,n}(\sigma_{g})\biggr\}^2.
\end{align*}
Hence, for each $g=1,\ldots,G^*$, we have
\begin{align*}
     &\quad\ \left|\frac{1}{n}\sum_{j\in S_g}\partial^2_{\sigma_g}\log\Biggl(\sum_{k=1}^{G^*}\pi_k\exp\Bigl\{\bar{\bf{H}}_{j,n}(\sigma_k)\Bigr\}\Biggr)+I({\bm{\sigma}})_{gg}\right|\\
     &\leq \sum_{j\in S_g}\sup_{{\bm{\sigma}}\in B^*(\rho)}{\bf{Q}}_{j,n}(g;\, {\bm{\sigma}},{\bm{\pi}})\sup_{\sigma\in[1/d,d]}\biggl|\frac{1}{n}\bar{\bf{H}}''_{j,n}(\sigma)-\bar{\mathbb{H}}''_{z_j}(\sigma)\biggr|\\
     &+\sum_{j\in S_g}\sup_{{\bm{\sigma}}\in B^*(\rho)}\Bigl|{\bf{Q}}_{j,n}(g;\, {\bm{\sigma}},{\bm{\pi}})-1\Bigr|\sup_{\sigma\in[1/d,d]}
     \Bigl|\bar{\mathbb{H}}''_{z_j}(\sigma)\Bigr|\\
     &+\sum_{j\in S_g}\sup_{{\bm{\sigma}}\in B^*(\rho)}{\bf{Q}}_{j,n}(g;\, {\bm{\sigma}},{\bm{\pi}})\sup_{{\bm{\sigma}}\in B^*(\rho)}n\Bigl|{\bf{Q}}_{j,n}(g;\, {\bm{\sigma}},{\bm{\pi}})-1\Bigr|
     \sup_{\sigma\in[1/d,d]}\biggl|\frac{1}{n}\bar{\bf{H}}'_{j,n}(\sigma)\biggr|^2
\end{align*}
for every ${\bm{\sigma}}\in B^*(\rho)$, so that Lemma \ref{Hbarlemma} and \eqref{Qj1} imply
\begin{align}
    \sup_{{\bm{\sigma}}\in B^*(\rho)}\left|\frac{1}{n}\sum_{j\in S_g}\partial^2_{\sigma_g}\log\Biggl(\sum_{k=1}^{G^*}\pi_k\exp\Bigl\{\bar{\bf{H}}_{j,n}(\sigma_k)\Bigr\}\Biggr)+I({\bm{\sigma}})_{gg}\right|\xlongrightarrow{p}0\label{supH2-1}
\end{align}
under ${\rm{P}}_{\bm{z}}$. Similarly, for each $g=1,\ldots,G^*$, as it holds from Lemma \ref{difflemma} that
\begin{align*}
     &\quad\left|\frac{1}{n}\sum_{j\notin S_g}\partial^2_{\sigma_g}\log\Biggl(\sum_{k=1}^{G^*}\pi_k\exp\Bigl\{\bar{\bf{H}}_{j,n}(\sigma_k)\Bigr\}\Biggr)\right|\\
     &\leq\sum_{j\notin S_g}\sup_{{\bm{\sigma}}\in B^*(\rho)}{\bf{Q}}_{j,n}(g;\, {\bm{\sigma}},{\bm{\pi}})\sup_{\sigma\in[1/d,d]}\biggl|\frac{1}{n}\bar{\bf{H}}''_{j,n}(\sigma)\biggr|\\
     &+\sum_{j\notin S_g}n\sup_{{\bm{\sigma}}\in B^*(\rho)}{\bf{Q}}_{j,n}(g;\, {\bm{\sigma}},{\bm{\pi}})\sup_{{\bm{\sigma}}\in B^*(\rho)}\Bigl\{1-{\bf{Q}}_{j,n}(g;\, {\bm{\sigma}},{\bm{\pi}})\Bigr\}
     \sup_{\sigma\in[1/d,d]}\biggl|\frac{1}{n}\bar{\bf{H}}'_{j,n}(\sigma)\biggr|^2
\end{align*}
for every ${\bm{\sigma}}\in B^*(\rho)$, by Lemma \ref{Hbarlemma} and \eqref{Qj2}, it is shown that
\begin{align}
    \sup_{{\bm{\sigma}}\in B^*(\rho)}\left|\frac{1}{n}\sum_{j\notin S_g}\partial^2_{\sigma_g}\log\Biggl(\sum_{k=1}^{G^*}\pi_k\exp\Bigl\{\bar{\bf{H}}_{j,n}(\sigma_k)\Bigr\}
    \Biggr)\right|\xlongrightarrow{p}0\label{supH2-2}
\end{align}
under ${\rm{P}}_{\bm{z}}$. Therefore, by \eqref{supH2-1} and \eqref{supH2-2}, we obtain
\begin{align*}
    &\quad\ \sup_{{\bm{\sigma}}\in B^*(\rho)}\biggl|\frac{1}{n}\partial^2_{\sigma_{g}}{\bf{H}}_n({\bm{\sigma}};{\bm{\pi}})+I({\bm{\sigma}})_{gg}\biggr|\\
    &\leq\sup_{{\bm{\sigma}}\in B^*(\rho)}\left|\frac{1}{n}\sum_{j\in S_g}\partial^2_{\sigma_g}\log\Biggl(\sum_{k=1}^{G^*}\pi_k\exp\Bigl\{\bar{\bf{H}}_{j,n}(\sigma_k)\Bigr\}\Biggr)+I({\bm{\sigma}})_{gg}\right|\\
    &\qquad\quad+\sup_{{\bm{\sigma}}\in B^*(\rho)}\left|\frac{1}{n}\sum_{j\notin S_g}\partial^2_{\sigma_g}\log\Biggl(\sum_{k=1}^{G^*}\pi_k\exp\Bigl\{\bar{\bf{H}}_{j,n}(\sigma_k)\Bigr\}
    \Biggr)\right|\xlongrightarrow{p}0
\end{align*}
under ${\rm{P}}_{\bm{z}}$, which yields \eqref{K1}. Next, we prove \eqref{K2}. For each $g_1,g_2\in\{1,\ldots,G^*\}$ with $g_1\neq g_2$, we consider a decomposition:
\begin{align*}
    &\quad\
    \frac{1}{n}
    \sum_{j=1}^J
    \partial_{\sigma_{g_1}}\partial_{\sigma_{g_2}}
    \log\Biggl(
    \sum_{k=1}^{G^*}
    \pi_k
    \exp\Bigl\{\bar{\bf{H}}_{j,n}(\sigma_k)\Bigr\}
    \Biggr)\\
    &=
    \frac{1}{n}
    \sum_{j\in S_{g_1}}
    \partial_{\sigma_{g_1}}\partial_{\sigma_{g_2}}
    \log\Biggl(
    \sum_{k=1}^{G^*}
    \pi_k
    \exp\Bigl\{\bar{\bf{H}}_{j,n}(\sigma_k)\Bigr\}
    \Biggr)\\
    &\qquad+
    \frac{1}{n}
    \sum_{j\in S_{g_2}}
    \partial_{\sigma_{g_1}}\partial_{\sigma_{g_2}}
    \log\Biggl(
    \sum_{k=1}^{G^*}
    \pi_k
    \exp\Bigl\{\bar{\bf{H}}_{j,n}(\sigma_k)\Bigr\}
    \Biggr)\\
    &\qquad+
    \frac{1}{n}
    \sum_{j\notin S_{g_1}\cup S_{g_2}}
    \partial_{\sigma_{g_1}}\partial_{\sigma_{g_2}}
    \log\Biggl(
    \sum_{k=1}^{G^*}
    \pi_k
    \exp\Bigl\{\bar{\bf{H}}_{j,n}(\sigma_k)\Bigr\}
    \Biggr).
\end{align*}
For each $g_1,g_2\in\{1,\ldots,G^*\}$ with $g_1\neq g_2$, since it follows from Lemma \ref{difflemma} that
\begin{align*}
    &\quad\ \left|
    \frac{1}{n}
    \sum_{j\in S_{g_1}}
    \partial_{\sigma_{g_1}}\partial_{\sigma_{g_2}}
    \log\Biggl(
    \sum_{k=1}^{G^*}
    \pi_k
    \exp\Bigl\{\bar{\bf{H}}_{j,n}(\sigma_k)\Bigr\}
    \Biggr)
    \right|\\
    &\leq\sum_{j\in S_{g_1}}\sup_{{\bm{\sigma}}\in B^*(\rho)}{\bf{Q}}_{j,n}(g_1;\, {\bm{\sigma}},{\bm{\pi}})\\
    &\qquad\qquad\qquad\times n\sup_{{\bm{\sigma}}\in B^*(\rho)}{\bf{Q}}_{j,n}(g_2;\, {\bm{\sigma}},{\bm{\pi}})\sup_{\sigma\in[1/d,d]}
    \biggl|\frac{1}{n}\bar{\bf{H}}^{\prime}_{j,n}(\sigma)\biggr|^2,
\end{align*}
it holds from Lemma \ref{Hbarlemma} and \eqref{Qj1}--\eqref{Qj2} that
\begin{align}
    \sup_{{\bm{\sigma}}\in B^*(\rho)}
    \left|\frac{1}{n}\sum_{j\in S_{g_1}}\partial_{\sigma_{g_1}}\partial_{\sigma_{g_2}}
    \log\Biggl(\sum_{k=1}^{G^*}\pi_k\exp\Bigl\{\bar{\bf{H}}_{j,n}(\sigma_k)\Bigr\}
    \Biggr)\right|\xlongrightarrow{p}0 \label{supH2d-1}
\end{align}
under ${\rm{P}}_{\bm{z}}$. Similarly, for each $g_1,g_2\in\{1,\ldots,G^*\}$ with $g_1\neq g_2$, we have
\begin{align}
    \sup_{{\bm{\sigma}}\in B^*(\rho)}
    \left|
    \frac{1}{n}
    \sum_{j\in S_{g_2}}
    \partial_{\sigma_{g_1}}\partial_{\sigma_{g_2}}
    \log\Biggl(
    \sum_{k=1}^{G^*}
    \pi_k
    \exp\Bigl\{\bar{\bf{H}}_{j,n}(\sigma_k)\Bigr\}
    \Biggr)
    \right|
    \xlongrightarrow{p}0
    \label{supH2d-2}
\end{align}
and
\begin{align}
    \sup_{{\bm{\sigma}}\in B^*(\rho)}
    \left|
    \frac{1}{n}
    \sum_{j\notin S_{g_1}\cup S_{g_2}}
    \partial_{\sigma_{g_1}}\partial_{\sigma_{g_2}}
    \log\Biggl(
    \sum_{k=1}^{G^*}
    \pi_k
    \exp\Bigl\{\bar{\bf{H}}_{j,n}(\sigma_k)\Bigr\}
    \Biggr)
    \right|
    \xlongrightarrow{p}0
    \label{supH2d-3}
\end{align}
under ${\rm{P}}_{\bm{z}}$. Therefore, by \eqref{supH2d-1}--\eqref{supH2d-3}, for each $g_1\neq g_2$, we obtain
\begin{align*}
     &\quad\ \sup_{{\bm{\sigma}}\in B^*(\rho)}
     \left|
     \frac{1}{n}
     \sum_{j=1}^J
     \partial_{\sigma_{g_1}}\partial_{\sigma_{g_2}}
     \log\Biggl(
     \sum_{k=1}^{G^*}
     \pi_k
     \exp\Bigl\{\bar{\bf{H}}_{j,n}(\sigma_k)\Bigr\}
     \Biggr)
     \right|\\
     &\leq\sup_{{\bm{\sigma}}\in B^*(\rho)}
     \left|
     \frac{1}{n}
     \sum_{j\in S_{g_1}}
     \partial_{\sigma_{g_1}}\partial_{\sigma_{g_2}}
     \log\Biggl(
     \sum_{k=1}^{G^*}
     \pi_k
     \exp\Bigl\{\bar{\bf{H}}_{j,n}(\sigma_k)\Bigr\}
     \Biggr)
     \right|\\
     &\quad+\sup_{{\bm{\sigma}}\in B^*(\rho)}
     \left|
     \frac{1}{n}
     \sum_{j\in S_{g_2}}
     \partial_{\sigma_{g_1}}\partial_{\sigma_{g_2}}
     \log\Biggl(
     \sum_{k=1}^{G^*}
     \pi_k
     \exp\Bigl\{\bar{\bf{H}}_{j,n}(\sigma_k)\Bigr\}
     \Biggr)
     \right|\\
     &\quad+\sup_{{\bm{\sigma}}\in B^*(\rho)}
     \left|
     \frac{1}{n}
     \sum_{j\notin S_{g_1}\cup S_{g_2}}
     \partial_{\sigma_{g_1}}\partial_{\sigma_{g_2}}
     \log\Biggl(
     \sum_{k=1}^{G^*}
     \pi_k
     \exp\Bigl\{\bar{\bf{H}}_{j,n}(\sigma_k)\Bigr\}
     \Biggr)
     \right|\xlongrightarrow{p}0
\end{align*}
under ${\rm{P}}_{\bm{z}}$, which gives \eqref{K2}. This completes the proof.
\end{proof}
\begin{proposition}\label{asymprop}
Suppose that ${[\bf A1]}$ and ${[\bf A2]}$ hold. Then, under ${\rm{P}}_{\bm{z}}$, for every ${\bm{\pi}}^{(G^*)}\in\Pi^{(G^*)}$,
\begin{align*}
    \sqrt{n}\Bigl(\hat{\bm{\sigma}}^{*}_n({\bm{\pi}^{(G^*)}})-{\bm{\sigma}}^*\Bigr)\xlongrightarrow{d}N_{G^*}\Bigl({\bm{0}}_{G^*},I({\bm{\sigma}}^*)^{-1}\Bigr)
\end{align*}
as $n\longrightarrow\infty$.
\end{proposition}
\begin{proof}
Fix an arbitrary ${\bm{\pi}}\in\Pi$. By Lemma \ref{sigmaunif}, there exists $\rho>0$ such that $B^*(\rho)\subset\Sigma^{(G^*)}$ and
\begin{align}
    \sup_{{\bm{\sigma}}\in B^*(\rho)}\biggl|\frac{1}{n}\partial^2_{\bm{\sigma}}{\bf{H}}_n({\bm{\sigma}};{\bm{\pi}})+I({\bm{\sigma}})\biggr|\xlongrightarrow{p}0 \label{delta1}
\end{align}
under ${\rm{P}}_{\bm{z}}$. Proposition \ref{consisprop} implies that there exists a positive sequence $\{\rho_n\}_{n\in\mathbb{N}}$ such that $\rho_n<\rho$ for every $n\in\mathbb{N}$, $\rho_n\longrightarrow 0$ as $n\longrightarrow\infty$ and 
\begin{align}
    {\rm{P}}_{\bm{z}}\Bigl(D_{1,n}({\bm{\pi}})\Bigr)\longrightarrow 1 \label{D1}
\end{align}
as $n\longrightarrow\infty$, where
\begin{align*}
    D_{1,n}({\bm{\pi}})=\Bigl\{\hat{\bm{\sigma}}^{*}_n({\bm{\pi}})\in B^*(\rho_n)\Bigr\}.
\end{align*}
By Taylor's theorem, it follows that
\begin{align*}
    0&=\frac{1}{\sqrt{n}}\partial_{\bm{\sigma}}{\bf{H}}_n(\hat{\bm{\sigma}}_n^*({\bm{\pi}});{\bm{\pi}})\\
    &=\frac{1}{\sqrt{n}}\partial_{\bm{\sigma}}{\bf{H}}_n({\bm{\sigma}}^*;{\bm{\pi}})-J_n({\bm{\pi}})\sqrt{n}\bigl(\hat{\bm{\sigma}}^{*}_n({\bm{\pi}})-{\bm{\sigma}}^*\bigr)
\end{align*}
on $D_{1,n}({\bm{\pi}})$
and hence
\begin{align}
    J_n({\bm{\pi}})\sqrt{n}\bigl(\hat{\bm{\sigma}}^{*}_n({\bm{\pi}})-{\bm{\sigma}}^*\bigr)=\frac{1}{\sqrt{n}}\partial_{\bm{\sigma}}{\bf{H}}_n({\bm{\sigma}}^*;{\bm{\pi}}) \label{Jneq}
\end{align}
on $D_{1,n}({\bm{\pi}})$, where
\begin{align*}
    \tilde{\bm{\sigma}}^*_{\lambda,n}({\bm{\pi}})
    ={\bm{\sigma}}^*+\lambda\bigl(\hat{\bm{\sigma}}^{*}_n({\bm{\pi}})-{\bm{\sigma}}^*\bigr)
\end{align*}
for $\lambda\in [0,1]$, and
\begin{align*}
    J_n({\bm{\pi}})=\begin{cases}
    \displaystyle
    -\frac{1}{n}\int_0^1
    \partial^2_{\bm{\sigma}}{\bf{H}}_{n}
    \bigl(\tilde{\bm{\sigma}}^*_{\lambda,n}({\bm{\pi}});{\bm{\pi}}\bigr)d\lambda
    & \text{on } D_{1,n}({\bm{\pi}}),\\[2mm]
    \displaystyle \mathbb{I}_{G^*}
    & \text{on } D_{1,n}({\bm{\pi}})^c.
\end{cases}
\end{align*}
For every $\lambda\in[0,1]$, on $D_{1,n}({\bm{\pi}})$, since 
\begin{align*}
    \bigl|\tilde{\bm{\sigma}}^*_{\lambda,n}({\bm{\pi}})-{\bm{\sigma}}^*\bigr|
    =\lambda\bigl|\hat{\bm{\sigma}}^{*}_n({\bm{\pi}})-{\bm{\sigma}}^*\bigr|
    \leq\bigl|\hat{\bm{\sigma}}^{*}_n({\bm{\pi}})-{\bm{\sigma}}^*\bigr|\leq\rho_n,
\end{align*}
we have $\tilde{\bm{\sigma}}^*_{\lambda,n}({\bm{\pi}})\in B^*(\rho_n)$, so that
\begin{align*}
    \bigl|J_n({\bm{\pi}})-I({\bm{\sigma}}^*)\bigr|&=\Biggl|\frac{1}{n}\int_0^1\partial^2_{\bm{\sigma}}{\bf{H}}_{n}(\tilde{\bm{\sigma}}^*_{\lambda,n}({\bm{\pi}});{\bm{\pi}})d\lambda
    +I({\bm{\sigma}}^*)\Biggr|\\
    &\leq\sup_{{\bm{\sigma}}\in B^*(\rho_n)}\biggl|\frac{1}{n}\partial^2_{\bm{\sigma}}{\bf{H}}_n({\bm{\sigma}};{\bm{\pi}})+I({\bm{\sigma}}^*)\biggr|\\
    &\leq\sup_{{\bm{\sigma}}\in B^*(\rho)}\biggl|\frac{1}{n}\partial^2_{\bm{\sigma}}{\bf{H}}_n({\bm{\sigma}};{\bm{\pi}})+I({\bm{\sigma}})\biggr|+\sup_{{\bm{\sigma}}\in B^*(\rho_n)}\bigl|I({\bm{\sigma}})-I({\bm{\sigma}}^*)\bigr|
\end{align*}
on $D_{1,n}({\bm{\pi}})$. Hence, for every $\varepsilon>0$, it holds from \eqref{delta1}--\eqref{D1} and the continuity of $I$ that
\begin{align*}
    &\quad\ {\rm{P}}_{\bm{z}}\biggl(\bigl|J_n({\bm{\pi}})-I({\bm{\sigma}}^*)\bigr|>\varepsilon\biggr)\\
    &\leq {\rm{P}}_{\bm{z}}\biggl(
    \Bigl\{\bigl|J_n({\bm{\pi}})-I({\bm{\sigma}}^*)\bigr|>\varepsilon\Bigr\}
    \cap D_{1,n}({\bm{\pi}})\biggr)\\
    &\qquad+{\rm{P}}_{\bm{z}}\biggl(
    \Bigl\{\bigl|J_n({\bm{\pi}})-I({\bm{\sigma}}^*)\bigr|>\varepsilon\Bigr\}
    \cap D^c_{1,n}({\bm{\pi}})\biggr)\\
    &\leq{\rm{P}}_{\bm{z}}\Biggl(\sup_{{\bm{\sigma}}\in B^*(\rho)}\biggl|\frac{1}{n}\partial^2_{\bm{\sigma}}{\bf{H}}_n({\bm{\sigma}};{\bm{\pi}})+I({\bm{\sigma}})\biggr|>\frac{\varepsilon}{2}\Biggr)\\
    &\qquad+{\rm{P}}_{\bm{z}}\Biggl(\sup_{{\bm{\sigma}}\in B^*(\rho_n)}\bigl|I({\bm{\sigma}})-I({\bm{\sigma}}^*)\bigr|>\frac{\varepsilon}{2}\Biggr)+{\rm{P}}_{\bm{z}}\Bigl(D^c_{1,n}({\bm{\pi}})\Bigr)\longrightarrow 0
\end{align*}
as $n\longrightarrow\infty$, which yields
\begin{align*}
    J_n({\bm{\pi}})\xlongrightarrow{p}I({\bm{\sigma}}^*)
\end{align*}
under ${\rm{P}}_{\bm{z}}$. Define
\begin{align*}
    D_{2,n}({\bm{\pi}})=\Bigl\{J_n({\bm{\pi}})>0\Bigr\},\quad \tilde{J}_n({\bm{\pi}})=\begin{cases}
        J_n({\bm{\pi}}) & \text{on } D_{2,n}({\bm{\pi}}),\\[2mm]
        \displaystyle \mathbb{I}_{G^*} & \text{on } D_{2,n}({\bm{\pi}})^c.
    \end{cases}
\end{align*}
By the positive definiteness of $I({\bm{\sigma}}^*)$, it is shown that
\begin{align}
    {\rm{P}}_{\bm{z}}\Bigl(D_{2,n}({\bm{\pi}})\Bigr)\longrightarrow 1 \label{D2}
\end{align}
as $n\longrightarrow\infty$, and 
\begin{align*}
    \tilde{J}_n({\bm{\pi}})\xlongrightarrow{p}I({\bm{\sigma}}^*)
\end{align*}
under ${\rm{P}}_{\bm{z}}$. Combining these results with Lemma \ref{Hnormal}, we obtain
\begin{align*}
    \tilde{J}_n({\bm{\pi}})^{-1}\frac{1}{\sqrt{n}}\partial_{\bm{\sigma}}{\bf{H}}_n({\bm{\sigma}}^*;{\bm{\pi}})\xlongrightarrow{d}\mathcal{Z}_1\sim N_{G^*}\Bigl({\bm{0}}_{G^*},I({\bm{\sigma}}^*)^{-1}\Bigr)
\end{align*}
under ${\rm{P}}_{\bm{z}}$. Let $\mathcal{C}\subset\mathbb{R}^{G^*}$ be an arbitrary closed set. By \eqref{Jneq}, we have
\begin{align*}
    &\quad\ {\rm{P}}_{\bm{z}}\biggl(\sqrt{n}\bigl(\hat{\bm{\sigma}}^{*}_n({\bm{\pi}})-{\bm{\sigma}}^*\bigr)\in\mathcal{C}\biggr)\\
    &={\rm{P}}_{\bm{z}}\Biggl(\biggl\{\tilde{J}_n({\bm{\pi}})^{-1}\frac{1}{\sqrt{n}}\partial_{\bm{\sigma}}{\bf{H}}_n({\bm{\sigma}}^*;{\bm{\pi}})\in\mathcal{C}\biggr\}\cap \bigl(D_{1,n}({\bm{\pi}})\cap D_{2,n}({\bm{\pi}})\bigr)\Biggr)\\
    &\qquad+{\rm{P}}_{\bm{z}}\Biggl(\biggl\{\sqrt{n}\bigl(\hat{\bm{\sigma}}^{*}_n({\bm{\pi}})-{\bm{\sigma}}^*\bigr)\in\mathcal{C}\biggr\}
    \cap\bigl(D_{1,n}({\bm{\pi}})\cap D_{2,n}({\bm{\pi}})\bigr)^c\Biggr)\\
    &\leq{\rm{P}}_{\bm{z}}\biggl(\tilde{J}_n({\bm{\pi}})^{-1}\frac{1}{\sqrt{n}}\partial_{\bm{\sigma}}{\bf{H}}_n({\bm{\sigma}}^*;{\bm{\pi}})\in\mathcal{C}\biggr)+{\rm{P}}_{\bm{z}}\Bigl(\bigl(D_{1,n}({\bm{\pi}})\cap D_{2,n}({\bm{\pi}})\bigr)^c\Bigr),
\end{align*}
so that \eqref{D1} and \eqref{D2} give
\begin{align*}
    \limsup_{n\longrightarrow\infty}{\rm{P}}_{\bm{z}}\biggl(\sqrt{n}\bigl(\hat{\bm{\sigma}}^{*}_n({\bm{\pi}})-{\bm{\sigma}}^*\bigr)\in\mathcal{C}\biggr)\leq{\rm{P}}_{\bm{z}}\Bigl(\mathcal{Z}_1\in\mathcal{C}\Bigr).
\end{align*}
Hence, it holds that
\begin{align*}
    \sqrt{n}\bigl(\hat{\bm{\sigma}}^{*}_n({\bm{\pi}})-{\bm{\sigma}}^*\bigr)\xlongrightarrow{d}\mathcal{Z}_1,
\end{align*}
which completes the proof.
\end{proof}
\begin{proof}[\textbf{Proof of Theorem \ref{consis}}]
Fix an arbitrary ${\bm{\pi}}_M\in\Pi_M$. On the event
\begin{align*}
    E_n({\bm\pi}_M)=\Bigl\{\hat{G}_n({\bm{\pi}}_M)=G^*\Bigr\},
\end{align*}
since
\begin{align*}
    {\bf{L}}^{(\hat{G}_n)}_{n}({\bm{\sigma}}^{(\hat{G}_n)};{\bm{\pi}}^{(\hat{G}_n)})={\bf{L}}^{(G^*)}_{n}({\bm{\sigma}}^{(G^*)};{\bm{\pi}}^{(G^*)}),
\end{align*}
$\hat{\bm\sigma}_n({\bm\pi}_M)$ is a maximizer of ${\bf L}_n^{(G^*)}$. Hence, we can choose $\hat{\bm{\sigma}}_n^{*}({\bm{\pi}}^{(G^*)})$ among maximizers of ${\bf L}_n^{(G^*)}$ such that
\begin{align*}
    \hat{\bm{\sigma}}_n({\bm{\pi}}_M)=\hat{\bm{\sigma}}_n^{*}({\bm{\pi}}^{(G^*)})
\end{align*}
on $E_n({\bm\pi}_M)$. By Theorem \ref{Gconsis} and Proposition \ref{consisprop}, for all $\varepsilon>0$, we have
\begin{align*}
    &\quad\  {\rm{P}}_{\bm{z}}\left(\, 
    \left|\begin{bmatrix}
    \hat{\bm{\sigma}}_n({\bm{\pi}}_M)\\
    {\bm{0}}_{M-\hat{G}_n({\bm{\pi}}_M)}
    \end{bmatrix}-
    \begin{bmatrix}
    {\bm{\sigma}}^*\\
    {\bm{0}}_{M-G^*}
    \end{bmatrix}
    \right|>\varepsilon
    \right)\\
    &={\rm{P}}_{\bm{z}}\left(
    \left\{\left|\begin{bmatrix}
    \hat{\bm{\sigma}}^*_n({\bm{\pi}}^{(G^*)})\\
    {\bm{0}}_{M-G^*}
    \end{bmatrix}-
    \begin{bmatrix}
    {\bm{\sigma}}^*\\
    {\bm{0}}_{M-G^*}
    \end{bmatrix}
    \right|>\varepsilon
    \right\}
    \cap E_n({\bm\pi}_M)\right)\\
    &\qquad+{\rm{P}}_{\bm{z}}\left(
    \left\{\left|\begin{bmatrix}
    \hat{\bm{\sigma}}_n({\bm{\pi}}_M)\\
    {\bm{0}}_{M-\hat{G}_n({\bm{\pi}}_M)}
    \end{bmatrix}-
    \begin{bmatrix}
    {\bm{\sigma}}^*\\
    {\bm{0}}_{M-G^*}
    \end{bmatrix}
    \right|>\varepsilon
    \right\}\cap E_n({\bm\pi}_M)^c\right)\\
    &\leq{\rm{P}}_{\bm{z}}\biggl(
    \bigl|\hat{\bm{\sigma}}^*_n({\bm{\pi}}^{(G^*)})-{\bm{\sigma}}^*\bigr|>\varepsilon
    \biggr)+{\rm{P}}_{\bm{z}}\Bigl(E_n({\bm\pi}_M)^c\Bigr)\longrightarrow 0
\end{align*}
as $n\longrightarrow\infty$, which yields
\begin{align*}
    \begin{bmatrix}
    \hat{\bm{\sigma}}_n({\bm{\pi}}_M)\\
    {\bm{0}}_{M-\hat{G}_n({\bm{\pi}}_M)}
    \end{bmatrix}
    \xlongrightarrow{p}
    \begin{bmatrix}
    {\bm{\sigma}}^*\\
    {\bm{0}}_{M-G^*}
    \end{bmatrix}
\end{align*}
under ${\rm{P}}_{\bm{z}}$. Proposition \ref{asymprop} implies
\begin{align}
    \sqrt{n}\Biggl(\,
    \begin{bmatrix}
    \hat{\bm{\sigma}}^*_n({\bm{\pi}}^{(G^*)})\\
    {\bm{0}}_{M-G^*}
    \end{bmatrix}
    -\begin{bmatrix}
    {\bm{\sigma}}^*\\
    {\bm{0}}_{M-G^*}
    \end{bmatrix}\, \Biggr)\xlongrightarrow{d}\mathcal{Z}_2\label{asymprop2}
\end{align}
under ${\rm{P}}_{\bm{z}}$, where
\begin{align*}
    \mathcal{Z}_2\sim N_{M}\left(
    {\bm{0}}_M,\,
    \begin{bmatrix}
    I({\bm{\sigma}}^*)^{-1} & O_{G^*,M-G^*}\\
    O_{M-G^*,G^*} & O_{M-G^*,M-G^*}
    \end{bmatrix}
    \right).
\end{align*}
Let $\mathcal{C}\subset\mathbb{R}^M$ be an arbitrary closed set. Since 
\begin{align*}
    &\quad\ {\rm{P}}_{\bm{z}}\Biggl(\sqrt{n}\Biggl(\,
    \begin{bmatrix}
    \hat{\bm{\sigma}}_n({\bm{\pi}}_M)\\
    {\bm{0}}_{M-\hat{G}_n({\bm{\pi}}_M)}
    \end{bmatrix}
    -\begin{bmatrix}
    {\bm{\sigma}}^*\\
    {\bm{0}}_{M-G^*}
    \end{bmatrix}\, \Biggr)\in\mathcal{C}\Biggr)\\
    &={\rm{P}}_{\bm{z}}\Biggl(\Biggl\{\sqrt{n}\Biggl(\,
    \begin{bmatrix}
    \hat{\bm{\sigma}}^*_n({\bm{\pi}}^{(G^*)})\\
    {\bm{0}}_{M-G^*}
    \end{bmatrix}
    -\begin{bmatrix}
    {\bm{\sigma}}^*\\
    {\bm{0}}_{M-G^*}
    \end{bmatrix}\, \Biggr)\in \mathcal{C}\Biggr\}\cap E_n({\bm\pi}_M)\Biggr)\\
    &\qquad+{\rm{P}}_{\bm{z}}\Biggl(\Biggl\{\sqrt{n}\Biggl(\,
    \begin{bmatrix}
    \hat{\bm{\sigma}}_n({\bm{\pi}}_M)\\
    {\bm{0}}_{M-\hat{G}_n({\bm{\pi}}_M)}
    \end{bmatrix}
    -\begin{bmatrix}
    {\bm{\sigma}}^*\\
    {\bm{0}}_{M-G^*}
    \end{bmatrix}\, \Biggr)\in\mathcal{C}\Biggr\}\cap E_n({\bm\pi}_M)^c\Biggr)\\
    &\leq{\rm{P}}_{\bm{z}}\Biggl(\sqrt{n}\Biggl(\,
    \begin{bmatrix}
    \hat{\bm{\sigma}}^*_n({\bm{\pi}}^{(G^*)})\\
    {\bm{0}}_{M-G^*}
    \end{bmatrix}
    -\begin{bmatrix}
    {\bm{\sigma}}^*\\
    {\bm{0}}_{M-G^*}
    \end{bmatrix}\, \Biggr)\in\mathcal{C}\Biggr)
    +{\rm{P}}_{\bm{z}}\Bigl(E_n({\bm\pi}_M)^c\Bigr),
\end{align*}
Theorem \ref{Gconsis} and (\ref{asymprop2}) imply
\begin{align*}
    &\quad\ \limsup_{n\longrightarrow\infty}{\rm{P}}_{\bm{z}}\Biggl(\sqrt{n}\Biggl(\,
    \begin{bmatrix}
    \hat{\bm{\sigma}}_n({\bm{\pi}}_M)\\
    {\bm{0}}_{M-\hat{G}_n({\bm{\pi}}_M)}
    \end{bmatrix}
    -\begin{bmatrix}
    {\bm{\sigma}}^*\\
    {\bm{0}}_{M-G^*}
    \end{bmatrix}\, \Biggr)\in\mathcal{C}\Biggr)\leq{\rm{P}}_{\bm{z}}\Bigl(\mathcal{Z}_2\in\mathcal{C}\Bigr).
\end{align*}
Therefore, we obtain
\begin{align*}
    \sqrt{n}\Biggl(\,
    \begin{bmatrix}
    \hat{\bm{\sigma}}_n({\bm{\pi}}_M)\\
    {\bm{0}}_{M-\hat{G}_n({\bm{\pi}}_M)}
    \end{bmatrix}
    -\begin{bmatrix}
    {\bm{\sigma}}^*\\
    {\bm{0}}_{M-G^*}
    \end{bmatrix}\, \Biggr)\xlongrightarrow{d}\mathcal{Z}_2
\end{align*}
under ${\rm{P}}_{\bm{z}}$. This completes the proof.
\end{proof}
\noindent Finally, we show Theorem \ref{zconsis}. As in the proof of Theorem \ref{consis}, we first consider the case where the true number of clusters $G^*$ is known. For each $j=1,\ldots,J$, we set
\begin{align*}
    \hat{\bf{Q}}^{*}_{j,n}(g;{\bm{\pi}}^{(G^*)})={\bf{Q}}^{(G^*)}_{j,n}\bigl(g;\hat{\bm{\sigma}}^*_n({\bm{\pi}}^{(G^*)}),{\bm{\pi}}^{(G^*)}\bigr)
\end{align*}
for $g=1,\ldots,G^*$ and 
\begin{align*}
    \hat{Z}^{*}_{j,n}({\bm{\pi}}^{(G^*)})\in\argmax_{g\in\{1,\ldots,G^*\}}\hat{\bf{Q}}^{*}_{j,n}(g;{\bm{\pi}}^{(G^*)}).
\end{align*}
If the argmax contains more than one index, we select the smallest index. For notational simplicity, in the proofs of Lemma \ref{Qlemma} and Proposition \ref{Zstarconsis}, we may suppress the model index $G^*$. 
\begin{lemma}\label{Qlemma}
Suppose that ${[\bf{A1}]}$ and ${[\bf{A2}]}$ hold. Then, under ${\rm{P}}_{\bm{z}}$, for every ${\bm{\pi}}^{(G^*)}\in\Pi^{(G^*)}$, each $j=1,\ldots,J$, and each $g=1,\ldots,G^*$,
\begin{align*}
    \hat{\bf{Q}}^{*}_{j,n}(g;{\bm{\pi}}^{(G^*)})\xlongrightarrow{p}\mathbb{Q}_{z_j}(g)
\end{align*}
as $n\longrightarrow\infty$.
\end{lemma}
\begin{proof}
Fix ${\bm{\pi}}\in\Pi$ and $j\in\{1,\ldots,J\}$ arbitrarily. Under ${\rm{P}}_{\bm{z}}$, if
\begin{align}
    \hat{\bf{Q}}^*_{j,n}(g;{\bm{\pi}})\xlongrightarrow{p}\mathbb{Q}_{z_j}(g)=0\label{probQ}
\end{align}
for each $g\in\{1,\ldots,G^*\}\setminus
\{z_j\}$, then
\begin{align*}
    \Bigl|\hat{\bf{Q}}^*_{j,n}(z_j;{\bm{\pi}})-1\Bigr|
    &=\left|-\sum_{g\neq z_j}\frac{\pi_{g}\exp\Bigl\{\bar{\bf{H}}_{j,n}(\hat{\sigma}^*_{g,n})\Bigr\}}{\displaystyle\sum_{k=1}^{G^*}\pi_{k}\exp\Bigl\{\bar{\bf{H}}_{j,n}(\hat{\sigma}^*_{k,n})\Bigr\}}\right|\\
    &=\Biggl|\sum_{g\neq z_j}\hat{\bf{Q}}^*_{j,n}(g;{\bm{\pi}})\Biggr|\xlongrightarrow{p}0,
\end{align*}
which yields
\begin{align*}
    \hat{\bf{Q}}^*_{j,n}(z_j;{\bm{\pi}})\xlongrightarrow{p}\mathbb{Q}_{z_j}(z_j)=1,
\end{align*}
where $\hat{\sigma}^*_{g,n}$ is the $g$-th component of $\hat{\bm{\sigma}}_n^*({\bm{\pi}})$. Hence, it suffices to prove \eqref{probQ}. By Lemma \ref{Qsuplemma}, there exists $\rho>0$ such that for each $g\in\{1,\ldots,G^*\}\setminus
\{z_j\}$, 
\begin{align}
    \sup_{{\bm{\sigma}}\in B^*(\rho)}{\bf{Q}}_{j,n}(g;{\bm{\sigma}},{\bm{\pi}})\xlongrightarrow{p}0 \label{prob01}
\end{align}
under ${\rm{P}}_{\bm{z}}$. On
\begin{align*}
   \Bigl\{|\hat{\bm{\sigma}}^*_n({\bm{\pi}})-{\bm{\sigma}}^*|\leq\rho\Bigr\},
\end{align*}
$\hat{\bm{\sigma}}^*_n({\bm{\pi}})\in B^*(\rho)$, so that for each $g\in\{1,\ldots,G^*\}\setminus
\{z_j\}$, 
\begin{align*}
    \hat{{\bf{Q}}}^*_{j,n}(g;{\bm{\pi}})={\bf{Q}}_{j,n}\bigl(g;\hat{\bm{\sigma}}^*_n({\bm{\pi}}),{\bm{\pi}}\bigr)\leq \sup_{{\bm{\sigma}}\in B^*(\rho)}{\bf{Q}}_{j,n}(g;{\bm{\sigma}},{\bm{\pi}}).
\end{align*}
Therefore, by Proposition \ref{consisprop} and \eqref{prob01}, for all $\varepsilon>0$ and each $g\in\{1,\ldots,G^*\}\setminus
\{z_j\}$, we obtain
\begin{align*}
    &\quad\ {\rm{P}}_{\bm{z}}\biggl(\hat{\bf{Q}}^*_{j,n}(g;{\bm{\pi}})>\varepsilon\biggr)\\
    &={\rm{P}}_{\bm{z}}\biggl(\Bigl\{\hat{\bf{Q}}^*_{j,n}(g;{\bm{\pi}})>\varepsilon\Bigr\}\cap\Bigl\{|\hat{\bm{\sigma}}^*_n({\bm{\pi}})-{\bm{\sigma}}^*|\leq\rho\Bigr\}\biggr)\\
    &\qquad+{\rm{P}}_{\bm{z}}\biggl(\Bigl\{\hat{\bf{Q}}^*_{j,n}(g;{\bm{\pi}})>\varepsilon\Bigr\}\cap\Bigl\{|\hat{\bm{\sigma}}^*_n({\bm{\pi}})-{\bm{\sigma}}^*|>\rho\Bigr\}\biggr)\\
    &\leq{\rm{P}}_{\bm{z}}\biggl(\sup_{{\bm{\sigma}}\in B^*(\rho)}{\bf{Q}}_{j,n}(g;{\bm{\sigma}},{\bm{\pi}})>\varepsilon\biggr)+{\rm{P}}_{\bm{z}}\biggl(|\hat{\bm{\sigma}}^*_n({\bm{\pi}})-{\bm{\sigma}}^*|>\rho\biggr)\longrightarrow 0
\end{align*}
as $n\longrightarrow\infty$, which yields \eqref{probQ}. This completes the proof.
\end{proof}
\begin{proposition}\label{Zstarconsis}
Suppose that ${[\bf{A1}]}$ and ${[\bf{A2}]}$ hold. Then, for every ${\bm{\pi}}^{(G^*)}\in\Pi^{(G^*)}$ and each $j=1,\ldots,J$,
\begin{align*}
   {\rm{P}}_{\bm{z}} \biggl(\hat{Z}^*_{j,n}({\bm{\pi}}^{(G^*)})=z_j\biggr)\longrightarrow 1
\end{align*}
as $n\longrightarrow\infty$.
\end{proposition}
\begin{proof}
Fix ${\bm{\pi}}\in\Pi$ and $j\in\{1,\ldots,J\}$ arbitrarily. Recall that $\mathbb{Q}_{z_j}(z_j)=1$ and $\mathbb{Q}_{z_j}(g)=0$ for each $g\in\{1,\ldots,G^*\}\setminus\{z_j\}$. On 
\begin{align*}
    \bigcap_{g=1}^{G^*}\biggl\{\bigl|\hat{\bf{Q}}^*_{j,n}(g;{\bm{\pi}})-\mathbb{Q}_{z_j}(g)\bigr|<\frac{1}{2}\biggr\},
\end{align*}
for each $g\in\{1,\ldots,G^*\}\setminus
\{z_j\}$, we have
\begin{align*}
    \hat{\bf{Q}}^*_{j,n}(z_j;{\bm{\pi}})&>\mathbb{Q}_{z_j}(z_j)-\frac{1}{2}=\mathbb{Q}_{z_j}(g)+\frac{1}{2}>\hat{\bf{Q}}^*_{j,n}(g;{\bm{\pi}}),
\end{align*}
so that
\begin{align*}
    \argmax_{g\in\{1,\ldots,G^*\}}\hat{\bf{Q}}^*_{j,n}(g;{\bm{\pi}})=\bigl\{z_j\bigr\},
\end{align*}
which yields
\begin{align*}
    \hat{Z}^*_{j,n}({\bm{\pi}})=z_j.
\end{align*}
Lemma \ref{Qlemma} implies 
\begin{align*}
    {\rm{P}}_{\bm{z}}\Biggl(\bigcap_{g=1}^{G^*}\biggl\{\bigl|\hat{\bf{Q}}^*_{j,n}(g;{\bm{\pi}})-\mathbb{Q}_{z_j}(g)\bigr|<\frac{1}{2}\biggr\}\Biggr)\longrightarrow 1
\end{align*}
as $n\longrightarrow\infty$. Therefore, together with 
\begin{align*}
    \bigcap_{g=1}^{G^*}\biggl\{\bigl|\hat{\bf{Q}}^*_{j,n}(g;{\bm{\pi}})-\mathbb{Q}_{z_j}(g)\bigr|<\frac{1}{2}\biggr\}\subset\biggl\{\hat{Z}^*_{j,n}({\bm{\pi}})=z_j\biggr\},
\end{align*}
we have
\begin{align*}
    {\rm{P}}_{\bm{z}}\biggl(\hat{Z}^*_{j,n}({\bm{\pi}})=z_j\biggr)\longrightarrow 1
\end{align*}
as $n\longrightarrow\infty$.
\end{proof}
\begin{proof}[\textbf{Proof of Theorem \ref{zconsis}}]
Fix an arbitrary ${\bm{\pi}}_M\in\Pi_M$. For each $j=1,\ldots,J$, we have
\begin{align*}
    \hat{Z}_{j,n}({\bm{\pi}}_M)=\hat{Z}^*_{j,n}({\bm{\pi}}^{(G^*)})
\end{align*}
on $E_n({\bm{\pi}}_M)$. Hence, for each $j=1,\ldots,J$, it holds from Theorem \ref{Gconsis} and Proposition \ref{Zstarconsis} that 
\begin{align*}
    {\rm{P}}_{\bm{z}}\biggl(\hat{Z}_{j,n}({\bm{\pi}}_M)\neq z_j\biggr)
    &={\rm{P}}_{\bm{z}}\biggl(\Bigl\{\hat{Z}_{j,n}({\bm{\pi}}_M)\neq z_j\Bigr\}
    \cap E_n({\bm{\pi}}_M)\biggr)\\
    &\qquad+{\rm{P}}_{\bm{z}}\biggl(\Bigl\{\hat{Z}_{j,n}({\bm{\pi}}_M)\neq z_j\Bigr\}\cap E_n({\bm{\pi}}_M)^c\biggr)\\
    &\leq{\rm{P}}_{\bm{z}}\biggl(\hat{Z}^*_{j,n}({\bm{\pi}}^{(G^*)})\neq z_j\biggr)
    +{\rm{P}}_{\bm{z}}\Bigl(E_n({\bm{\pi}}_M)^c\Bigr)\longrightarrow 0
\end{align*}
as $n\longrightarrow\infty$, which yields
\begin{align*}
    {\rm{P}}_{\bm{z}}\biggl(\hat{Z}_{j,n}({\bm{\pi}}_M)=z_j\biggr)\longrightarrow 1
\end{align*}
as $n\longrightarrow\infty$. This completes the proof.
\end{proof}

\section{Appendix}
\subsection{On the quasi-likelihood}\label{likelihood}
The quasi-likelihood for ${\bm{\sigma}}^{(G)}$ in \eqref{quasi} is constructed 
by using the Euler--Maruyama approximation as follows. Fix an arbitrary $G\in\{1,\ldots,M\}$. For each $j=1,\ldots,J$, since
\begin{align*}
    \Omega=\bigcup_{g=1}^{G}\bigl\{Z^{(G)}_j=g\bigr\}
\end{align*}
and
\begin{align*}
    \bigl\{Z^{(G)}_j=g_1\bigr\}\cap\bigl\{Z^{(G)}_j=g_2\bigr\}=\emptyset
    \quad (g_1\neq g_2),
\end{align*}
for $B_j\in\mathcal{B}_{n+1}$, we have
\begin{align*}
    \mathbb{P}\Bigl(\bigl\{X^{(G)}_{j,t_i^n}\bigr\}_{i=0}^n\in B_j\Bigr)
    &=\mathbb{P}\biggl(\Bigl\{\bigl\{X^{(G)}_{j,t_i^n}\bigr\}_{i=0}^n\in B_j\Bigr\}\cap\Omega\biggr)\\
    &=\mathbb{P}\Biggl(
    \bigcup_{g=1}^{G}
    \biggl[\Bigl\{\bigl\{X^{(G)}_{j,t_i^n}\bigr\}_{i=0}^n\in B_j\Bigr\}\cap\Bigl\{Z^{(G)}_j=g\Bigr\}
    \biggr]\Biggr)\\
    &=\sum_{g=1}^{G}\mathbb{P}\biggl(\Bigl\{\bigl\{X^{(G)}_{j,t_i^n}\bigr\}_{i=0}^n\in B_j\Bigr\}\cap\Bigl\{Z^{(G)}_j=g\Bigr\}\biggr)\\
    &=\sum_{g=1}^{G}\mathbb{P}\Bigl(\bigl\{X^{(G)}_{j,t_i^n}\bigr\}_{i=0}^n\in B_j\big| Z^{(G)}_j=g\Bigr)
    \mathbb{P}\Bigl(Z^{(G)}_j=g\Bigr)\\
    &=\sum_{g=1}^{G}\pi^{(G)}_g\mathbb{P}\Bigl(\bigl\{X^{(G)}_{j,t_i^n}\bigr\}_{i=0}^n\in B_j\big| Z^{(G)}_j=g\Bigr),
\end{align*}
where $\mathcal{B}_{n+1}$ is the Borel $\sigma$-field on $\mathbb{R}^{n+1}$. Therefore, the independence of $\{X^{(G)}_{j,t}\}_{t\geq 0}\ (j=1,\ldots,J)$ implies that
\begin{align*}
    &\quad\ \mathbb{P}\Bigl(
    \bigl\{X^{(G)}_{1,t_i^n}\bigr\}_{i=0}^n\in B_1,
    \bigl\{X^{(G)}_{2,t_i^n}\bigr\}_{i=0}^n\in B_2,\ldots,
    \bigl\{X^{(G)}_{J,t_i^n}\bigr\}_{i=0}^n\in B_J\Bigr)\\
    &=\prod_{j=1}^J
    \mathbb{P}\Bigl(\bigl\{X^{(G)}_{j,t_i^n}\bigr\}_{i=0}^n\in B_j\Bigr)\\
    &=\prod_{j=1}^J\sum_{g=1}^{G}\pi^{(G)}_g\mathbb{P}\Bigl(\bigl\{X^{(G)}_{j,t_i^n}\bigr\}_{i=0}^n\in B_j\big| Z^{(G)}_j=g\Bigr).
\end{align*}
Moreover, for each $j=1,\ldots,J$ and $g=1,\ldots,G$, since the conditional law of $\{X^{(G)}_{j,t}\}_{t\geq 0}$ given $Z^{(G)}_j=g$ is specified as the law of the diffusion process satisfying
\begin{align*}
    dY^{(G)}_{g,t}=
    a^{(G)}_g(Y^{(G)}_{g,t})dt+\sigma^{(G)}_g dW_t,\quad Y^{(G)}_{g,0}=\xi^{(G)}_{g},
\end{align*}
we have
\begin{align*}
   \mathbb{P}\Bigl(\bigl\{X^{(G)}_{j,t_i^n}\bigr\}_{i=0}^n\in B_j\big| Z^{(G)}_j=g
   \Bigr)=\mathbb{P}\Bigl(
   \bigl\{Y^{(G)}_{g,t_i^n}\bigr\}_{i=0}^n\in B_j
   \Bigr).
\end{align*}
By the Euler--Maruyama approximation, the conditional distribution of
\begin{align*}
    Y^{(G)}_{g,t_i^n}-Y^{(G)}_{g,t_{i-1}^n}
    \big| Y^{(G)}_{g,t_{i-1}^n}=y^{(G)}_{g,t_{i-1}^n}
\end{align*}
can be approximated by
\begin{align*}
    N\Bigl(h_na^{(G)}_g\bigl(y^{(G)}_{g,t_{i-1}^n}\bigr),h_n\sigma_g^{(G)2}\Bigr).
\end{align*}
For inference on the diffusion coefficients, we omit the drift term and approximate 
\begin{align*}
    \mathbb{P}\Bigl(\bigl\{Y^{(G)}_{g,t_i^n}\bigr\}_{i=0}^n\in B_j\Bigr)
\end{align*}
by 
\begin{align*}
    \int_{B_j}\prod_{i=1}^n\frac{1}{\sqrt{2\pi h_{n}\sigma_g^{(G)2}}}
    \exp\Biggl\{-\frac{(y_i-y_{i-1})^2}
    {2h_{n}\sigma_g^{(G)2}}\Biggr\}\mu_g^{(G)}(dy_0)dy_1\cdots dy_n,
\end{align*}
where $\mu_g^{(G)}$ denotes the distribution of $\xi^{(G)}_{g}$.
Hence, the approximation of the joint probability
\begin{align*}
    \mathbb{P}\Bigl(
    \bigl\{X^{(G)}_{1,t_i^n}\bigr\}_{i=0}^n\in B_1,
    \bigl\{X^{(G)}_{2,t_i^n}\bigr\}_{i=0}^n\in B_2,\ldots,
    \bigl\{X^{(G)}_{J,t_i^n}\bigr\}_{i=0}^n\in B_J\Bigr)
\end{align*}
is given by
\begin{align*}
    \prod_{j=1}^J\sum_{g=1}^{G}\pi^{(G)}_g\int_{B_j}\prod_{i=1}^n
    \frac{1}{\sqrt{2\pi h_{n}\sigma_g^{(G)2}}}
    \exp\Biggl\{-\frac{(y_i-y_{i-1})^2}
    {2h_{n}\sigma_g^{(G)2}}\Biggr\}\mu_g^{(G)}(dy_0)dy_1\cdots dy_n.
\end{align*}
Using this approximation, we define the quasi-likelihood for ${\bm{\sigma}}^{(G)}$ as follows:
\begin{align*}
    \prod_{j=1}^J\sum_{g=1}^{G}\pi^{(G)}_g\prod_{i=1}^n\frac{1}{\sqrt{2\pi h_{n}\sigma_g^{(G)2}}}
    \exp\Biggl\{-\frac{(X_{j,t_i^n}-X_{j,t_{i-1}^n})^2}
    {2h_{n}\sigma_g^{(G)2}}\Biggr\}.
\end{align*}
\subsection{Proof of Lemma \ref{Sigmalemma}}\label{Lemma}
For a stochastic process $\{S_t\}_{t\geq0}$, we set
\begin{align*}
    \Delta_i^n S=S_{t_i^n}-S_{t_{i-1}^n}
\end{align*}
for $i=1,\ldots,n$. For notational simplicity, we write $a^{(G^*)}_{g}$ and $\sigma_g^{(G^*)}$
as $a_{g}$ and $\sigma_g$ respectively for $g=1,\ldots,G^*$. 
$C_L$ denotes a generic positive constant that depends on $L$ and whose value may vary depending on the context without being specifically mentioned. For a positive sequence $u_{n}$, $R:[0,\infty)^2\times\mathbb{R}\rightarrow \mathbb{R}$ denotes a generic function satisfying $|R({u_{n}},\sigma,x)|\leq u_{n}C(1+|x|)^C$ for some $C>0$ independent of $\sigma$, $n$, $i$ and $j$.
\begin{lemma}\label{ELlemma}
Suppose that ${[\bf{A1}]}$ and ${[\bf{A2}]}$ hold. Then, for every $L\geq 1$, each $i=1,\ldots,n$ and each $j=1,\ldots,J$, 
\begin{align*}
    {\rm{E}}_{\bm{z}}\Bigl[\bigl|X_{j,t}-X_{j,t_{i-1}^n}\bigr|^L\big|\mathcal{F}_{t_{i-1}^n}\Bigr]\leq C_L\Bigl(1+\bigl|X_{j,t_{i-1}^n}\bigr|^L\Bigr)\bigl|t-t_{i-1}^n\bigr|^{L/2}
\end{align*}
for every $t\in(t_{i-1}^n,t_i^n]$.
\end{lemma}
\begin{proof}
Fix arbitrary $L\geq 1$, $i\in\{1,\ldots,n\}$, $j\in\{1,\ldots,J\}$ and $t\in(t_{i-1}^n,t_i^n]$. First, we note that
\begin{align*}
    X_{j,t}-X_{j,t_{i-1}^n}&=\int_{t_{i-1}^n}^t a_{z_j}(X_{j,s})ds
    +\sigma^*_{z_j}\bigl(W_{j,t}-W_{j,t_{i-1}^n}\bigr)\\
    &=\int_{t_{i-1}^n}^t\bigl\{a_{z_j}(X_{j,s})-a_{z_j}(X_{j,t_{i-1}^n})\bigr\}ds\\
    &\qquad+a_{z_j}(X_{j,t_{i-1}^n})\bigl(t-t_{i-1}^n\bigr)+\sigma^*_{z_j}\bigl(W_{j,t}-W_{j,t_{i-1}^n}\bigr)
\end{align*}
under ${\rm{P}}_{\bm{z}}$, so that 
\begin{align*}
    &\quad\ {\rm{E}}_{\bm{z}}\Bigl[\bigl|X_{j,t}-X_{j,t_{i-1}^n}\bigr|^L\big|\mathcal{F}_{t_{i-1}^n}\Bigr]\\
    &\leq C_L{\rm{E}}_{\bm{z}}\Biggl[
    \biggl|\int_{t_{i-1}^n}^{t}\bigl|a_{z_j}(X_{j,s})-a_{z_j}(X_{j,t_{i-1}^n})\bigr|ds\biggr|^L
    \bigg|\mathcal{F}_{t_{i-1}^n}\Biggr]\\
    &\qquad+C_L\bigl|a_{z_j}(X_{j,t_{i-1}^n})\bigr|^L\bigl|t-t_{i-1}^n\bigr|^L+C_L{\rm{E}}_{\bm{z}}\Bigl[\bigl|W_{j,t}-W_{j,t_{i-1}^n}\bigr|^L\big|\mathcal{F}_{t_{i-1}^n}\Bigr].
\end{align*}
By H\"older's inequality, we have
\begin{align*}
    &\quad\ \int_{t_{i-1}^n}^{t}\bigl|a_{z_j}(X_{j,s})-a_{z_j}(X_{j,t_{i-1}^n})\bigr|ds\\
    &\leq\Biggl(\int_{t_{i-1}^n}^{t}\bigl|a_{z_j}(X_{j,s})-a_{z_j}(X_{j,t_{i-1}^n})\bigr|^Lds\Biggr)^{1/L}\Biggl(\int_{t_{i-1}^n}^{t}ds\Biggr)^{(L-1)/L}.
\end{align*}
By Tonelli's theorem and {\bf{[A1]}}, it holds that
\begin{align*}
    &\quad\ {\rm{E}}_{\bm{z}}\Biggl[
    \biggl|\int_{t_{i-1}^n}^{t}\bigl|a_{z_j}(X_{j,s})-a_{z_j}(X_{j,t_{i-1}^n})\bigr|ds\biggr|^L
    \bigg|\mathcal{F}_{t_{i-1}^n}\Biggr]\\
    &\leq\bigl|t-t_{i-1}^n\bigr|^{L-1}{\rm{E}}_{\bm{z}}\Biggl[\int_{t_{i-1}^n}^{t}\bigl|a_{z_j}(X_{j,s})-a_{z_j}(X_{j,t_{i-1}^n})\bigr|^Lds\bigg|\mathcal{F}_{t_{i-1}^n}\Biggr]\\
    &=\bigl|t-t_{i-1}^n\bigr|^{L-1}\int_{t_{i-1}^n}^{t}{\rm{E}}_{\bm{z}}\Bigl[\bigl|a_{z_j}(X_{j,s})-a_{z_j}(X_{j,t_{i-1}^n})\bigr|^L\big|\mathcal{F}_{t_{i-1}^n}\Bigr]ds\\
    &\leq C_L\bigl|t-t_{i-1}^n\bigr|^{L-1}\int_{t_{i-1}^n}^{t}{\rm{E}}_{\bm{z}}\Bigl[\bigl|X_{j,s}-X_{j,t_{i-1}^n}\bigr|^L\big|\mathcal{F}_{t_{i-1}^n}\Bigr]ds.
\end{align*}
Moreover, under ${\rm P}_{\bm z}$, $W_{j,t}-W_{j,t_{i-1}^n}$ is independent of $\mathcal F_{t_{i-1}^n}$ and follows the normal distribution with mean $0$ and variance $t-t_{i-1}^n$. Hence, we obtain 
\begin{align}
   {\rm{E}}_{\bm{z}}\Bigl[\bigl|W_{j,t}-W_{j,t_{i-1}^n}\bigr|^L\big|\mathcal{F}_{t_{i-1}^n}\Bigr]\leq C_L\bigl|t-t_{i-1}^n\bigr|^{L/2}.\label{WL}
\end{align}
Consequently, 
\begin{align*}
    &\quad\ {\rm{E}}_{\bm{z}}\Bigl[\bigl|X_{j,t}-X_{j,t_{i-1}^n}\bigr|^L\big|\mathcal{F}_{t_{i-1}^n}\Bigr]\\
    &\leq C_L\bigl|a_{z_j}(X_{j,t_{i-1}^n})\bigr|^L\bigl|t-t_{i-1}^n\bigr|^L+C_L\bigl|t-t_{i-1}^n\bigr|^{L/2}\\
    &\qquad+C_L\bigl|t-t_{i-1}^n\bigr|^{L-1}\int_{t_{i-1}^n}^{t}{\rm{E}}_{\bm{z}}\Bigl[\bigl|X_{j,s}-X_{j,t_{i-1}^n}\bigr|^L\big|\mathcal{F}_{t_{i-1}^n}\Bigr]ds\\
    &\leq C_L\Bigl(\bigl|a_{z_j}(X_{j,t_{i-1}^n})\bigr|^L+1\Bigr)\bigl|t-t_{i-1}^n\bigr|^{L/2}\\
    &\qquad+C_L\int_{t_{i-1}^n}^{t}{\rm{E}}_{\bm{z}}\Bigl[\bigl|X_{j,s}-X_{j,t_{i-1}^n}\bigr|^L\big|\mathcal{F}_{t_{i-1}^n}\Bigr]ds.
\end{align*}
By Gronwall's inequality and {\bf [A1]}, we obtain
\begin{align*}
    &\quad\ {\rm{E}}_{\bm{z}}\Bigl[\bigl|X_{j,t}-X_{j,t_{i-1}^n}\bigr|^L\big|\mathcal{F}_{t_{i-1}^n}\Bigr]\\
    &\leq C_L\Bigl(1+\bigl|X_{j,t_{i-1}^n}\bigr|^L\Bigr)\bigl|t-t_{i-1}^n\bigr|^{L/2}\exp\Biggl(C_L\int_{t_{i-1}^n}^{t}ds\Biggr)\\
    &\leq C_L\Bigl(1+\bigl|X_{j,t_{i-1}^n}\bigr|^L\Bigr)\bigl|t-t_{i-1}^n\bigr|^{L/2},
\end{align*}
which completes the proof. 
\end{proof}
\begin{lemma}\label{EX24lemma}
Suppose that ${[\bf{A1}]}$ and ${[\bf{A2}]}$ hold. Then, for each $i=1,\ldots,n$ and $j=1,\ldots,J$,
\begin{align}
    {\rm{E}}_{\bm{z}}\Bigl[(X_{j,t_i^n}-X_{j,t_{i-1}^n})^2\big|\mathcal{F}_{t_{i-1}^n}\Bigr]=h_n\sigma^{*2}_{z_j}+R\bigl(h^2_n,\sigma_{z_j},X_{j,t_{i-1}^n}\bigr) \label{EX2}
\end{align}
and
\begin{align}
    {\rm{E}}_{\bm{z}}\Bigl[(X_{j,t_i^n}-X_{j,t_{i-1}^n})^4\big|\mathcal{F}_{t_{i-1}^n}\Bigr]=3h^2_n\sigma^{*4}_{z_j}+R\bigl(h^3_n,\sigma_{z_j},X_{j,t_{i-1}^n}\bigr). \label{EX4}
\end{align}
\end{lemma}
\begin{proof}
Fix arbitrary $i\in\{1,\ldots,n\}$ and $j\in\{1,\ldots,J\}$. Set
\begin{align*}
    U_{i,j}=\int_{t_{i-1}^n}^{t_i^n}\bigl\{a_{z_j}(X_{j,t})-a_{z_j}(X_{j,t_{i-1}^n})\bigr\}dt.
\end{align*}
Under ${\rm{P}}_{\bm{z}}$, we note that
\begin{align*}
    X_{j,t_i^n}-X_{j,t_{i-1}^n}=\sigma_{z_j}^*\Delta_i^n W_{j}+h_na_{z_j}(X_{j,t_{i-1}^n})+U_{i,j}.
\end{align*}
In an analogous manner to the proof of Lemma \ref{ELlemma}, for each $L\geq 1$, it is shown that
\begin{align*}
    {\rm{E}}_{\bm{z}}\Bigl[|U_{i,j}|^L\big|\mathcal{F}_{t_{i-1}^n}\Bigr]&\leq C_Lh_n^{L-1}\int_{t_{i-1}^n}^{t_i^n}{\rm{E}}_{\bm{z}}\Bigl[\bigl|X_{j,t}-X_{j,t_{i-1}^n}\bigr|^L\big|\mathcal{F}_{t_{i-1}^n}\Bigr]dt
\end{align*}
and Lemma \ref{ELlemma} gives that
\begin{align*}
    \int_{t_{i-1}^n}^{t_i^n}{\rm{E}}_{\bm{z}}\Bigl[\bigl|X_{j,t}-X_{j,t_{i-1}^n}\bigr|^L\big|\mathcal{F}_{t_{i-1}^n}\Bigr]dt\leq h_n\times C_Lh_n^{L/2}\Bigl(1+\bigl|X_{j,t_{i-1}^n}\bigr|^L\Bigr),
\end{align*}
which yields
\begin{align}
    {\rm{E}}_{\bm{z}}\Bigl[|U_{i,j}|^L\big|\mathcal{F}_{t_{i-1}^n}\Bigr]=R\bigl(h^{3L/2}_n,\sigma_{z_j},X_{j,t_{i-1}^n}\bigr). \label{UL}
\end{align}
Furthermore, for each $M,N\in\mathbb{N}$, by the Cauchy--Schwarz inequality, 
\begin{align*}
    {\rm{E}}_{\bm{z}}\Bigl[|\Delta_i^n W_{j}|^M|U_{i,j}|^N\big|\mathcal{F}_{t_{i-1}^n}\Bigr]&\leq {\rm{E}}_{\bm{z}}\Bigl[|\Delta_i^n W_{j}|^{2M}\big|\mathcal{F}_{t_{i-1}^n}\Bigr]^{1/2}{\rm{E}}_{\bm{z}}\Bigl[|U_{i,j}|^{2N}\big|\mathcal{F}_{t_{i-1}^n}\Bigr]^{1/2},
\end{align*}
so that \eqref{WL} and \eqref{UL} imply
\begin{align}
    {\rm{E}}_{\bm{z}}\Bigl[|\Delta_i^n W_{j}|^M|U_{i,j}|^N\big|\mathcal{F}_{t_{i-1}^n}\Bigr]=R\bigl(h^{(M+3N)/2}_n,\sigma_{z_j},X_{j,t_{i-1}^n}\bigr).
    \label{WU}
\end{align}
Hence, by noting that
\begin{align}
     {\rm{E}}_{\bm{z}}\Bigl[\Delta_i^n W_{j}\big|\mathcal{F}_{t_{i-1}^n}\Bigr]=0,\quad
    {\rm{E}}_{\bm{z}}\Bigl[(\Delta_i^n W_{j})^2\big|\mathcal{F}_{t_{i-1}^n}\Bigr]=h_n, \label{W2}
\end{align}
it follows from \eqref{UL} and \eqref{WU} that
\begin{align*}
    &\quad\ {\rm{E}}_{\bm{z}}\Bigl[(X_{j,t_i^n}-X_{j,t_{i-1}^n})^2\big|\mathcal{F}_{t_{i-1}^n}\Bigr]\\
    &=\sigma_{z_j}^{*2}{\rm{E}}_{\bm{z}}\Bigl[(\Delta_i^n W_{j})^2\big|\mathcal{F}_{t_{i-1}^n}\Bigr]+h^2_na_{z_j}(X_{j,t_{i-1}^n})^2+{\rm{E}}_{\bm{z}}\Bigl[(U_{i,j})^2\big|\mathcal{F}_{t_{i-1}^n}\Bigr]\\
    &\quad+2h_na_{z_j}(X_{j,t_{i-1}^n})\sigma_{z_j}^{*}{\rm{E}}_{\bm{z}}\Bigl[\Delta_i^n W_{j}\big|\mathcal{F}_{t_{i-1}^n}\Bigr]+2\sigma_{z_j}^{*}{\rm{E}}_{\bm{z}}\Bigl[(\Delta_i^n W_{j})(U_{i,j})\big|\mathcal{F}_{t_{i-1}^n}\Bigr]\\
    &\quad+2h_na_{z_j}(X_{j,t_{i-1}^n}){\rm{E}}_{\bm{z}}\Bigl[U_{i,j}\big|\mathcal{F}_{t_{i-1}^n}\Bigr]
    =h_n\sigma^{*2}_{z_j}+R(h^{2}_n,\sigma_{z_j},X_{j,t_{i-1}^n}),
\end{align*}
which yields \eqref{EX2}. Moreover, since
\begin{align*}
    {\rm{E}}_{\bm{z}}\Bigl[(\Delta_i^n W_{j})^3\big|\mathcal{F}_{t_{i-1}^n}\Bigr]=0,\quad
    {\rm{E}}_{\bm{z}}\Bigl[(\Delta_i^n W_{j})^4\big|\mathcal{F}_{t_{i-1}^n}\Bigr]=3h_n^2,
\end{align*}
it holds from \eqref{UL}--\eqref{W2} that
\begin{align*}
    &\quad\ {\rm{E}}_{\bm{z}}\Bigl[(X_{j,t_i^n}-X_{j,t_{i-1}^n})^4\big|\mathcal{F}_{t_{i-1}^n}\Bigr]\\
    &=\sigma_{z_j}^{*4}{\rm{E}}_{\bm{z}}\Bigl[(\Delta_i^n W_{j})^4\big|\mathcal{F}_{t_{i-1}^n}\Bigr]+h^4_na_{z_j}(X_{j,t_{i-1}^n})^4+{\rm{E}}_{\bm{z}}\Bigl[(U_{i,j})^4\big|\mathcal{F}_{t_{i-1}^n}\Bigr]\\
    &\quad+4h_na_{z_j}(X_{j,t_{i-1}^n})\sigma_{z_j}^{*3}{\rm{E}}_{\bm{z}}\Bigl[(\Delta_i^n W_{j})^3\big|\mathcal{F}_{t_{i-1}^n}\Bigr]+4\sigma_{z_j}^{*3}{\rm{E}}_{\bm{z}}\Bigl[(\Delta_i^n W_{j})^3(U_{i,j})\big|\mathcal{F}_{t_{i-1}^n}\Bigr]\\
    &\quad+4h^3_na_{z_j}(X_{j,t_{i-1}^n})^3\sigma_{z_j}^{*}{\rm{E}}_{\bm{z}}\Bigl[
    \Delta_i^n W_{j}\big|\mathcal{F}_{t_{i-1}^n}\Bigr]+4\sigma_{z_j}^{*}{\rm{E}}_{\bm{z}}\Bigl[(\Delta_i^n W_{j})(U_{i,j})^3\big|\mathcal{F}_{t_{i-1}^n}\Bigr]\\
    &\quad+4h^3_na_{z_j}(X_{j,t_{i-1}^n})^3{\rm{E}}_{\bm{z}}\Bigl[U_{i,j}\big|\mathcal{F}_{t_{i-1}^n}\Bigr]+4h_na_{z_j}(X_{j,t_{i-1}^n}){\rm{E}}_{\bm{z}}\Bigl[(U_{i,j})^3\big|\mathcal{F}_{t_{i-1}^n}\Bigr]\\
    &\quad+6h^2_na_{z_j}(X_{j,t_{i-1}^n})^2\sigma_{z_j}^{*2}{\rm{E}}_{\bm{z}}\Bigl[(\Delta_i^n W_{j})^2\big|\mathcal{F}_{t_{i-1}^n}\Bigr]+6\sigma_{z_j}^{*2}
    {\rm{E}}_{\bm{z}}\Bigl[(\Delta_i^n W_{j})^2(U_{i,j})^2\big|\mathcal{F}_{t_{i-1}^n}\Bigr]\\
    &\quad+6h^2_na_{z_j}(X_{j,t_{i-1}^n})^2{\rm{E}}_{\bm{z}}\Bigl[(U_{i,j})^2\big|\mathcal{F}_{t_{i-1}^n}\Bigr]\\
    &\quad+12h_na_{z_j}(X_{j,t_{i-1}^n})\sigma_{z_j}^{*2}{\rm{E}}_{\bm{z}}\Bigl[(\Delta_i^n W_{j})^2U_{i,j}\big|\mathcal{F}_{t_{i-1}^n}\Bigr]\\
    &\quad+12h^2_na_{z_j}(X_{j,t_{i-1}^n})^2\sigma_{z_j}^{*}{\rm{E}}_{\bm{z}}\Bigl[(\Delta_i^n W_{j})(U_{i,j})\big|\mathcal{F}_{t_{i-1}^n}\Bigr]\\
    &\quad+12h_na_{z_j}(X_{j,t_{i-1}^n})\sigma_{z_j}^{*}{\rm{E}}_{\bm{z}}\Bigl[(\Delta_i^n W_{j})(U_{i,j})^2\big|\mathcal{F}_{t_{i-1}^n}\Bigr]=3h^2_n\sigma^{*4}_{z_j}+R(h^{3}_n,\sigma_{z_j},X_{j,t_{i-1}^n}),
\end{align*}
so that \eqref{EX4} holds. This completes the proof.
\end{proof}
\begin{proof}[Proof of Lemma \ref{Sigmalemma}]
Fix an arbitrary $j\in\{1,\ldots,J\}$. We use the following decomposition:
\begin{align*}
    \hat{\Sigma}^2_{j,n}-\sigma_{z_j}^{*2}=\frac{1}{nh_n}
    \sum_{i=1}^n\Bigl\{(\Delta_i^nX_j)^2-h_n\sigma_{z_j}^{*2}\Bigr\}=M_{j,n}+R_{j,n},
\end{align*}
where
\begin{align*}
    M_{j,n}=\frac{1}{nh_n}
    \sum_{i=1}^n\biggl\{
    (\Delta_i^nX_j)^2-{\rm E}_{\bm{z}}\Bigl[(\Delta_i^nX_j)^2\big|\mathcal{F}_{t_{i-1}^n}\Bigr]\biggr\}
\end{align*}
and
\begin{align*}
    R_{j,n}=\frac{1}{nh_n}\sum_{i=1}^n\biggl\{{\rm E}_{\bm{z}}\Bigl[(\Delta_i^nX_j)^2\big|\mathcal{F}_{t_{i-1}^n}\Bigr]-h_n\sigma_{z_j}^{*2}\biggr\}.
\end{align*}
First, we show \eqref{Sigmacons}. By Lemma \ref{EX24lemma}, it is shown that
\begin{align*}
    &\quad\ {\rm E}_{\bm{z}}\biggl[\Bigl\{(\Delta_i^nX_j)^2-{\rm E}_{\bm{z}}\Bigl[(\Delta_i^nX_j)^2\big|\mathcal{F}_{t_{i-1}^n}\Bigr]\Bigr\}^2\Big|\mathcal{F}_{t_{i-1}^n}\biggr]\\
    &={\rm E}_{\bm{z}}\Bigl[(\Delta_i^nX_j)^4\big|\mathcal{F}_{t_{i-1}^n}\Bigr]-{\rm E}_{\bm{z}}\Bigl[(\Delta_i^nX_j)^2\big|\mathcal{F}_{t_{i-1}^n}\Bigr]^2=R(h_n^2,\sigma_{z_j},X_{j,t_{i-1}^n}).
\end{align*}
Therefore, noting that $M_{j,n}$ is a martingale sum and
\begin{align*}
    \sup_{t\in[0,T]}{\rm{E}}_{{\bm{z}}}
    \Bigl[\bigl|X_{j,t}\bigr|^{L}\Bigr]<\infty
\end{align*}
for every $L>0$, we have
\begin{align*}
    {\rm E}_{\bm{z}}\Bigl[M_{j,n}^2\Bigr]
    &=\frac{1}{n^2h_n^2}\sum_{i=1}^n{\rm E}_{\bm{z}}\biggl[\Bigl\{(\Delta_i^nX_j)^2-{\rm E}_{\bm{z}}\Bigl[(\Delta_i^nX_j)^2\big|\mathcal{F}_{t_{i-1}^n}\Bigr]\Bigr\}^2\biggr]\\
    &\leq\frac{C}{n^2}
    \sum_{i=1}^n\biggl\{1+{\rm E}_{\bm{z}}\Bigl[\bigl|X_{j,t_{i-1}^n}\bigr|^C\Bigr]\biggr\}\leq\frac{C}{n},
\end{align*}
which implies
\begin{align}
    M_{j,n}\xlongrightarrow{p}0\label{Mjnp}
\end{align}
under ${\rm{P}}_{\bm{z}}$. Since it follows from Lemma \ref{EX24lemma} that
\begin{align*}
    R_{j,n}&=\frac{h_n}{n}\sum_{i=1}^n R(1,\sigma_{z_j},X_{j,t_{i-1}^n}),
\end{align*}
under ${\rm{P}}_{\bm{z}}$, we have 
\begin{align}
    R_{j,n}\xlongrightarrow{p}0.
    \label{Rjnp}
\end{align}
Therefore, together with \eqref{Mjnp} and \eqref{Rjnp}, by Slutsky’s theorem, we obtain 
\begin{align*}
    \hat{\Sigma}^2_{j,n}\xlongrightarrow{p}\sigma_{z_j}^{*2},
\end{align*}
which completes the proof of \eqref{Sigmacons}. Next, we prove \eqref{Sigmaasym}. Note that
\begin{align*}
    \sqrt{n}\Bigl(\hat{\Sigma}^2_{j,n}-\sigma_{z_j}^{*2}\Bigr)
    =\sum_{i=1}^n \eta_{i,j,n}+\sqrt{n}R_{j,n},
\end{align*}
where
\begin{align*}
    \eta_{i,j,n}
    =\frac{1}{\sqrt{n}h_n}\biggl\{(\Delta_i^nX_j)^2
    -{\rm E}_{\bm{z}}\Bigl[(\Delta_i^n X_j)^2\big|\mathcal{F}_{t_{i-1}^n}\Bigr]\biggr\}.
\end{align*}
As it holds from Lemma \ref{EX24lemma} that
\begin{align*}
    \sum_{i=1}^n{\rm E}_{\bm{z}}\Bigl[\eta_{i,j,n}^2\big|\mathcal{F}_{t_{i-1}^n}\Bigr]
    &=\frac{1}{nh_n^2}\sum_{i=1}^n\biggl\{
    {\rm E}_{\bm{z}}\Bigl[(\Delta_i^nX_j)^4\big|\mathcal{F}_{t_{i-1}^n}\Bigr]-{\rm E}_{\bm{z}}\Bigl[(\Delta_i^nX_j)^2\big|\mathcal{F}_{t_{i-1}^n}\Bigr]^2\biggr\}\\
    &=\frac{1}{nh_n^2}\sum_{i=1}^n\biggl\{2h^2_n\sigma^{*4}_{z_j}+R\bigl(h^{3}_n,\sigma_{z_j},X_{j,t_{i-1}^n}\bigr)\biggr\}\\
    &=2\sigma^{*4}_{z_j}+\frac{h_n}{n}\sum_{i=1}^n R\bigl(1,\sigma_{z_j},X_{j,t_{i-1}^n}\bigr),
\end{align*}
we have 
\begin{align}
    \sum_{i=1}^n{\rm E}_{\bm{z}}\Bigl[\eta_{i,j,n}^2\big|\mathcal{F}_{t_{i-1}^n}\Bigr]\xlongrightarrow{p}2\sigma^{*4}_{z_j} \label{var}
\end{align}
under ${\rm{P}}_{\bm{z}}$. Moreover, it follows from Lemmas \ref{ELlemma} and \ref{EX24lemma} that
\begin{align*}
    {\rm E}_{\bm{z}}\Bigl[\eta_{i,j,n}^4\big|\mathcal{F}_{t_{i-1}^n}\Bigr]&\leq\frac{C}{n^2h_n^4}{\rm E}_{\bm{z}}
    \Bigl[(\Delta_i^nX_j)^8\big|\mathcal{F}_{t_{i-1}^n}\Bigr]+\frac{C}{n^2h_n^4}{\rm E}_{\bm{z}}\Bigl[(\Delta_i^nX_j)^2\big|\mathcal{F}_{t_{i-1}^n}\Bigr]^4\\
    &\leq\frac{1}{n^2}R\bigl(1,\sigma_{z_j},X_{j,t_{i-1}^n}\bigr),
\end{align*}
so that 
\begin{align*}
    \sum_{i=1}^n{\rm E}_{\bm{z}}\Bigl[\eta_{i,j,n}^2{\bf 1}_{\{|\eta_{i,j,n}|>\varepsilon\}}\big|\mathcal{F}_{t_{i-1}^n}\Bigr]
    &\leq\frac{1}{\varepsilon^2}
    \sum_{i=1}^n{\rm E}_{\bm{z}}\Bigl[\eta_{i,j,n}^4\big|\mathcal{F}_{t_{i-1}^n}\Bigr]\\
    &\leq\frac{C}{\varepsilon^2}\frac{1}{n^2}
    \sum_{i=1}^n R\bigl(1,\sigma_{z_j},X_{j,t_{i-1}^n}\bigr)
\end{align*}
for every $\varepsilon>0$. Hence, for every $\varepsilon>0$, we obtain
\begin{align}
    \sum_{i=1}^n{\rm E}_{\bm{z}}\Bigl[\eta_{i,j,n}^2{\bf 1}_{\{|\eta_{i,j,n}|>\varepsilon\}}\big|\mathcal{F}_{t_{i-1}^n}\Bigr]\xlongrightarrow{p}0 \label{lindb}
\end{align}
under ${\rm{P}}_{\bm{z}}$. Consequently, by \eqref{var} and \eqref{lindb}, the martingale central limit theorem (Theorems 3.2 and 3.4 in Hall and Heyde \cite{Hall(1980)}) yields 
\begin{align*}
    \sum_{i=1}^n \eta_{i,j,n}\xlongrightarrow{d}N\Bigl(0,2\sigma_{z_j}^{*4}\Bigr)
\end{align*}
under ${\rm{P}}_{\bm{z}}$. Furthermore, it holds from Lemma \ref{EX24lemma} that
\begin{align*}
    \sqrt{n}R_{j,n}&=\frac{1}{\sqrt{n}h_n}\sum_{i=1}^n\biggl\{{\rm E}_{\bm{z}}\Bigl[(\Delta_i^nX_j)^2\big|\mathcal{F}_{t_{i-1}^n}\Bigr]-h_n\sigma_{z_j}^{*2}\biggr\}\\
    &=\sqrt{nh_n^2}\times\frac{1}{n}\sum_{i=1}^n R(1,\sigma_{z_j},X_{j,t_{i-1}^n})\xlongrightarrow{p}0
\end{align*}
under ${\rm{P}}_{\bm{z}}$. Therefore, by Slutsky's theorem, we get
\begin{align*}
    \sqrt{n}\Bigl(\hat{\Sigma}^2_{j,n}-\sigma_{z_j}^{*2}\Bigr)\xlongrightarrow{d}N\Bigl(0,2\sigma_{z_j}^{*4}\Bigr)
\end{align*}
under ${\rm{P}}_{\bm{z}}$. This completes the proof.
\end{proof}

\begin{thebibliography}{10}
\fontsize{10pt}{16pt}\selectfont
\bibitem{Aghabozorgi(2015)} Aghabozorgi, S., Shirkhorshidi, A. S. and Wah, T. Y. (2015). Time-series clustering-a decade review. \textit{Information Systems}, {\bf{53}}, 16-38.
\bibitem{Boyko(2022)} Boyko, V., Krumscheid, S. and Vercauteren, N. (2022). Statistical learning of nonlinear stochastic differential equations from nonstationary time series using variational clustering. \textit{Multiscale Modeling \& Simulation}, {\bf{20}}(4), 1251-1283.
\bibitem{Cuesta-Albertos(2023)} Cuesta-Albertos, J. A. and Dutta, S. (2023). On perfect clustering for Gaussian processes. \textit{Transactions on Machine Learning Research}.
\bibitem{Delattre(2016)} Delattre, M., Genon-Catalot, V. and Samson, A. (2016). Mixtures of stochastic differential equations with random effects: application to data clustering. \textit{Journal of Statistical Planning and Inference}, {\bf{173}}, 109-124.
\bibitem{De Gregorio(2010)} De Gregorio, A. and Iacus, S. M. (2010). Clustering of discretely observed diffusion processes. \textit{Computational Statistics \& Data Analysis}, {\bf{54}}(2), 598-606.
\bibitem{Eguchi(2018)} Eguchi, S. and Masuda, H. (2018). Schwarz type model comparison for LAQ models. \textit{Bernoulli}, {\bf{24}(3)}, 2278-2327.
\bibitem{Genon(1993)}  Genon-Catalot, V. and Jacod, J. (1993). On the estimation of the diffusion coefficient for multidimensional diffusion processes. \textit{Annales de l'Institut Henri Poincar{\'e} (B) Probabilit{\'e}s et Statistiques}, \textbf{29}, 119-151. 
\bibitem{Hall(1980)} Hall, P. and Heyde, C. C. (1980). \textit{Martingale limit theory and its application.} \textit{Academic Press}.
\bibitem{Kalpakis(2001)} Kalpakis, K., Gada, D. and Puttagunta, V. (2001). Distance measures for effective clustering of ARIMA time-series. \textit{In Proceedings 2001 IEEE International Conference on Data Mining (pp. 273-280)}. IEEE.
\bibitem{kessler(1997)} Kessler, M. (1997). Estimation of an ergodic diffusion from discrete observations. \textit{Scandinavian Journal of Statistics}, \textbf{24(2)}, 211-229.
\bibitem{Kutoyants(2004)} Kutoyants, Y. A. (2004). Statistical inference for ergodic diffusion processes. \textit{Springer Science \& Business Media}.
\bibitem{Li(1999)} Li, C. and Biswas, G. (1999). Temporal pattern generation using hidden Markov model based unsupervised classification. \textit{In International Symposium on Intelligent Data Analysis (pp. 245-256)}. Berlin, Heidelberg: Springer Berlin Heidelberg.
\bibitem{Liao(2005)} Liao, T. W. (2005). Clustering of time series data—a survey. \textit{Pattern Recognition}, {\bf{38}}(11), 1857-1874.
\bibitem{Nguyen(2024)} Nguyen, H. D. (2024). PanIC: Consistent information criteria for general model selection problems. \textit{Australian \& New Zealand Journal of Statistics}, \textbf{66}(4), 441-466.
\bibitem{Nguyen(2026)} Nguyen, H. D. and Nguyen, T. (2026). Modifications of the BIC for order selection in finite mixture models. \textit{Annals of the Institute of Statistical Mathematics}. 
https://doi.org/10.1007/s10463-026-00999-4
\bibitem{Panuccio(2002)} Panuccio, A., Bicego, M. and Murino, V. (2002). A Hidden Markov Model-based approach to sequential data clustering. \textit{In Joint IAPR international workshops on statistical techniques in pattern recognition (SPR) and structural and syntactic pattern recognition (SSPR) (pp. 734-743)}. Berlin, Heidelberg: Springer Berlin Heidelberg.
\bibitem{Piccolo(1990)} Piccolo, D. (1990). A distance measure for classifying ARIMA models. \textit{Journal of Time Series Analysis}, \textbf{11}(2), 153-164.
\bibitem{Ramoni(2000)} Ramoni, M., Sebastiani, P. and Cohen, P. (2000). Multivariate clustering by dynamics. \textit{In AAAI/IAAI (pp. 633-638)}.
\bibitem{Ren(2022)} Ren, B. and Barnett, I. (2022). Autoregressive mixture models for clustering time series. \textit{Journal of Time Series Analysis}, {\bf{43}}(6), 918-937.
\bibitem{Uchi-Yoshi(2012)} 
Uchida, M. and Yoshida, N. (2012). 
Adaptive estimation of an ergodic diffusion process based on sampled data. 
\textit{Stochastic Processes and their Applications}, 
\textbf{122(8)}, 2885-2924. 
\bibitem{Xiong(2002)} Xiong, Y. and Yeung, D. Y. (2002). Mixtures of ARMA models for model-based time series clustering. \textit{In 2002 IEEE International Conference on Data Mining, 2002. Proceedings. (pp. 717-720)}. IEEE.
\bibitem{Xiong(2004)} Xiong, Y. and Yeung, D. Y. (2004). Time series clustering with ARMA mixtures. \textit{Pattern Recognition}, {\bf{37}}(8), 1675-1689.
\bibitem{Yoshida(1992)} Yoshida, N. (1992). Estimation for diffusion processes from discrete observation. \textit{Journal of Multivariate Analysis}, {\bf 41}, 220--242. 
\end{thebibliography}
\end{document}